%% file: main.tex
\documentclass[11pt]{article}
\usepackage{latexsym, amscd, amsfonts, mathrsfs, amsmath, amssymb, amsthm, enumitem, stmaryrd, tikz-cd, mathrsfs, bbm, url, esint, todonotes, listings, moreverb}

\usepackage[T1]{fontenc} 

\usepackage{algpseudocode}
\usepackage{algorithm}
\usepackage{subfig}

\usepackage{hyperref}
\hypersetup{colorlinks=true,citecolor=red,linkcolor=blue}

\def\bN{\mathbb{N}}

\def\bR{\mathbb{R}}
\def\bS{\mathbb{S}}

\newcommand{\R}{\mathbb{R}}
\newcommand{\tr}{\mathrm{tr}}

\newcommand{\new}{\mathrm{new}}

\newcommand{\argmin}{\mathrm{argmax}}

\newcommand{\poly}{\mathrm{poly}}
\newcommand{\ov}{\overline}
\newcommand{\wt}{\widetilde}
\newcommand{\wh}{\widehat}

\renewcommand{\hat}{\widehat}
\newcommand{\Tmat}{{\cal T}_{\mathrm{mat}}}
\newcommand{\triangleup}{\vartriangle}

\newcommand{\sdp}{\mathrm{sdp}}
\newcommand{\lp}{\mathrm{lp}}
\renewcommand{\new}{\mathrm{new}}
\newcommand{\apx}{\mathrm{apx}}
\renewcommand{\vec}{\mathrm{vec}}
\newcommand{\nnz}{\mathrm{nnz}}
\newcommand{\op}{\mathrm{op}}

\def\cB{\mathcal{B}}

\def\cD{\mathcal{D}}
\def\cI{\mathcal{I}}
\def\cK{\mathcal{K}}
\def\cN{\mathcal{N}}

\def\cP{\mathcal{P}}
\def\cS{\mathcal{S}}
\def\cT{\mathcal{T}}
\def\cW{\mathcal{W}}

\DeclareMathOperator\supp{\mathrm{supp}}
\DeclareMathOperator\tw{\mathrm{tw}}

\DeclareMathOperator\polylog{\mathrm{polylog}}

\newcommand{\rom}[1]{\textup{\uppercase\expandafter{\romannumeral#1}}}

\newtheorem{theorem}{Theorem}[section]
\newtheorem{lemma}[theorem]{Lemma}
\newtheorem{definition}[theorem]{Definition}

\newtheorem{corollary}[theorem]{Corollary}

\newtheorem{assumption}[theorem]{Assumption}

\newtheorem{remark}[theorem]{Remark}

\usepackage[margin=1in]{geometry}

\graphicspath{{./figs/}}

\definecolor{b2}{RGB}{51,153,255}
\definecolor{mygreen}{RGB}{80,180,0}

\allowdisplaybreaks

\begin{document}
\title{A Faster Small Treewidth SDP Solver}
\author{
Yuzhou Gu\thanks{\texttt{yuzhougu@mit.edu}. MIT.}
\and 
Zhao Song\thanks{\texttt{zsong@adobe.com}. Adobe Research.}
}
\date{}

\begin{titlepage}
\maketitle
\input{abstract}
  \thispagestyle{empty}
\end{titlepage}

\newpage

\input{intro}
\input{tech}

\section*{Acknowledgements}
The authors would like to thank Guanghao Ye and Lichen Zhang for useful discussions.

\newpage

\input{preli}

\newpage
\clearpage
\input{framework}

\newpage
\clearpage
\input{sdp_first}
\newpage
\clearpage
\input{sdp_second}

\newpage
\clearpage
\input{sdp_general}

\newpage
\clearpage
\input{lp}

\newpage

\bibliographystyle{alpha}
\bibliography{ref}

\newpage
\appendix 
\section*{Appendix}

\input{robust_ipm}

\end{document}

%% file: abstract.tex
Semidefinite programming is a fundamental tool in optimization and theoretical computer science. It has been extensively used as a black-box for solving many problems, such as embedding, complexity, learning, and discrepancy.

One natural setting of semidefinite programming is the small treewidth setting. 
The best previous SDP solver under small treewidth setting is due to \cite{zl18}, which takes $n^{1.5} \cdot \tau^{6.5}$ time.
In this work, we show how to solve a semidefinite programming with $n \times n$ variables, $m$ constraints and $\tau$ treewidth in $n \cdot \tau^{2\omega+1/2}$ time, where $\omega$ denotes the exponent of matrix multiplication. We give the first SDP solver that runs in time in linear in number of variables under this setting.

In addition, we improve the running time that solves a linear programming with $\tau$ treewidth from $n\tau^2$ \cite{y20,dly21} to $n \tau^{(\omega+1)/2}$.

%% file: intro.tex
\section{Introduction}

Semidefine programming is one fundamental problem in optimization and theoretical computer science. It has many applications in computer science, such as verifying the robustness of neural network \cite{rsl18}, solving the discrepancy problems \cite{b10}, sparse matrix factorization \cite{cstz22}, sparsest cut \cite{arv09}, $3$-colorable graph \cite{kms94}, terminal embeddings \cite{cn21}, quantum complexity theory \cite{jjuw11}, obtaining tight approximation ratio for MAXCUT \cite{gw95}, solving sum of squares programs \cite{bs16,fkp19}, and so on.

Mathematically, a semidefinite program can be defined as follows.\begin{definition}[Semidefinite Programming (SDP)]\label{def:sdp_primal}
Given a collection of matrices $A_1, A_2, \cdots A_m \in \R^{n \times n}$, and $b \in \R^m$, the goal is to find a matrix $X \in \R^{n \times n}$ such that 
\begin{align}\label{eqn:sdp-primal}
    \min \qquad & C \bullet X\\
    \text{s.t.} \qquad & A_i \bullet X = b_i \qquad \forall i\in [m] \nonumber\\
    & X \succeq 0 \nonumber
\end{align}
Here $C, X \in \R^{n \times n}$, $A\in \R^{m\times n^2}$, $b\in \R^m$.
\end{definition}

Over the last century, many efforts have been put into optimizing the running time of semidefinite programming \cite{s77,yn76,k80,kte88,nn89,v89,km03,bv02,lsw15,jlsw20,jlsw20,jklps20,hjstz22}. Based on how many iterations you need to solve semi-definite programming, the previous work of solving semi-definite programming can be splitted into three lines: cutting plane method \cite{lsw15,jlsw20}, log barrier function \cite{jklps20,hjstz22}, volumetric/hybrid barrier function \cite{hjstz22}.

Many graph problems can be written as semidefinite programs with certain structure.
For example, the famous Goemans-Williamson algorithm~\cite{gw95} relaxes $\mathsf{MaxCut}$ problem to a semidefinite program.
When the underlying graph has low treewidth, the corresponding SDP also has small treewidth. While the $\mathsf{MaxCut}$ problem takes exponential time in treewidth \cite{lms11} (under plausible hyposthesis like Exponetial Time Hypothesis~\cite{ip01}), when treewidth is constant, it can be easily solved in linear time; on the other hand, before our work, it was not known whether $\mathsf{MaxCut}$ SDP can be solved in nearly linear time even when treewidth is constant.
Another graph problem is Lov\'asz theta function~\cite{l79}, it is an important quantity in noisy communication and is naturally defined as a semidefinite program.

In practice, data sets on graphs often have small treewidth, see e.g.~\cite[Table 1]{zl18}, and \cite[Appendix B]{dly21}. Therefore SDP with small treewidth is a problem of practical interest.

In this work, we consider SDP for which the treewidth is small.
Formally,
\begin{definition}[Treewidth of SDP] \label{def:sdp-graph}
Given an SDP (Definition~\ref{def:sdp_primal}), we define its SDP graph as a graph $G = (V, E)$ where $V = [n]$ and
\begin{align} \label{eqn:sdp-graph}
    E(G) = \{(u,v) \in V\times V: C_{u,v}\ne 0\} \cup \bigcup_{i\in [m]} (\supp(A_i) \times \supp(A_i))
\end{align}
where $$\supp ( A ) := \{u \in V : A_{u,*} \ne 0\}.$$

Treewidth of an SDP is defined as treewidth of the SDP graph.
\end{definition}
Let $\tau$ denote the treewidth of the SDP graph.
The best previous work on small treewidth semidefinite programming is due to \cite{zl18}, which runs in $n^{1.5} \tau^{6.5}$ time.
A natural question to consider is,
\begin{center}
{\it
    Can we solve SDP much faster under treewidth assumption, i.e., in particular can we solve SDP in nearly linear time in $n$?
}
\end{center}
In this work, we answer the above question in affirmative.

\subsection{Our Results}
\subsubsection{Low-Treewidth SDP}
We make the following linear independence assumption on our SDP. This assumption is standard in IPM literature, and is also present previous works \cite{zl18}.

\begin{assumption}[Linear Independence] \label{assump:linear-independence}
The constraint matrices $A_1, \ldots, A_m \in \R^{n \times n}$ are linearly independent.
\end{assumption}

Under Assumption~\ref{assump:linear-independence}, the total number of constraints $m$ is bounded by $m \le \tau n$.
 
For the standard semidefinite programming (without low-treewidth assumptions), the best previous algorithms are due to \cite{jklps20,hjstz22}. Those algorithms are based on dual-only central path algorithms. Due to technical reasons, even with small treewidth assumption, those techniques lead to algorithm with running time at least $\Omega(m^\omega)$,\footnote{Here $\omega$ denotes the exponent of matrix multiplication, i.e., multiplying an $n \times n$ matrix with another $n \times n$ matrix takes $n^{\omega}$ time. Currently, $\omega \approx 2.373$ \cite{w12,aw21}.} and cannot solve semidefinite programming with time only linear in $n$. 

We state our main result as follows:

\begin{theorem}[Our Result on SDP, Informal Version of Theorem~\ref{thm:main_sdp_second}]\label{thm:main_sdp_informal}
There is an algorithm that solves semidefinite programming \eqref{eqn:sdp-primal} with accuracy $\epsilon>0$ in 
\begin{align*}
O( n \cdot \tau^{2\omega+0.5} \log(1/\epsilon))
\end{align*}
time.
\end{theorem}

\subsubsection{Low-Treewidth LP}
Furthermore, we also improve the best previous low-treewidth LP result due to \cite{dly21} (which runs in $n \tau^2$ time). Before stating our result for LP, let us define treewidth of an LP.
\begin{definition}[Linear Programming (LP)]
Given matrix $A \in \R^{m\times n}$ and vectors $b\in \R^m$ and $c\in \R^n$, the goal is to find a vector $x \in \R^n$ such that
\begin{align}\label{eqn:lp-primal}
    \min \qquad & c^\top x\\
    \text{s.t.} \qquad & A x = b \nonumber\\
    & x \ge 0 \nonumber
\end{align}
\end{definition}
\begin{definition}[Treewidth of LP] \label{defn:lp-dual-graph}
Given an LP, we construct the LP dual graph where vertices are constraints, and there is an edge between two vertices if they share a common variable (with non-zero coefficients). Then treewidth of the LP is defined as the treewidth of its LP dual graph.
\end{definition}

\begin{theorem}[Our Result on LP, Informal Version of Theorem~\ref{thm:main_lp_formal}]\label{thm:main_lp_informal}
There is an algorithm that solves linear programming \eqref{eqn:lp-primal} with accuracy $\epsilon>0$ in
\begin{align*} 
O(n \cdot \tau^{(\omega+1)/2} \log(1/\epsilon) ) 
\end{align*}
time.
\end{theorem}

\begin{table}[!ht]
 \centering
 \begin{tabular}{|l|l|l|l|l|} \hline 
    {\bf Problem} & {\bf Year} & {\bf Authors} & {\bf References} & {\bf Time} \\ \hline 
    LP & 2020 & Ye & \cite{y20} & $n^{1.25} \poly(\tau)$ \\ \hline 
    LP & 2021 & Dong, Lee, Ye & \cite{dly21} & $n \tau^2$ \\ \hline 
    LP & 2022 & This work & Theorem~\ref{thm:main_lp_informal} & $n \tau^{(\omega+1)/2}$ \\ \hline
    {\bf Problem} & {\bf Year} & {\bf Authors} & {\bf References} & {\bf Time} \\ \hline
    SDP & 2018 & Zhang and Lavaei & \cite{zl18} & $n^{1.5} \tau^{6.5}$ \\ \hline 
    SDP & 2022 & This work & Theorem~\ref{thm:main_sdp_informal} & $n \tau^{2\omega+0.5}$ \\ \hline
 \end{tabular}
 \caption{Summary of results on low-treewidth LP and SDP.}
 \label{tab:summary_of_our_running_time}
\end{table}

\subsubsection{Decomposable SDP}
There are cases where Theorem \ref{thm:main_sdp_informal} does not apply but we can still obtain an efficient SDP solving algorithm. One such case is summarized in the following theorem. To state the result, let us define the sparsity graph, a graph with potentially lower treewidth than the SDP graph (Definition~\ref{def:sdp-graph}).
\begin{definition}[{Sparsity graph \cite{zl18}}] \label{def:sparsity-graph}
Given an SDP (Definition~\ref{def:sdp_primal}), we define its sparsity graph as a graph $G = (V, E)$ where $V = [n]$ and
\begin{align}
E(G) = \{(u,v) \in V\times V : C_{u,v}\ne 0\} \cup \bigcup_{i\in [m]} \{(u,v)\in V\times V: A_{i,u,v} \ne 0\}.
\end{align}
\end{definition}

\begin{theorem}[Our Result on Decomposable SDP, Informal Version of Theorem~\ref{thm:main_sdp_general}] \label{thm:sdp-general}
Given an SDP \eqref{eqn:sdp-primal}.
Suppose we have a tree decomposition $\cT = \{T_1,\ldots, T_\ell\}$ of the sparsity graph $G$ (Definition~\ref{def:sparsity-graph}).
Suppose in addition that for every constraint $A_i$, we are given a connected (in the tree decomposition) set of bags $T_{i,1},\ldots, T_{i,m_i}$ such that the union of these bags contains the support of the constraint, i.e.,
\begin{align*}
\{(u,v)\in V\times V: A_{i,u,v} \ne 0\} \subseteq \bigcup_{j\in [m_i]} (T_{i,j} \times T_{i,j}).
\end{align*}

Let $\gamma_{\max}$ be the maximum number of constraints a bag correspond to.
Let $\tau$ be the maximum size of a bag.
Then the SDP can be solved in time $\wt O(n \tau^{0.5} (\tau^2 + \gamma_{\max})^\omega)$.
\end{theorem}

A special case for which the condition in Theorem~\ref{thm:sdp-general} is satisfied is the ``Network Flow Semidefinite Program'' studied in \cite{zl18}.
\begin{definition}[Network Flow Semidefinite Program]
A network flow SDP is an SDP (Definition~\ref{def:sdp_primal}) where for every constraint matrix $A_i$, there exists a vertex $k_i \in [n]$ such that $A_{i,u,v} \ne 0 \Rightarrow k_i \in \{u,v\}$.
\end{definition}

\cite{zl18} gave an algorithm to solve network flow semidefinite programs in $\wt O(n^{1.5} \tau^{3.5} (\tau+d_{\max}m_{\max})^{3.5})$ time, where $d_{\max}$ is the maximum degree of the tree $\cT$ and $m_{\max}$ is the maximum number of network flow constraints at any vertex (i.e., $m_{\max} := \max_{k\in V}\#\{i\in [m]: k_i = k\}$).

Using Theorem~\ref{thm:sdp-general}, we can achieve almost linear running time for network flow semidefinite programs.
\begin{corollary}
The network flow semidefinite program can be solved in time $\wt O(n \tau^{0.5} (\tau^2 + \gamma_{\max})^\omega)$.
\end{corollary}
Note that for network flow SDPs, a trivial bound for $\gamma_{\max}$ is $\gamma_{\max} \le \tau m_{\max}$.

\subsection{Applications}
In this section, we discuss three applications: MaxCut SDP, Lovasz Theta function, and low rank matrix completion.

\subsubsection{\texorpdfstring{$\mathsf{MaxCut}$}{} SDP and \texorpdfstring{$\mathsf{MaxkCut}$}{} SDP}
For a graph, a maximum cut is a cut whose size is at least the size of any other cut. That is, it is a partition of the graph's vertices into two complementary sets $S$ and $T$, such that the number of edges between the set $S$ and the set $T$ is as large as possible. The problem of finding a maximum cut in a graph is known as the $\mathsf{MaxCut}$ Problem. This problem is NP-complete \cite{k72}, and also APX-hard \cite{py91}.
The $\mathsf{MaxkCut}$ problem is a natural generalization of $\mathsf{MaxCut}$, where we would like to partition the set of vertices into $k$ disjoint parts, and maximize the number of edges between different parts. A $\mathsf{Max2Cut}$ problem is the same as a $\mathsf{MaxCut}$ problem.

$\mathsf{MaxCut}$ and $\mathsf{MaxkCut}$ admit natural SDP relaxations.
\begin{definition}[{$\mathsf{MaxCut}$ SDP \cite{gw95} and $\mathsf{MaxkCut}$ SDP \cite{fj97}}]
Let $L_G$ be the weighted Laplacian matrix for a graph $G$ with $n$ vertices. There is a randomized algorithm to solve $\mathsf{MaxkCut}$ with an approximation ratio of $1-1/k$ based on solving
\begin{align*}
 \mathrm{maximize} & ~ \frac{k-1}{2k}  L_G \bullet X  \\
 \mathrm{subject~to} & ~ X_{i,i} = 1, & ~ \forall i \in [n] \\
 & ~ X_{i,j} \geq \frac{-1}{k-1}, & ~ \forall (i,j) \in E(G) \\
 & ~ X \succeq 0.
\end{align*}
\end{definition}
The classic Goemans-Williamson 0.878-approximation algorithm \cite{gw95} for $\mathsf{MaxCut}$ is recovered by setting $k=2$ and removing the redundant constraint $X_{i,j}\geq -1$.
This is the best possible approximation ratio under the Unique Games Conjecture \cite{kkmo07}.

In both the $\mathsf{MaxCut}$ relaxation and the $\mathsf{MaxkCut}$ relaxation, note that the SDP graph (Definition~\ref{def:sdp-graph}) is the same as the given graph $G$. Therefore SDP treewidth $\tau_{\sdp}$ is equal to treewidth of graph $G$.\footnote{To apply Theorem~\ref{thm:main_sdp_informal}, we need a way to deal with inequality constraints. See Section~\ref{sec:sdp_first:discuss_inequality} for a discussion.}

\subsubsection{Lov\'asz Theta Function}

We discuss another famous application of SDP which is the Lov\'asz theta function,
\begin{definition}[Lovasz Theta Function, \cite{l79}]\label{def:lovasz_theta}
Given a graph $G = (V,E)$ and $y_{i,j} \in \R$ for each $(i,j) \in E$. Let $n = |V|$. Let $\ov{E}$ denote the complement of $E$. Let $\ov{G} = (V, \ov{E})$. 
The Lov\'asz number $\vartheta(G)$ of a graph $G$ is the optimal value to the following semidefinite program

\begin{align*}
    \mathrm{minimize~} & ~  \begin{bmatrix} I & {\bf 1} \\ {\bf 1}^\top & 0 \end{bmatrix} \bullet X \\
    \mathrm{subject~to~} & ~ X_{i,j} = 0 , \forall (i,j) \in \ov{E} \\
    & ~ X_{n+1,n+1} = 1 \\
    & ~ X \succeq 0 .
\end{align*}
 \end{definition}

The SDP graph (Definition~\ref{def:sdp-graph}) is obtained by adding a new vertex $o$ to the complement graph $\ov G$ of input graph $G$, with an edge $(i,o)$ for all $i\in V$.
Given a tree decomposition of $\ov G$ with bag size $\tau$, a tree decomposition of the SDP graph with bag size $\tau+1$ can be obtained by adding the new vertex $o$ into every bag.
So SDP treewidth is small when $\ov G$ has small treewidth.

\subsubsection{Low Rank Matrix Completion}
In the low rank matrix completion problem, we are given a partially filled matrix $B$, where entries $(i,j)\in \Omega$ are filled. The goal is to find a matrix $X$ of smallest rank that agrees with $B$ on $\Omega$.
\begin{align*}
    \mathrm{minimize~} & ~ \mathrm{rank}(X) \\
    \mathrm{subject~to~} & ~ X_{i,j} = B_{i,j} , \forall (i,j) \in \Omega.
\end{align*}

This problem is non-convex and in general difficult to solve.
A popular relaxation \cite{candes2009exact} of the problem is to minimize the nuclear norm $\|\cdot\|_*$ instead.
\begin{align*}
    \mathrm{minimize~} & ~ \|X\|_* \\
    \mathrm{subject~to~} & ~ X_{i,j} = B_{i,j} , \forall (i,j) \in \Omega.
\end{align*}
Furthermore, \cite{candes2009exact} rewrote the above formulation in the following equivalent form:
\begin{align*}
    \mathrm{minimize~} & ~ I \bullet Y\\
    \mathrm{subject~to~} & ~ Y_{i,n+j} = B_{i,j}, \forall (i,j) \in \Omega \\
    & ~ Y \succeq 0
\end{align*}
where $Y \in \R^{2n \times 2n}$.

Let $G=([n], \Omega)$ be the adjacency graph.
The SDP graph (Definition~\ref{def:sdp-graph}) is as follows:
for every edge $(i,j)\in \Omega$, the SDP graph contains the clique $\{i,j,i+n,j+n\}$.

Given any tree decomposition of $G$, we can define a tree decomposition of the SDP graph by replacing every bag $S$ with $\{i,i+n : i\in S\}$.
So SDP treewidth is at most twice $\tw(G)$.
When $\tw(G)$ is small, SDP treewidth is small.

\subsection{Related Work}

\paragraph{Linear programming}

Linear programming is an old topic in computer science. Simplex algorithm~\cite{d47} is one of the most important algorithm in the history of linear programming, and it has exponential running time. 
Ellipsoid method reduced the running time to polinomial~\cite{k80}, however, it is in practice slower than simplex method. 
The interior point method~\cite{k84} is a major breakthrough, because it has have theoretical polynomial running time and stably fast practical performance on real-world problems.
Assuming $d$ is the number of constraints and $n$ is the number of variables, when $d = \Omega(n)$, Karmarkar's algorithm has a running time of $O^{*}(n^{3.5})$.
The running time is further improved to $O^{*}(n^3)$ in ~\cite{v87, r88} and $O^{*}(n^{2.5})$ in ~\cite{v89}. Recently, \cite{cls19} shows how to solve linear program in $O( n^{\omega} + n^{2.5-\alpha / 2} + n^{2+1/6} )$ time, where the $\omega$ is the exponent of matrix multiplication, and $\alpha$ is the dual exponent of matrix multiplication. Later, \cite{lsz19} generalize that algorithm to a more broader optimization problem called empirical risk minimization. \cite{sy21} shows how to re-produce the result in \cite{cls19} via oblivious sketching matrices instead of non-oblivious sampling matrices. When the linear programming constraint matrix is tall and dense, \cite{blss20} shows how to solve it in $\wt{O}(n d) + \poly(d)$ time. For a linear program with small treewidth $\tau$, the work \cite{dly21} shows how to solve it in $\wt{O}(n \tau^2)$ time.
 
 \paragraph{Semidefinite programming}
 
In the previous SDP literature, there are two lines of SDP solvers: second order methods and first order methods.
Second order methods usually have running time with logarithmic dependence on $\epsilon$, while first order methods usually have polynomial dependence on $\epsilon$.

Second order methods can be further splitted into several lines: first line is cutting plane methods, for example \cite{lsw15,jlsw20}. The algorithm in \cite{lsw15,jlsw20} can solve semidefinite programming in $O( m (mn^2 + m^2 + n^{\omega}) )$ time. The second line is interior point methods (associated with certain self-concordant barrier functions), for example, \cite{nn92,jklps20,hjstz22} use log-barrier function, and \cite{a00,hjstz22} use hybrid barrier function. The algorithm in \cite{jklps20} can solve SDP in $O( \sqrt{n}(mn^2 + n^{\omega} + n^{\omega}) )$ time. Furthermore, \cite{hjstz22} shows that as long as $m =\Omega(n^2)$, we can solve semidefinite programming in $O( m^{\omega} + m^{2+1/4} )$ time.
 
The interior point method is a second-order algorithm. Second-order algorithms usually have logarithmic dependence on the error parameter $1/\epsilon$. First-order algorithms do not need to use second-order information, but they usually have polynomial dependence on $1/\epsilon$. There is a long list of work focusing on first-order algorithms \cite{ak07,gh16,al17,cdst19,lp19,ytfuc19,jy11,alo16}.

For a semidefinite program with bounded treewidth, the work \cite{zl18} shows how to solve it in $\wt O(n^{1.5} \tau^{6.5})$ time.

%% file: tech.tex
\section{Technique Overview}

In this section, we explain the previous technique and summarize our new techniques.
\begin{itemize}
    \item In Section~\ref{sec:previous_technique}, we briefly summarize the techniques in \cite{zl18} and explain the barrier of that algorithm.
    \item In Section~\ref{sec:tech_n_poly_tau}, we provide an overview of an $n \cdot \poly (\tau)$ time algorithm.
    \item In Section~\ref{sec:bottleneck}, we briefly analyze running time of the algorithm in previous section and explain the bottlenecks.
    \item In Section~\ref{sec:improve_bottleneck}, we discuss how we improve the running time for bottlenecks, thus achieving an improved algorithm for low-treewidth SDP.
    \item In Section~\ref{sec:improve_lp_result}, we discuss how we achieve improved algorithm for low-treewidth linear program.
\end{itemize}

\subsection{An Overview of Previous Techniques}\label{sec:previous_technique}
In this section, we briefly summarize the technique in \cite{zl18}.

\cite{zl18} uses interior point method to solve SDP. However, running IPM directly would be very expensive. The first step of their algorithm is to rewrite the SDP and reduce the number of (real) variables from $n^2$ to $n\cdot \poly(\tau)$.

\paragraph{Variable reduction}
Suppose we have an SDP of form \eqref{eqn:sdp-primal} with a tree decomposition of maximum bag size $\tau$.

\cite{zl18} creates another SDP of smaller number of variables, where every bag in the tree decomposition becomes one $\tau\times \tau$ PSD matrix, which corresponds to an $\tau\times \tau$ principal minor in the original PSD matrix.

\begin{align}\label{eqn:sdp-ctc}
    \min \qquad & \sum_{j\in [b]} C_j \bullet X_j\\
    \text{s.t.} \qquad & A_i \bullet X_{j_i} = b_i \qquad \forall i\in [m] \nonumber\\
    & \cN_{i,j} (X_i) = \cN_{j,i} (X_j) \qquad \forall (i,j)\in E(T) \nonumber\\
    & X_j \succeq 0  \qquad \forall j\in [b] \nonumber
\end{align}
See Appendix~\ref{sec:sdp_first} for a more detailed explanation of the program.

The new SDP has three types of contraints:
\begin{itemize}
    \item $A_i \bullet X_{j_i} = b_i$, which corresponds to linear constraints in the original SDP;
    \item $\cN_{i,j}(X_i) = \cN_{j,i}(X_j)$, which says that when two minors overlap in the original PSD matrix, they should have the same value in overlapping positions;
    \item $X_j\ge 0$, meaning that each minor should be PSD.
\end{itemize}

\cite{zl18} proved that given any feasible solution to the new SDP, we can compute in $O(n \tau^3)$ time a succinct representation of a feasible solution to the original SDP with the same objective value. The succinct representation is of form $X = UU^\top$ where $U\in \R^{n\times \tau}$ is a rectangular matrix.

After this reduction, we run IPM iteration directly.\footnote{In fact, \cite{zl18} ran IPM on the dual SDP, because the maximum degree of a bag in the tree decomposition could be large. However, we can easily make the maximum degree $O(1)$ by adding $O(n)$ bags. This does not change the asymptotic running time but results in a simpler algorithm.}

\cite{zl18} uses interior point method to solve \eqref{eqn:sdp-ctc}. For every step, we need to perform computation of the following kind
\begin{align} \label{eqn:zl18-computation-key-step}
    \left( A_\lp H_x A_\lp^\top \right)^{-1} v 
\end{align}
where
\begin{itemize}
    \item $A_\lp = \begin{bmatrix}
    A \\ N
    \end{bmatrix} \in \R^{m_\lp \times n_\lp}$ is the constraint matrix (where $m_\lp = O(n_\lp)$) (to be more precise, each of the first $m_\sdp$ rows corresponds one original SDP constraint, and rows after that correspond to overlap constraints),
    \item $H_x$ is the Hessian matrix, a block-diagonal matrix.
\end{itemize}
Because of the low-treewidth assumption, these updates can be implemeneted efficiently, as discussed below.

\paragraph{Number of iterations}
Number of iterations we need to run is controlled by the number of variables and geometry of the primal space. (More precisely, number of iterations is square root of sum of self-concordance value of all blocks.)
Because we can have $O(n)$ $\tau\times \tau$ matrix variables, each of them is in the PSD cone, which is is $\tau$-self-concordant, the total number of iterations is
\begin{align*}
(\text{\# of matrix vars}\cdot \text{self-concordance parameter for PSD cone})^{0.5}
& = \wt O(\sqrt{ n \tau}).
\end{align*}

\paragraph{Cost per iteration}
In each step, we need to perform a computation of form Eq.~\eqref{eqn:zl18-computation-key-step}.
When the tree decomposition has $O(1)$ maximum degree, $A_\lp H_x A_\lp^\top$ is block-sparse and we can compute a Cholesky decomposition $A_\lp H_x A_\lp^\top = LL^\top$ is $O(n \cdot \poly(\tau))$ time.
Then $(A_\lp H_x A_\lp^\top)^{-1} v = L^{-\top} L^{-1} v$ can be computed in $O(n\cdot \poly(\tau))$ time by performing multiplication from right to left.

Each of the matrix factors in the above equation can be computed in $\wt O(n\tau^6)$ time, and contains $O(n\tau^4)$ non-zero entries, so cost per iteration is $\wt O(n\tau^6)$.

\paragraph{Putting it all together}

We have
\begin{align*}
 \#\mathrm{iters} \cdot \mathrm{~(cost~per~iters)} 
= & ~ \wt O(\sqrt{n\tau}) \cdot \wt O(n \tau^6)  \\
= & ~ \wt O(n^{1.5} \tau^{6.5})
\end{align*}
So the overall running time of \cite{zl18}'s algorithm is $\wt O(n^{1.5} \tau^{6.5})$.

\subsection{How to Get An \texorpdfstring{$n \cdot \poly(\tau)$}{} Algorithm?}\label{sec:tech_n_poly_tau}
\cite{zl18} solved \eqref{eqn:sdp-ctc} using IPM with $O(n^{0.5} \tau^{0.5})$ iterations and $O(n \poly(\tau))$ time per iteration. The number of iterations usually pays at least $\Omega (\sqrt{n})$  in the field of IPM methods due to known lower bound of barrier method.\footnote{\cite{nn94} showed that any self-concordant barrier of $n$-dimensional simplex or hypercube must have self-concordance parameter at least $n$. This bound is tight, because \cite{ly21} recently showed that there always exists an $n$-self-concordant barrier} Therefore we seek to improve cost per iteration.
Using robust IPM to solve linear program in tree-width setting, in each iteration, we only need to update those coordinates where there is a significant enough change, and this would reduce total number of coordinate updates to $O(n \poly(\tau))$, giving an $O(n\poly(\tau))$ algorithm.

More specifically, \cite{dly21} gave an $O(n\tau^2)$ algorithm for low-treewidth LP. Their central path equation is 
\begin{align*}
    s/t + w \nabla \phi(x) =&~ \mu\\
    Ax =&~ b \\
    A^\top y + s =&~ c 
\end{align*}
where $x$ is the primal variable, $y$ is the dual variable, $s$ is the slack variable, and $\mu$ denotes the error.
The central path is defined as the path of $(x,s)$ as $t$ goes from initial value to $0$.

Then the Newton step is
\begin{align*}
    \delta_s/ t + w H_x \delta_x =&~ \delta_\mu \\
    A \delta_x = & ~ 0 \\
    A^\top \delta_y + \delta_s =&~ 0
\end{align*}
where $H_x = \nabla^2 \phi(x)$ is the Hessian.

We only need to follow the central path approximately, meaning that when computing the steps, we could use $\ov x$ instead of $x$, where $\ov x$ is a vector close enough to $x$.
The IPM algorithm works roughly in the following sense.
\begin{itemize}
    \item $t\gets t_{\mathrm{start}}$
    \item While $t \ge t_{\mathrm{end}}$ do
    \begin{itemize}
    \item $\delta_{\mu} \gets \argmin_{\|\delta_\mu\|_{\ov x}^* = \alpha} \langle \nabla_\mu \Psi_\lambda(\|\mu_i\|_{\ov x_i}^*), \mu+\delta_\mu\rangle$
    \item $\delta_x \gets (H_{\ov x}^{-1} - H_{\ov x}^{-1} A_\lp\top (A_\lp H_{\ov x}^{-1} A_\lp^\top)^{-1} A_\lp H_{\ov x}^{-1}) \delta_\mu$
    \item $\delta_s \gets t A_\lp^\top (A_\lp H_{\ov x}^{-1} A_\lp^\top)^{-1} A_\lp H_{\ov x}^{-1} \delta_\mu$
    \item $x \gets x+\delta_x$, $s \gets s + \delta_s$
    \item $t \gets  t \cdot (1-h)$
    \end{itemize}
\end{itemize}
The parameters satisfy $t_{\mathrm{start}} / t_{\mathrm{end}} = \poly(n)$, $h = \Theta(n^{-1/2})$, so there are $\wt O(n^{1/2})$ steps.

The general LP and SDP central path algorithm (without treewidth assumption) maintains $(A H_{\ov{x}}^{-1} A^\top )^{-1}$ directly, since we don't have the benefit of Cholesky decomposition (coming from treewidth assumption). This would lead to algorithms with cost-per-iteration super-linear in $n$.

The key idea of \cite{dly21} is that, instead of maintaining the pair $(x,s)$ on the central path explicitly, we maintain them implicitly (``multiscale coefficients'')
\begin{align} \label{eqn:multiscale-coef}
x &= \wh x + H_{\ov{x}}^{-1/2} \beta_x c_x - H_{\ov{x}}^{-1/2} \cW^\top (\beta_x h + \epsilon_x) \\
s &= \wh s + H_{\ov{x}}^{1/2} \cW^\top(\beta_s h + \epsilon_s), \nonumber
\end{align}
where $\cW = L_{\ov x}^{-1} A H_{\ov x}^{-1/2}$ and $\hat x, h, \hat s$ are sparsely changing vectors.
Here $L_{\ov x}$ is the Cholesky factor in the Cholesky decomposition $AH_{\ov x}^{-1} A^\top = L_{\ov x}L_{\ov x}^\top$. Under a central path step and under a sparse update to $\ov x$, all terms in the multiscale representation \eqref{eqn:multiscale-coef} can be maintained efficiently (see Section~\ref{sec:framework:exact_ds}). Therefore if the total number of coordinate updates is $O(n)$, and the time cost per coordinate is $O(\poly(\tau))$, then the overall running time is $O(n\poly(\tau))$.

Similar ideas can be applied to our treewidth SDP case.
We maintain primal-dual pairs $(x_i,s_i)$, and let them follow the central path. In each step, we need to run approximate Newton step
\begin{align} \label{eqn:newton-step}
\delta_x &= (H_{\ov x}^{-1} - H_{\ov x}^{-1} A_\lp^\top (A_\lp H_{\ov x}^{-1} A_\lp^\top)^{-1} A_\lp H_{\ov x}^{-1}) \delta_\mu,\\
\delta_s &= t A_\lp^\top(A_\lp H_{\ov x}^{-1} A_\lp^\top)^{-1} A_\lp H_{\ov x}^{-1} \delta_\mu, \nonumber
\end{align}
where
\begin{itemize}
    \item $x = \begin{bmatrix}\vec(x_1) & \cdots & \vec(x_{n})\end{bmatrix}^\top \in \R^{n_\lp}$ is the primal variable,
    \item $s = \begin{bmatrix}\vec(s_1) & \cdots & \vec(s_{n})\end{bmatrix}^\top \in \R^{n_\lp}$ is the slack variable,
    \item $A_\lp = \begin{bmatrix}
    A \\ N
    \end{bmatrix} \in \R^{m_\lp \times n_\lp}$ is the constraint matrix (where $m_\lp = O(n_\lp)$) (to be more precise, each of the first $m_\sdp$ rows corresponds one original SDP constraint, and rows after that correspond to overlap constraints),
    \item $H_{\ov x} = \nabla^2 \phi(\ov x)\in \R^{n_\lp}$ is the Hessian matrix for some potential function $\phi$,
    \item $\delta_\mu \in \R^{n_\lp}$ is the step of the error vector $\mu$ (see Section \ref{sec:robust_ipm} for its definition).
\end{itemize}

We maintain the a block version of the multiscale coefficients.
We remark that one of the major difference for semidefinite progamming compared to linear programming is that the Hessian matrix $H_x$ is no longer a diagonal matrix.
Nevertheless, the Hessian matrix is a block diagonal matrix. Therefore its inverse and square root can be computed in $\poly(\tau)$ time.
When the tree decomposition has $O(1)$ maximum degree, our program \eqref{eqn:sdp-ctc}  has treewidth $O(\tau_\lp) = O(\tau^2)$ in the sense of Definition~\ref{def:general_lp_treewidth}.
Therefore we can still efficiently compute and maintain the Cholesky decomposition of one factor in the Newton step: $A_\lp H_{\ov x}^{-1} A_\lp^\top = L_{\ov x}L_{\ov x}^\top$.

To maintain the central path, we need to have a data structure which can support (1) central path steps \eqref{eqn:newton-step}, and (2) report coordinates on which $x$ and $\ov x$ differ too much.

The first task (central path steps) is maintained using \textsc{ExactDS} (Section~\ref{sec:framework:exact_ds}).
For the second task (finding important coordinates), a natural idea is to use Johnson–Lindenstrauss sketch (Lemma~\ref{lem:jl_matrix}).
Suppose we have a data structure which can answer range queries on $\Phi \ov x$ and $\Phi x$ (where $\Phi \in \R^{r\times n_\lp}$ is the JL matrix), then we can efficiently find the coordinates that differ much.

However, because we maintain $(x,s)$ implicitly, maintaining $\Phi x$ under central path steps and under change of multiscale representation and answering range queries is not a trivial task.
We resolve this issue by using \textsc{BlockBalancedSketch} (Section~\ref{sec:framework:block_balanced_sketch}), a block version of balanced sampling tree sketch in \cite{dly21}.
Using \textsc{BlockBalancedSketch} and \textsc{BlockVectorSketch} (Section~\ref{sec:framework:block_vector_sketch}, an easy application of segment trees), we are able to implement \textsc{BatchSketch} (Section~\ref{sec:framework:batch_sketch}), a data structure for maintaining sketches of $(x,s)$.

Furthermore, even equipped with the sketching data structures, we still need an algorithm to maintain an approximation of the sketched vector. Because we are able to control the step size in the IPM algorithm, we know that $(x,s)$ is in fact slowly changing under $\ell_2$. Therefore, we are able to maintain a sparsely changing approximation of $(x,s)$.
This is achieved using \textsc{ApproxDS} (Section~\ref{sec:framework:approx_ds}).
By removing a sampling step in the corresponding data structure in \cite{dly21}, we fix a technical issue of \cite{dly21} (see Remark~\ref{rmk:dly_linf_approx_error}).
\textsc{ApproxDS} uses \textsc{BatchSketch} as a subroutine.

Finally, by combining \textsc{ExactDS} and \textsc{ApproxDS}, we are able to maintain $(x,s)$ implicitly, and to maintain an approximation $(\ov x, \ov s)$ explicitly.
This gives our central path maintenance algorithm \textsc{CentralPathMaintenance} (Section~\ref{sec:framework:cpm_ds}).

\subsection{Bottlenecks for Running Time} \label{sec:bottleneck}
In this section we give a brief analysis of running time. The actual analysis is much more complicated, but the highlight the bottlenecks here.

In our central path maintenance data structure \textsc{CentralPathMaintenance} we regularly updates and reconstructs the data structure \textsc{ApproxDS}, because its running time is quadratic in the number of iterations between two reconstructs. (This is because after $q$ iterations, $\|x^{(q)} - x^{(0)}\|_2^2$ is of order $q^2$, and $\ell_\infty$ approximations of $x^{(q)}$ and of $x^{(0)}$ can differ in $q^2$ coordinates.)
There are $N = O(\sqrt{n_\sdp \tau_\sdp w})$ iterations (where $w$ is a parameter to be chosen, it does not affect running time as long as it is large enough) and we reconstruct the data structure after every $q$ (a parameter to be chosen) iterations.

\textbf{Restart time:} Over the entire $N$ iterations of the algorithm (Algorithm~\ref{alg:informal_cpm}), we need to do $N/q$ times reconstructions (restart). Each reconstruction takes $n \poly(\tau)$ time. Thus, the total number of restart time is $O(N/q \cdot n \poly(\tau))$.

\textbf{Update time:} Between every two restarts, there are $\wt O(q^2)$ coordinate changes. So there are $\wt O(N/q \cdot q^2) = O(N q)$ coordinate changes in total. Each coordinate change takes $O(\poly(\tau))$ time to update. So the total cost for updates is $O(N q \cdot \poly(\tau))$.

Therefore, the final running time is a result of balancing the total restart time and update time.

\begin{algorithm}[!ht]\caption{An informal version of our central path maintenance algorithm, Algorithm~\ref{alg:cpm}}\label{alg:informal_cpm}
\begin{algorithmic}[1]
\State Initialize \textsc{ExactDS} and \textsc{ApproxDS}
\For{$i = 1 \to N$}
    \If{we have not restarted for $q$ iterations}
        \State {\color{blue}/* Restart */}
        \State Restart our data structure
    \EndIf 
    \State {\color{blue}/* Move */}
    \State Run one step of central path in \textsc{ExactDS} \Comment{Implicitly maintain $x$ and $s$}
    \State Use \textsc{ApproxDS} to detect which coordinates of $(x,s)$ have heavy changes. \Comment{Compute update to $(\ov x, \ov s)$}
    \State Update \textsc{ExactDS} according to update in $(\ov x, \ov s)$
    \State Update \textsc{ApproxDS} according to update in $(\ov x, \ov s)$
\EndFor 
\end{algorithmic}
\end{algorithm}

In the following, we plug in the exact exponents of $\tau$ and compute the running time.

\paragraph{Restart}
Restart can be splitted into two steps, one is initialization and the other is output. The most time-consuming step is computing Choleky factorization in initialization.

We have $n_\sdp$ blocks, where each block is a $\tau_\sdp \times \tau_\sdp$ matrix. By structure of our SDP, the Cholesky factor $L_{\ov x}$ in $AH_{\ov x}^{-1} A^\top = LL^\top$ is $\wt O(\tau_\lp)$-sparse (i.e., every column has $\wt O(\tau_\lp)$ non-zero coordinates), where $\tau_\lp = \tau_\sdp^2$. So computing Cholesky factorization takes $\wt O(n_\lp \tau_\lp^2) = \wt O(n_\sdp \tau_\sdp^6)$ time.

So
\begin{align*}
    \text{overall restart time} &= \text{number of restarts} \cdot \text{cost per restart} \\
    &= \wt O(N/q) \cdot \wt O(n_\sdp \tau_\sdp^6) \\
    &= \wt O(n_\sdp^{1.5} w^{0.5} q^{-1} \tau_\sdp^{6.5}).
\end{align*}

\paragraph{Update}
As described in Algorithm~\ref{alg:informal_cpm}, update can be splitted into mainly four substeps: Central path (\textsc{ExactDS}) move, \textsc{ApproxDS} query (using sketching data structure \textsc{BatchSketch}), \textsc{ExactDS} update, and sketching data structure \textsc{BatchSketch} update.

\begin{itemize}
\item The central path move step takes $O(1)$ time per iteration because it only updates constant coefficients in the multiscale representation.

\item \textsc{ApproxDS} query takes $\wt O(q w^{-1})$ queries (on average) to the sketching data structure \textsc{BatchSketch} on average. Note that this also implies there are $\wt O(q w^{-1})$ variable block changes per iteration.
Every sketching data structure query takes $\wt O(\tau_\lp^2)$ time.

\item For each variable block change in $\ov x$, updating \textsc{ExactDS} takes $\wt O(\tau_\lp^3) = \wt O(\tau_\sdp^6)$ time due to the cost of updating Cholesky factorization. This is because one variable block change in $\ov x$ leads to a rank $O(\tau_\lp)$-change in $A H_{\ov x}^{-1} A$, and a rank-$1$ change in $A H_{\ov x}^{-1} A$ takes $\wt O(\tau_\lp^2)$ to update the Cholesky factor $L_{\ov x}$.
Therefore on average, every iteration takes $\wt O(q w^{-1} n_\sdp^6)$ time.

\item For each variable block change in $\ov x$, updating \textsc{BatchSketch} takes $\wt O(\tau_\lp^3) = \wt O(\tau_\sdp^6)$ time due to the cost of updating Cholesky factorization.
Therefore on average, every iteration takes $\wt O(q w^{-1} n_\sdp^6)$ time.
\end{itemize}

As analyzed above, on average, the update time for each iteration is $\wt O(q w^{-1} n_\sdp^6)$.
So
\begin{align*}
\text{overall update time} &= \text{number of iters} \cdot \text{update cost per iter} \\
&= \wt O(N) \cdot \wt O(q w^{-1} n_\sdp^6) \\
&= \wt O(n_\sdp^{0.5} w^{-0.5} q \tau_\sdp^{6.5}).
\end{align*}

Combining everything above, we have
\begin{align*}
    \text{overall running time} &= \text{overall restart time} + \text{overall update time} \\
    &= \wt O(n_\sdp^{1.5} w^{0.5} q^{-1} \tau_\sdp^{6.5}) + \wt O(n_\sdp^{0.5} w^{-0.5} q \tau_\sdp^{6.5}) \\
    &= \wt O(n_\sdp \tau_\sdp^{6.5})
\end{align*}
where in the last step we take $w = \tau_\sdp$, $q = n_\sdp^{0.5}$.

\subsection{Improve Running Time Bottleneck} \label{sec:improve_bottleneck}
In this section we discuss how to improve the running time bottleneck.
As described in the last section, the most time consuming step is computing Cholesky factorization $A H_{\ov x}^{-1} A^\top = L_{\ov x} L_{\ov x}^\top$ (which affects restart time), and updating Cholesky factorization under change of one variable block (which affects update time).

The way we achieve improvement is by using block structure in Cholesky-related computations.
Previously, we use the property that every column $L_{\ov x}$ is $\wt O(\tau_\lp)$-sparse, to achieve running time $\wt O(n_\lp \tau_\lp^2)$ for computing the Cholesky factorization.
In fact, $L_{\ov x}$ has more structures.
Because of our reduction to form~\eqref{eqn:sdp-ctc}, we can divide the indices of $L_{\ov x}$ into $\wt O(n_\sdp)$ blocks, where maximum block size is $\wt O(\tau_\lp)$, and such that every block column of $L_{\ov x}$ is $\wt O(1)$-block sparse.

Using this block structure, we are able to devise an algorithm that computes the Cholesky factorization $L_{\ov x}$ in $\wt O(n_\sdp \tau_\lp^\omega) = \wt O(n_\sdp \tau_\sdp^{2\omega})$ time (Lemma~\ref{lem:block_cholesky_decomposition}). There are a few other computations in restart, which can all be improved to $\wt O(n_\sdp \tau_\sdp^{2\omega})$ time using this idea.
Thus the time cost per restart becomes $\wt O(n_\sdp \tau_\sdp^{2\omega})$.
We have
\begin{align*}
    \text{overall restart time} &= \text{number of restarts} \cdot \text{cost per restart} \\
    &= \wt O(N/q) \cdot \wt O(n_\sdp \tau_\sdp^{2\omega}) \\
    &= \wt O(n_\sdp^{1.5} w^{0.5} q^{-1} \tau_\sdp^{2\omega + 0.5}).
\end{align*}

Furthermore, we are able to achieve $\wt O(\tau_\lp^\omega) = \wt O(\tau_\sdp^{2\omega})$ time for updating the Cholesky factorization $L_{\ov x}$.
There are a few other computations in update, which can all be improved to $\wt O(\tau_\sdp^{2\omega})$ time per variable block update.
Therefore update time for iteration becomes $\wt O(q w^{-1} \tau_\sdp^{2\omega})$ on average.
We have
\begin{align*}
\text{overall update time} &= \text{number of iters} \cdot \text{update cost per iter} \\
&= \wt O(N) \cdot \wt O(q w^{-1} n_\sdp^{2\omega}) \\
&= \wt O(n_\sdp^{0.5} w^{-0.5} q \tau_\sdp^{2\omega+0.5}).
\end{align*}

Combining everything above, we have
\begin{align*}
    \text{overall running time} &= \text{overall restart time} + \text{overall update time} \\
    &= \wt O(n_\sdp^{1.5} w^{0.5} q^{-1} \tau_\sdp^{2\omega+0.5}) + \wt O(n_\sdp^{0.5} w^{-0.5} q \tau_\sdp^{2\omega+0.5}) \\
    &= \wt O(n_\sdp \tau_\sdp^{2\omega+.5})
\end{align*}
where in the last step we take $w = \tau_\sdp$, $q = n_\sdp^{0.5}$.

\subsection{Improve Running Time of Low-Treewidth LP} \label{sec:improve_lp_result}

In this section, we discuss the technique we use to achieve a more efficient algorithm for LP with bounded treewidth.

Recall that the key idea of \cite{dly21} is, instead of maintaining the primal and slack variables on the (approximate) central path explicitly, we maintain a sparsely-changing representation (multiscale representation) of them.
The central path maintenance algorithm has two main parts (1) restart (which contains initialize and output), and (2) update.
In every iteration, if certain conditions are satisfied, we restart the data structure. Afterwards, we perform a central path move and maintain the relevant data structures \textsc{ExactDS} and \textsc{ApproxDS}.
The final running time is achieved by balancing restart time and update time.

We improve restart time by devising improved algorithms for computations related to Cholesky factorization.
One main step in initialization part is to compute the Cholesky factorization of the matrix $A H_{\ov x}^{-1} A^\top = L_{\ov x}L_{\ov x}^{\top}$.
Because of the treewidth assumption, the Cholesky factor $L_{\ov x}$ is $\wt O(\tau)$-column sparse, i.e., every column has at most $\wt O(\tau)$ non-zero entries.
Using this property, \cite{dly21} is able to compute the Choleksy factorization in $\wt O(n\tau^2)$ time.

We can do better by utilization more properties of the matrix. The key observation is that we can divide the indices of $A H_{\ov x}^{-1} A^\top$ into blocks of size $O(\tau)$, such that in the Cholesky factor is $\wt O(1)$-block sparse, i.e., every block column has at most $\wt O(1)$ non-zero block entries.

In this way, we can perform the Choleksy factorization algorithm on block level,
achieving an algorithm (c.f.~Lemma~\ref{lem:block_cholesky_decomposition},~\ref{lem:lp:TL}) running in time
\begin{align*}
& \#\text{blocks} \cdot (\#\text{column block sparsity})^2 \cdot \Tmat(\text{block size}) \\
& = \wt O(n/\tau) \cdot \wt O(1) \cdot \Tmat(\tau) \\
& = \wt O(n \tau^{\omega-1}).
\end{align*}

Using similar ideas, we are able to improve other Cholesky-related computations, including Cholesky factor inverse (c.f.~Lemma~\ref{lem:fast_cholesky_inverse}), product of Cholesky factor and a batch of vectors (c.f.~Lemma~\ref{lem:fast_cholesky_inverse_vector}), and so on. These are all bottleneck steps in the original \cite{dly21} algorithm.
By combining all these improvements, we achieve improved running time for restart.

The overall running time is sum of restart time and update time.
In the LP case, the number of iterations $N$ is $\sqrt n$.
Let $q$ be the data structure restart threshold.
Then overall restart time is given by
\begin{align*}
\text{overall restart time}
= & ~\text{number of restarts} \cdot \text{cost per restart} \\
= & ~ \wt O(N/q) \cdot \wt O(n\tau^{\omega-1}) \\
= & ~ \wt O(n^{1.5} q^{-1} \tau^{\omega-1}).
\end{align*}

Overall update time is
\begin{align*}
 \text{overall update time} 
= &~\text{number of iters} \cdot \text{update cost per iter} \\
= &~ \wt O(N) \cdot \wt O(q \tau^2)\\
= &~\wt O(N q \tau^2).
\end{align*}

Therefore
\begin{align*}
\text{overall running time} 
= & ~ \text{overall restart time} + \text{overall update time} \\ 
= & ~ \wt O(n^{1.5} q^{-1} \tau^{\omega-1}) + \wt O(n^{0.5} q \tau^2) \\
= & ~ \wt O(n\tau^{(\omega+1)/2})
\end{align*}
where the last step is by taking $q = n^{0.5} \tau^{(\omega-3)/2}$.

%% file: preli.tex
{\bf Roadmap.}

Here, we provide an organization for the rest of the paper.

\begin{itemize}
    \item In Section~\ref{sec:preli}, we present a few basic definitions and results used in the paper.
    \item In Section~\ref{sec:framework}, we present a general framework which covers both semidefinite programming and linear programming.
    \item In Section~\ref{sec:sdp_first}, we show how to apply our general framework to semidefinite programming solver to achieve an $O(n \cdot \poly(\tau))$ algorithm.
    \item In Section~\ref{sec:sdp_second}: we show how to further improve the running time of semidefinite programming solver to achieve an $\wt O(n \tau^{2\omega+0.5})$ time algorithm.
    \item In Section~\ref{sec:sdp_general}, we show how to solve a more general kind of low-treewidth SDP (which we call decomposable SDP) using our general framework.
    \item In Section~\ref{sec:lp}, we show how to improve running time of low-treewidth linear programming to $\wt O(n \tau^{(\omega+1)/2})$.
\end{itemize}

\section{Preliminaries}\label{sec:preli}

 In Section~\ref{sec:preli:basic}, we define some basic notations.  In Section~\ref{sec:preli:barrier}, we define self concordant barrier. In Section~\ref{sec:preli:treewidth}, we define treewidth.  In Section~\ref{sec:preli:sketching}, we introduce some backgrounds for sketching matrices. 

\subsection{Basic Notions}\label{sec:preli:basic}

For any positive integer $n$, we use $[n]$ to denote the set $\{1,2,\cdots,n\}$. 
For a matrix $A$, we use $A^\top$ to denote the transpose of $A$. For two matrices $A$ and $B$, we use $A \bullet B$ and $\langle A, B \rangle$ to denote the inner product between two matrices, i.e., $\sum_{i,j} A_{i,j} B_{i,j}$. For a square matrix $A$, we use $\det(A)$ to denote the determinant of $A$. We use $\tr(A)$ to denote the trace of matrix $A$.

For two functions $f,g$, we use the shorthand $f \lesssim g$ (resp. $\gtrsim$) to indicate that $f \leq C \cdot g $ (resp. $\geq $) for an absolute constant $C$. We use $f \eqsim g$ to mean $c f \leq g \leq C f$ for constants $c,C$.

We use $\omega$ to denote the exponent of matrix multiplication, i.e., multiplying an $n \times n$ matrix with another $n \times n$ matrix takes $n^{\omega}$ time. Currently, $\omega \approx 2.373$ \cite{w12,aw21}.

We will deal with block matrices and vectors a lot. Therefore we use a block-friendly notation.

For a fixed block pattern $N= \sum_{i\in [n]} n_i$, for $x\in R^N$, we denote $x=(x_1,\ldots, x_n)$, where $x_i \in \R^{n_i}$.
Therefore $x_i$ naturally denotes one block of $x$.
When we need to refer to coordinates of $x$, we use $x_{\text{coord}~i}$.

\begin{definition}[Mixed Norm]\label{def:mixed_norm}
Let $x\in \R^N$ be a vector with fixed block pattern $N = \sum_{i\in [n]} n_i$. We define its $(p,q)$-norm $\|x\|_{p,q}$ as
the $q$-norm of the vector $(\|x_i\|_p)_{i\in [N]}$.
\begin{align*}
    \|x\|_{p,q} := \|y \|_q , \text{~where~} y = ( \| x_1 \|_p , \cdots, \| x_n \|_p ) \in \R^{n}
\end{align*}
\end{definition}
In particular $\|x\|_{0,1} = \|x\|_0$, $\|x\|_{2,2} = \|x\|_2$.

\subsection{Self-Concordant Barrier}\label{sec:preli:barrier}

We start with defining self concordant barrier,
\begin{definition}[Self-Concordant Barrier, {\cite{nesterov2003introductory}}]
A function $\phi$ is a $\nu$-self-concordant barrier for a non-empty open convex set $K$ if $\mathrm{dom} (\phi) = K$, $\phi(x) \rightarrow + \infty$ as $x \rightarrow \partial K$, and for any $x \in K$ and for any $u \in \R^n$
\begin{align*}
    D^3\phi(x) [ u, u, u ] = 2 \| u \|_{ \nabla^2 \phi(x) } \text{~~~and~~~} \| \nabla \phi(x) \|_{ ( \nabla^2 \phi(x) )^{-1} } \leq \sqrt{\nu} .
\end{align*}
A function $\phi$ is a self-concordant barrier if the first condition holds.
\end{definition}

\begin{definition}[Log-Barrier for PSD Cone] \label{defn:log-barrier-psd}
For the PSD cone $K = \cS_n := \{X \in \R^{n\times n} : X \succeq 0\}$, we define the log barrier as $\phi(X) := -\log ( \det ( X ) )$.
\end{definition}

\begin{lemma}[{\cite{nn94,hjstz22}}] \label{lem:log-barrier-psd}
Log-barrier for the $n\times n$ PSD cone is $n$-self-concordant.
\end{lemma}

\subsection{Treewidth}\label{sec:preli:treewidth}

\begin{definition}[Tree decomposition and treewidth] \label{defn:treewidth}
Let $G=(V,E)$ be a graph, a tree decomposition of $G$ is a tree $T$ with $b$ vertices, and $b$ sets $J_1,\ldots, J_b \subseteq V$ (called bags), satisfying the following properties:
\begin{itemize}
    \item For every edge $(u,v) \in E$, there exists $j\in [b]$ such that $u,v\in J_j$;
    \item For every vertex $v\in V$, $\{j\in [b]: v\in J_j\}$ is a non-empty subtree of $T$.
\end{itemize}

The treewidth of $G$ is defined as the minimum value of $\max\{|J_j| : j\in [b]\}-1$ over all tree decompositions.
\end{definition}

\subsection{Sketching Matrices}\label{sec:preli:sketching}

\begin{lemma}[Johnson–Lindenstrauss transform, \cite{jl84}]  
\label{lem:jl_matrix}
    Let $\epsilon \in (0, 1)$ be the precision parameter.  Let $\delta \in (0, 1)$ be the failure probability. Let $A \in \R^{m \times n}$ be a real matrix.
    Let $r = \epsilon^{-2} \log(mn/\delta)$.
    For $R \in \R^{r\times n}$ whose entries are i.i.d~$\cN(0, \frac 1r)$, the following holds with probability at least $1-\delta$:
    \begin{align*}
        (1 - \epsilon)\|a_i\|_2 \le \| R  a_i\|_2 \le (1 + \epsilon)\|a_i\|_2, ~\forall i \in [m],
    \end{align*}
    where for a matrix $A$, $a_i^\top$ denotes the $i$-th row of matrix $A \in \R^{m \times n}$.
\end{lemma}

\subsection{Sparse Cholesky Decomposition}
In this section we state a few basic results on sparse Cholesky decomposition.
The following definition essentially come from \cite{schreiber1982new}.

\begin{definition}[Elimination tree] \label{defn:elim_tree}
Let $G$ be an undirected graph on $n$ vertices. An elimination tree $\cT$ is a rooted tree on $V(G)$ together with an ordering $\pi$ of $V(G)$ such that for any vertex $v$, its parent is the smallest (under $\pi$) element $u$ such that there exists a path $P$ from $v$ to $u$, such that $\pi(w) \le \pi(v)$ for all $w\in P-u$.
\end{definition}
\begin{lemma}[\cite{schreiber1982new}] \label{lem:elim_tree_imply_cholesky}
    Let $M$ be a PSD matrix and $\cT$ be an elimination tree of the adjacency graph of $M$ (i.e., $(i,j)\in E(G)$ iff $M_{i,j} \ne 0$) together with an elimination ordering $\pi$. Let $P$ be the permutation matrix $P_{i,v} = \mathbbm\{v=\pi(i)\}$.
    Then the Cholesky factor $L$ of $PMP^\top$ (i.e., $PMP^\top = LL^\top$) satisfies $L_{i,j}\ne 0$ only if $\pi(i)$ is an ancestor of $\pi(j)$.
\end{lemma}

The following lemma is useful in proving elimination trees.
\begin{lemma} \label{lem:prove_elim_tree}
Let $G$ be an undirected graph on $n$ vertices. Let $\cT$ be a rooted tree with $n$ vertices, satisfying the property that:
for any path $u=v_1,\cdots,v_k=v$ with $(v_i,v_{i+1})\in E(G)$ for all $i\in [k-1]$, then there exists $i\in [k]$ such that $v_i$ is an ancestor of both $u$ and $v$.
Then $\cT$ together with any post-order traversal of $\cT$ is an elimination tree.
\end{lemma}
\begin{proof}
    We prove that $\cT$ satisfies Definition~\ref{defn:elim_tree}.
    Let $v\in V(G)$ and $P$ be any path from $v$ to a vertex outside the subtree rooted at $v$. By assumption, there exists a vertex $u\in P$ which is an ancestor of $v$.
    Let $w$ be the parent of $v$. Then $\pi(u) \ge \pi(w) > \pi(v)$.
    Therefore $w$ is the smallest element reachable from $v$ using only elements before $v$.
\end{proof}
When any post-order traversal of $\cT$ works, we omit the choice of $\pi$ and say $\cT$ is an elimination tree.

\begin{lemma}[\cite{george1994computer}]\label{lem:known_cholesky_TL}
Under the setting of Lemma~\ref{lem:elim_tree_imply_cholesky}, if depth of $\cT$ is at most $\tau$, then the Cholesky decomposition $L$ can be computed in $O(n \tau^2)$ time.
\end{lemma}

\begin{lemma}[\cite{dh99}]\label{lem:known_cholesky_TDeltaLmax}
Under the setting of Lemma~\ref{lem:known_cholesky_TL}, suppose we are already given the Cholesky decomposition $L$.
Let $w\in \R^n$ be a vector such that $M+ww^\top$ has the same adjacency graph as $M$.
Then we can compute $\Delta_L\in \R^{n\times n}$ such that $L+\Delta_L$ is the Cholesky factor of $M+ww^\top$ in $O(\tau^2)$ time.
\end{lemma}

%% file: framework.tex
\section{General Treewidth Program Solver Framework} \label{sec:framework}

In this section we establish a general framework for solving treewidth LP/SDP and related problems.
Our framework takes block size into consideration and do not assume various block size-related parameters are constant.
Our framework takes in certain subroutine running time as parameters. This allows us to see which subroutines are running time bottlenecks.
Our SDP and LP results follow from the general framework with minimal problem-specific results in addition.

We briefly describe the outline of this section.
\begin{itemize}
    \item In Section~\ref{sec:framework:definition}, we present the definitions and backgrounds for this Section~\ref{sec:framework}.
    \item In Section~\ref{sec:framework:cpm_ds}, we present the main data structure which is central path maintenance data structure.
    \item In Section~\ref{sec:framework:block_elim_tree}, we present several computation-related lemmas which will be heavily used in our tasks.
    \item In Section~\ref{sec:framework:being_used_in_cpm}, we present several data structures that are being used in central path maintenance, including \textsc{ExactDS} (Section~\ref{sec:framework:exact_ds}), \textsc{ApproxDS} (Section~\ref{sec:framework:approx_ds}), \textsc{BatchSketch} (Section~\ref{sec:framework:batch_sketch}) and \textsc{BlockVectorSketch}(Section~\ref{sec:framework:block_vector_sketch}).
    \item In Section~\ref{sec:framework:cpm_analysis}, we prove correctness and running time of the central path maintenance data structure.
    \item In Section~\ref{sec:framework:main}, we prove the main result (Theorem~\ref{thm:algo_general}).
\end{itemize}

\subsection{Main Statement}\label{sec:framework:definition}

We consider programs of the following form.
\begin{align}\label{eqn:lp_general}
    \min \qquad & c^\top x\\
    \text{s.t.} \qquad & A_i x = b_i \qquad \forall i\in [m] \nonumber\\
    & x_i\in \cK_i \qquad \forall i\in [n] \nonumber
\end{align}
where
\begin{itemize}
    \item $\cK_i \subset \R^{n_i}$ is a convex set .
    \item $x = (x_1,\ldots, x_n) \in \R^{n_\lp}$, where $n_\lp := \sum_{i\in [n]} n_i$ .
    \item $c = (c_1,\ldots, c_n) \in \R^{n_\lp}$, where $c_i \in \R^{n_i}$ .
    \item $A\in \R^{{m_\lp}\times n_\lp}$, $b\in \R^{m_\lp}$ .
\end{itemize}

\begin{definition}[Treewidth of a general linear program] \label{def:general_lp_treewidth}
Treewidth $\tau_\lp$ of Program~\eqref{eqn:lp_general} is defined as the treewidth (Definition~\ref{defn:treewidth}) of its the LP dual graph, where the LP dual graph is defined as follows:

The vertex set is $[m_\lp]$. There is an edge between $i,j\in [m_\lp]$ if and only if there exists $k\in [n]$ such that $A_{i,k}\ne 0$, $A_{j,k}\ne 0$.

\end{definition}

\begin{definition}[Parameters for General Treewidth Program] \label{defn:lp_general_param}
We make the following assumptions on Program~\eqref{eqn:lp_general} and define relevant parameters.
These parameters will have an impact to the final running time (Theorem~\ref{thm:algo_general}).
\begin{itemize}
    \item Assume that for each $\cK_i$, we have a $\nu_i$-self-concordant barrier function $\phi_i$.
    Let $\nu_{\max} := \max_{i\in [n]} \nu_i$.
    \item Assume that $\phi_i$, $\nabla \phi_i$ and $\nabla^2 \phi_i$ can be computed in $T_{H,i}$ time. Define $T_{H,\max} := \max_{i\in [n]} T_{H,i}$ and $T_H := \sum_{i\in [n]} T_{H,i}$.
    \item Assume that we are given a tree decomposition of the LP dual graph, such that every bag has total size at most $\tau_\lp$.
    \item Assume that we are given a block elimination tree (Definition~\ref{defn:block-elimination-tree}) with $m$ vertices and maximum block-depth $\eta$.
    The block elimination tree partitions the constraints into $m$ blocks, each of size $m_i$ ($i\in [m]$), with $\sum_{i\in [m]} m_i = m_\lp$.
    \item Assume that it takes $T_n$ time to perform matrix multiplication of two block diagonal matrices $\in \R^{n_\lp\times n_\lp}$ with block signature $(n_1,\ldots,n_n)$. We have $T_n = \sum_{i\in [n]} \Tmat(n_i)$.
    \item Assume that it takes $T_m$ time to perform matrix multiplication of two block diagonal matrices $\in \R^{m_\lp\times m_\lp}$ with block signature $(m_1,\ldots, m_m)$. We have $T_m = \sum_{i\in [m]} \Tmat(m_i)$.
    \item Assume that it takes $T_L$ time to compute Cholesky decomposition $A H A^\top = LL^\top$.
    \item Assume that it takes $T_{\Delta_L,\max}$ time to update Cholesky decomposition, i.e., given $A H A^\top = LL^\top$ and $\Delta_{H}$ supported on a single diagonal block, computing $\Delta_{L}$ such that $A (H+\Delta_{H}) A^\top = (L+\Delta_{L}) (L+\Delta_{L})^\top$.
    \item Assume that it takes $T_Z$ time to compute $L^{-1} v_i$ for all $i\in [m]$, where $v_i$ is supported on $\bigcup_{j\in \cP(i)} B_j$, where $\cP(i)$ is the set of ancestors of vertex $i$. This is used in Lemma~\ref{lem:block_balanced_sketch}.
    \item Assume that $R$ is the diameter of $\cK_i$. Assume that we are given an initial point $x$ such that $B(x,r) \subseteq \cK$.
    Assume that there exists $x$ such that $Az=b$ and $B(x,r) \subseteq \cK$.
\end{itemize}
\end{definition}

We are ready to state our main theorem.
\begin{theorem}[General Treewidth Program Solver] \label{thm:algo_general}
Under assumptions in Definition~\ref{defn:lp_general_param}, given any $0<\epsilon \le \frac 12$, we can find $x\in \cK$ with $Ax=b$ such that
\begin{align*}
c^\top x \le \min_{Ax=b, x\in \cK} c^\top x + \epsilon \|c\|_2 R
\end{align*}
in expected time
\begin{align*}
&\wt{O} ( n^{0.5} \nu_{\max}^{0.5} \cdot (T_H + T_n + T_L + T_Z + \nnz(A) + \eta m_\lp m_{\max})^{0.5} \\
&\cdot (\Tmat(n_{\max}) + T_{\Delta_L,\max} + T_{H,\max} + \eta^2 m_{\max}^2)^{0.5} \log(R/(r\epsilon))).
\end{align*}
where $\wt O$ hides $\polylog(n_\lp)$ terms.
\end{theorem}
We defer the proofs into Section~\ref{sec:framework:main}.

\begin{table}[!ht]
\centering
\begin{tabular}{|l|l|l|l|l|}\hline
{\bf Notation} & {\bf Meaning}   \\ \hline
$n_\lp$ & Total variable dimension   \\ \hline
$n$ & Number of variable blocks   \\ \hline
$n_i$ & Dimension of $i$-th variable block   \\ \hline
$n_{\max}$ & Max dimension of a block     \\ \hline
$m_\lp$ & Number of constraints   \\ \hline
$m$ & Number of constraint blocks   \\ \hline
$m_i$ & Size of $i$-th constraint block   \\ \hline
$m_{\max}$ & Max size of constraint block  \\ \hline
$\tau_\lp$ & Max size of a bag  \\ \hline
$\eta$ & Block height of elimination tree  \\ \hline
$\nu_i$ & Self-concordance of $i$-th block   \\ \hline
$\nu_{\max}$ & Max self-concordance of a block \\ \hline
$\nnz(A)$ & number of non-zero entries of $A$  \\ \hline
$T_n$ & Matrix mult time for $n_\lp\times n_\lp$ block-diag matrix   \\ \hline
$T_m$ & Matrix mult time for $m_\lp\times m_\lp$ block-diag matrix   \\ \hline
$T_H$ & time to compute Hessian   \\ \hline
$T_{H,\max}$ & time to compute Hessian of one block   \\ \hline
$T_L$ & time to compute Cholesky   \\ \hline
$T_{\Delta_L,\max}$ & time to update Cholesky  \\ \hline
\end{tabular}
\caption{
Parameters needed to call Theorem~\ref{thm:algo_general}.
Their values are summarized in Table~\ref{tab:args_general_with_choices}
}
\label{tab:args_general}
\end{table}

\begin{table}[!ht]
\centering
\begin{tabular}{|l|l|l|l|l|}\hline
{\bf Notation} & {\bf SDP First} & {\bf SDP Second} & {\bf Decomposable SDP} & {\bf LP} \\ \hline
$n_\lp$  & $n_\sdp \tau_\sdp^2$ & $n_\sdp \tau_\sdp^2$ & $n_\sdp\tau_\sdp^2$ & $n$ \\ \hline
$n$  & $n_\sdp$ & $n_\sdp$ & $n_\sdp$ & $n$ \\ \hline
$n_i$  & $\tau_\sdp^2$ & $\tau_\sdp^2$ & $\tau_\sdp^2$ & $1$ \\ \hline
$n_{\max}$   & $\tau_\sdp^2$ &$\tau_\sdp^2$ & $\tau_\sdp^2$ & $1$  \\ \hline
$m_\lp$   & $m_{\sdp} + n_\sdp \tau_\sdp^2$ & $m_{\sdp} + n_\sdp \tau_\sdp^2$ & $m_\sdp+n_\sdp (\tau_\sdp^2+\gamma_{\max})$ & $m$ \\ \hline
$m$   & $m_\lp$ & $n_{\sdp}$ & $n_\sdp$ & $m$ \\ \hline
$m_i$   & $1$ & $\tau_\lp$ & $\tau_\lp$ & $1$ \\ \hline
$m_{\max}$   & $1$ & $\tau_\lp$ & $\tau_\lp$ & $1$ \\ \hline
$\tau_\lp$   & $\tau_\sdp^2$ & $\tau_\sdp^2$ & $\tau_\sdp^2 + \gamma_{\max}$ & $\tau$ \\ \hline
$\eta$   & $\tau_\lp$ & $1$ & $1$ & $\tau$ \\ \hline
$\nu_i$   & $\tau_\sdp$ & $\tau_\sdp$ & $\tau_\sdp$ & $1$ \\ \hline
$\nu_{\max}$   & $\tau_\sdp$ & $\tau_\sdp$ & $\tau_\sdp$ & $1$ \\ \hline
$\nnz(A)$   & $n_\sdp \tau_\sdp^4$ & $n_\sdp \tau_\sdp^4$ & $n_\sdp \tau_\sdp^2 (\tau_\sdp^2+\gamma_{\max})$ & $n\tau$ \\ \hline
$T_n$  & $n_{\sdp} \tau_{\sdp}^{2\omega}$ & $n_{\sdp}\tau_{\sdp}^{2\omega}$ & $n_{\sdp}\tau_{\sdp}^{2\omega}$ & $n$ \\ \hline
$T_m$  & $m_\lp$ & $n_{\sdp}\tau_{\sdp}^{2\omega}$ & $n_{\sdp}(\tau_{\sdp}+\gamma_{\max})^{2\omega}$ & $n$ \\ \hline
$T_H$  & $n_{\sdp}\tau_{\sdp}^{2\omega}$ & $n_{\sdp}\tau_{\sdp}^{2\omega}$ & $n_{\sdp}\tau_{\sdp}^{2\omega}$ & $n$ \\ \hline
$T_{H,\max}$  & $n_{\max}^\omega$ & $n_{\max}^\omega$ & $n_{\max}^\omega$ & $1$ \\ \hline
$T_L$ & $n_\lp \tau_\lp^2$ & $n_\sdp \tau_\lp^\omega$ & $n_\sdp \tau_\lp^\omega$ & $n \tau^{\omega-1}$ \\ \hline
$T_{\Delta_L,\max}$ & $\tau_\lp^3$ & $\tau_\lp^\omega$ & $\tau_\lp^\omega$ & $\tau^2$ \\ \hline
\end{tabular}
\caption{
Values of parameters in different results. Definition of these parameters are provided in Table~\ref{tab:args_general}.
Validity of the choices are proved in Section~\ref{sec:sdp_first}, Section~\ref{sec:sdp_second}, Section~\ref{sec:sdp_general}, and Section~\ref{sec:lp}.
$\wt O$ notation is omitted.
}
\label{tab:args_general_with_choices}
\end{table}

\subsection{Algorithm structure and \texorpdfstring{$\textsc{CentralPathMaintenance}$}{}}\label{sec:framework:cpm_ds}
Our algorithm is a robust Interior Point Method (robust IPM).
Details of the robust IPM will be given in Section~\ref{sec:robust_ipm}.

In robust IPM, we maintain the primal-dual solution pair $(x,s)\in \R^{n_\lp} \times \R^{n_\lp}$ on the robust central path. In addition, we maintain approximations $(\ov x, \ov s) \in \R^{n_\lp} \times \R^{n_\lp}$ which change sparsely.
In each iteration, we implicitly perform update
\begin{align*}
    x &\gets x + (H_{\ov x}^{-1} - H_{\ov x}^{-1} A^\top (A H_{\ov x}^{-1} A^\top)^{-1} A H_{\ov x}^{-1}) \delta_\mu(\ov x, \ov s, \ov t), \\
    s & \gets s + t A^\top (A H_{\ov x}^{-1} A^\top)^{-1} A H_{\ov x}^{-1} \delta_\mu(\ov x, \ov s, \ov t)
\end{align*}
and explicitly maintain $(\ov x, \ov s)$ such that they remain close to $(x, s)$.

This task is handled by the \textsc{CentralPathMaintenance} data structure, which is the main data structure.
The robust IPM algorithm (Algorithm~\ref{alg:robust_ipm_main},\ref{alg:robust_ipm_centering_data_structure}) directly calls it in every iteration.
This data structure is a generalization of \textsc{CentralPathMaintenance} data structure in \cite{dly21}.

The \textsc{CentralPathMaintenance} data structure (Algorithm~\ref{alg:cpm}) has two main sub data structures, \textsc{ExactDS} (Algorithm~\ref{alg:exact_ds_part_1} and Algorithm~\ref{alg:exact_ds_part_2}) and \textsc{ApproxDS}.
\textsc{ExactDS} is used to to maintain $(x,s)$, and \textsc{ApproxDS} is used to monitor changes in $(x,s)$, and update $(\ov x, \ov s)$ when necessary.

\begin{theorem} \label{thm:central_path_maintenance_general}
Data structure \textsc{CentralPathMaintenance} (Algorithm~\ref{alg:cpm}) implicitly maintains the central path primal-dual solution pair $(x,s) \in \R^{n_{\lp}} \times \R^{n_{\lp}}$
and explicitly maintains its approximation $(\ov{x}, \ov{s}) \in \R^{n_{\lp}} \times \R^{n_{\lp}}$ using the following functions:
\begin{itemize}
    \item \textsc{Initialize}$(x \in \R^{n_{\lp}}, s \in \R^{n_{\lp}}, t_0, q)$: (See Lemma~\ref{lem:init_time_general}) Initializes the data structure with initial primal-dual solution pair $(x,s) \in \R^{n_{\lp}} \times \R^{ n_{\lp}}$, initial central path timestamp $t_0$, and a run-time tuning parameter $q$ in 
    \begin{align*} 
    \wt O(T_H + T_n + T_L + T_Z + \nnz(A) + \eta m_\lp m_{\max})
    \end{align*}
    time.
    \item \textsc{MultiplyAndMove}$(t)$: (See Lemma~\ref{lem:multiply_and_move_time_general}) It implicitly maintains
    \begin{align*}
        x &\leftarrow x + H_{\ov{x}}^{-1/2}(I-P_{\ov{x}}) H_{\ov{x}}^{-1/2} \delta_\mu(\ov{x}, \ov{s}, \ov{t})\\
        s &\leftarrow s + t \cdot  H_{\ov{x}}^{1/2} P_{\ov{x}} H_{\ov{x}}^{-1/2} \delta_\mu(\ov{x}, \ov{s}, \ov{t})
    \end{align*}
    where $H_{\ov{x}} := \nabla^2 \phi(\ov{x}) \in \R^{ n_{\lp} \times n_{\lp} }$, $P_{\ov{x}} := H_{\ov{x}}^{-1/2} A^\top (A H_{\ov{x}}^{-1} A^\top)^{-1} A H_{\ov{x}}^{-1/2} \in \R^{n_{\lp} \times n_{\lp}}$,
    and $\ov{t}$ is some earlier timestamp satisfying $|t-\ov{t}| \le \epsilon_t \cdot \ov{t}$.
    
    It also explicitly maintains $(\ov{x}, \ov{s}) \in \R^{n_{\lp} \times n_{\lp}}$ such that $\|\ov{x}_i-x_i\|_{\ov{x}_i} \le \ov{\epsilon}$ and $\|\ov{s}_i-s_i\|^*_{\ov{x}_i} \le t \ov{\epsilon} w_i$ for all $i\in [n]$ with probability at least $0.9$.
    
    Assuming the function is called at most $N$ times and $t$ is monotonically decreasing from $t_{\max}$ to $t_{\min}$, the total running time is 
    \begin{align*} 
    &\wt{O} ( n^{0.5} \nu_{\max}^{0.5} \cdot (T_H + T_n + T_L + T_Z + \nnz(A) + \eta m_\lp m_{\max})^{0.5} \\
&\cdot (\Tmat(n_{\max}) + T_{\Delta_L,\max} + T_{H,\max} + \eta^2 m_{\max}^2)^{0.5} \log(t_{\max}/t_{\min})).
    \end{align*}
    
    \item \textsc{Output}: (Lemma~\ref{lem:output_time_general}) It computes $(x,s) \in \R^{n_{\lp} \times n_{\lp}}$ exactly and outputs them in 
    \begin{align*} 
        O(\nnz(A) + T_H + \eta T_m)
    \end{align*} 
    time.
\end{itemize}
\end{theorem}

Correctness and running time analysis of \textsc{CentralPathMaintenance} is deferred to Section~\ref{sec:framework:cpm_analysis} and Section~\ref{sec:framework:main}, after we establish properties of the sub data structures.

\begin{algorithm}[!ht]\caption{Central Path Maintenance. 
This is used in Algorithm~\ref{alg:robust_ipm_centering_data_structure}.
}\label{alg:cpm}
\begin{algorithmic}[1]
\State {\bf data structure} \textsc{CentralPathMaintenance} \Comment{Theorem~\ref{thm:central_path_maintenance_general}}
\State {\bf private : member}
\State \hspace{4mm} \textsc{ExactDS} $\mathsf{exact}$ \Comment{Algorithm~\ref{alg:exact_ds_part_1}, \ref{alg:exact_ds_part_2}}
\State \hspace{4mm}  \textsc{ApproxDS} $\mathsf{approx}$ \Comment{ Algorithm~\ref{alg:approxds_1}}
\State \hspace{4mm} $\ell\in \bN$
\State {\bf end members}
\Procedure{\textsc{Initialize}}{$x, s\in \R^{n_\lp}, t\in \R_+, \ov \epsilon \in (0, 1)$} \Comment{Lemma~\ref{lem:init_time_general}}
    \State $\mathsf{exact}.\textsc{Initialize}(x, s, x, s, t)$ \Comment{Algorithm~\ref{alg:exact_ds_part_1}}
    \State $\ell \leftarrow 0$
    \State $w \gets \nu_{\max}$, $N\gets \sqrt{n \nu_{\max} w}$
    \State $q \gets n^{0.5} \nu_{\max}^{0.5} (T_H+T_L+\eta T_m+T_Z)^{0.5} (\Tmat(n_{\max})+T_{\Delta_L,\max}+T_{H,\max}+\eta^2 m_{\max}^2)^{-0.5}$
    \State $\epsilon_{\apx,x} \leftarrow \ov \epsilon, \zeta_x \leftarrow 2 \alpha w^{-1/2}, \delta_\apx \leftarrow \frac 1N$
    \State $\epsilon_{\apx,s} \leftarrow \ov \epsilon \cdot \ov t \cdot w, \zeta_s \leftarrow 2\alpha \ov t w^{1/2}$
    \State $\mathsf{approx}.\textsc{Initialize}(x, s, h, \epsilon_x, \epsilon_s, H_{\ov x}^{1/2} \wh x, H_{\ov x}^{-1/2} \wh s, c_x, \beta_x, \beta_s, q, \& \mathsf{exact}, \epsilon_{\apx,x}, \epsilon_{\apx,s}, \zeta_x, \zeta_s, \delta_{\apx})$
    \State
    \Comment{Parameters from $x$ to $\beta_s$ come from $\mathsf{exact}$. $\& \mathsf{exact}$ is pointer to $\mathsf{exact}$}
\EndProcedure
\Procedure{\textsc{MultiplyAndMove}}{$t\in \R_+$} \Comment{Lemma~\ref{lem:multiply_and_move_time_general}, \ref{lem:cpm_correct_general}}
    \State $\ell\gets \ell + 1$
    \If{$|\ov t - t| > \ov t \cdot \epsilon_t$ or $\ell > q$}
        \State $x, s \gets \mathsf{exact}.\textsc{Output}()$ \Comment{Algorithm~\ref{alg:exact_ds_part_1}}
        \State \textsc{Initialize}$(x,s,t,\ov \epsilon)$
    \EndIf
    \State $\beta_x, \beta_s \gets \mathsf{exact}.\textsc{Move}()$ \Comment{Algorithm~\ref{alg:exact_ds_part_1}}
    \State $\delta_{\ov x}, \delta_{\ov s } \gets \mathsf{approx}.\textsc{MoveAndQuery}(\beta_x, \beta_s)$ \Comment{Algorithm~\ref{alg:approxds_1}}
    \State $\delta_{h}, \delta_{\epsilon_x}, \delta_{\epsilon_s}, \delta_{H_{\ov x}^{1/2} \wh x}, \delta_{H_{\ov x}^{-1/2} \wh s}, \delta_{c_x}\gets \mathsf{exact}.\textsc{Update}(\delta_{\ov x}, \delta_{\ov s})$ \Comment{Algorithm~\ref{alg:exact_ds_part_2}}
    \State $\mathsf{approx}.\textsc{Update}(\delta_{\ov x}, \delta_{h}, \delta_{\epsilon_x}, \delta_{\epsilon_s}, \delta_{H_{\ov x}^{1/2} \wh x}, \delta_{H_{\ov x}^{-1/2} \wh s}, \delta_{c_x})$
\EndProcedure
\Procedure{\textsc{Output}}{$ $} \Comment{Lemma~\ref{lem:output_time_general}}
    \State \Return $\mathsf{exact}.\textsc{Output}()$ \Comment{Algorithm~\ref{alg:exact_ds_part_1}}
\EndProcedure
\State {\bf end data structure}
\end{algorithmic}
\end{algorithm}

\subsection{Block Elimination Tree and Computation-Related Lemmas}\label{sec:framework:block_elim_tree}
Our algorithm is based on efficient computations involving Cholesky factorization and the block elimination tree.
In this section we introduce the block elimination tree and prove a few useful lemmas on running time of relevant computations.
They will be used repeatedly in the running time analysis of our algorithm.

\begin{definition}[Block elimination tree] \label{defn:block-elimination-tree}
Let $G$ be a graph with $m_\lp$ vertices. A block elimination tree is a rooted tree with $m$ vertices, satisfying the following properties.
\begin{itemize}
    \item The vertices correspond to a partition $[m_\lp] = B_1 \cup\cdots \cup B_m$, where $|B_i| = m_i$.
    \item We can compute a permutation $P$ such that for any PSD matrix $M\in \bS^{m_\lp}$ supported on $G$ (i.e., $M_{\text{coord } i, \text{coord }j}\ne 0$ only if $i=j$ or $(i,j)\in E(G)$), in the Cholesky factorization $P M P^\top = LL^\top$,
    the non-zero entries of $L e_{B_i}$ are subsets of $\bigcup_{j\in \cP(i)} B_j$, where $\cP(i)$ is the set of ancestors of vertex $i$.
    In other words, the (block) non-zero pattern of the $i$-th (block) column is the set of vertices on the path from $i$ to the root.
\end{itemize}
\end{definition}

Under the setting of Program~\ref{eqn:lp_general}, we always take $G$ to be the LP dual graph, and $M$ to be $A H_{\ov x}^{-1} A^\top$.

The following lemma is the block verion of Lemma~\ref{lem:prove_elim_tree}.
\begin{lemma}\label{lem:prove_block_elim_tree}
Let $G$ be an undirected graph on $m_\lp$ vertices. Let $\cT$ be a tree whose vertices corresponds to a partition of $[m_\lp]$.
Suppose $\cT$ satisfying the following property:
for any path $u=v_1,\ldots, v_k=v$ such that for all $i\in [k-1]$, there exists an edge connecting some pair of elements in $v_i$ and $v_{i+1}$ (i.e., there exists $a\in B_i$, $b\in B_j$ such that $(a,b)\in E(G)$), there exists $i\in [k]$ such that $v_i$ is an ancestor of both $u$ and $v$.
Then $\cT$ is a valid block elimination tree of $G$.
\end{lemma}
\begin{proof}
We define another tree $\cT'$ with block pattern $(1,\ldots, 1)$.
For every vertex $v$ in $\cT$, we replace it with a path (of an arbitrary ordering of elements in $v$). For all edges in $\cT$, we connect top element of the child and the bottom element of the parent.

In this way, $\cT'$ satisfies the property that for any path $u=v_1,\cdots,v_k=v$ with $(v_i,v_{i+1})\in E(G)$ for all $i\in [m_k-1]$, then there exists $i\in [k]$ such that $v_i$ is an ancestor of both $u$ and $v$.
By Lemma~\ref{lem:prove_elim_tree}, $\cT'$ is a valid elimination tree.
This implies that $\cT$ is a valid block elimination tree.
\end{proof}

As in Definition~\ref{defn:lp_general_param}, we assume that we are given a block elimination tree with (block) depth $\eta$. We can without loss of generality assume that the blocks are labeled in postorder, i.e., for any $i$ and $j\in \cP(i)$, we have $i<j$.

Given a block elimination tree, we can efficiently perform many computations related to the Cholesky decomposition, as shown in the following lemma.
\begin{lemma} \label{lem:mat_vec_mult_time}
Assume that we are given a block elimination tree with block structure $(m_1,\ldots, m_m)$ and block depth $\eta$.
Assume that we are given the Cholesky factorization $A H A^\top = LL^{\top}$ together with inverses of the diagonal blocks of $L$, i.e., $L_{i,i}^{-1}$ for all $i\in [m]$.

Then we have the following running time for matrix-vector multiplications.
\begin{enumerate}[label=(\roman*)]
    \item For $v\in \R^{m_\lp}$, computing $L v$, $L^\top v$, $L^{-1} v$, $L^{-\top} v$ takes $O(\eta m_\lp m_{\max})$ time.
    \label{item:lem_mat_vec_mult_time_Linvv}
    \item For $v\in \R^{m_\lp}$, computing $L v$ takes $O(\|v\|_{2,0} \eta m_{\max}^2)$ time.
    \label{item:lem_mat_vec_mult_time_Lv_sparse}
    \item For $v\in \R^{m_\lp}$, computing $L^{-1} v$ takes $O(\|L^{-1}v\|_{2,0} \eta m_{\max}^2)$ time.
    \label{item:lem_mat_vec_mult_time_Linvv_sparse}
    \item For $v\in \R^{m_\lp}$, if $v$ is supported on a path in the block elimination tree, then computing $L^{-1}v$ takes $O(\eta^2 m_{\max}^2)$ time.
    \label{item:lem_mat_vec_mult_time_Linvv_path}
    \item For $v\in \R^{n_\lp}$, given $H$, computing $H^{-1/2} v$, $H^{-1} v$, $H^{1/2} v$, $H v$ takes $O(T_n)$ time.
    \label{item:lem_mat_vec_mult_time_Hinvv}
    \item For $v\in \R^{n_\lp}$, computing $A v$ takes $O(\nnz(A))$ time.
    \label{item:lem_mat_vec_mult_time_Av}
    \item For $v\in \R^{m_\lp}$, computing $A^\top v$ takes $O(\nnz(A))$ time.
    \label{item:lem_mat_vec_mult_time_ATv}
    \item For $v\in \R^{n_\lp}$, computing $\cW^\top v$ takes $O(T_n + \eta m_\lp m_{\max} + \nnz(A))$ time, where $\cW = L^{-1} A H^{-1/2}$.
    \label{item:lem_mat_vec_mult_time_WTv}
\end{enumerate}
\begin{proof}
\begin{enumerate}[label=(\roman*)]
    \item Because every (block) column of $L$ is $\eta$-block sparse, every (coordinate) column of $L$ is $\eta m_{\max}$-coordinate sparse. So computing $Lv$ takes $\nnz(L) \le \eta m_\lp m_{\max}$ time.
    
    For $L^{-1} v$, let us consider Algorithm~\ref{alg:solve_Linvv}. Recall that we assume $L_{i,i}^{-1}$ are given. So the algorithm takes $\nnz(L) + \sum_{i\in [m]} m_i^2 = \eta m_\lp m_{\max}$ time.
\begin{algorithm}[!ht]
\caption{Solving block triangular systems $L x = v$}
\label{alg:solve_Linvv}
\begin{algorithmic}[1]
\Procedure{SolveBlock}{$ $}
\State $x\leftarrow 0 \in \R^{m_\lp}$
\For{increasing $j$ with $v_j\ne 0$}
\State $x_j \leftarrow L_{j,j}^{-1} v_j$
\State $v \leftarrow v - L_{*,j} x_j$
\EndFor
\EndProcedure
\end{algorithmic}
\end{algorithm}

    Statements about $L^{\top} v$ and $L^{-\top} v$ follow from the Transposition principle \cite{bordewijk1957inter}.
    \item If $v$ is supported on a single block $i\in [m]$, then computing $L v$ takes $\sum_{j\in \cP(i)} m_j^2 = O(\eta m_{\max}^2)$ time by block column sparsity of $L$.
    So computing $Lv$ for general $v$ takes $O(\|v\|_{2,0} \eta m_{\max}^2)$ time.
    \item Consider Algorithm~\ref{alg:solve_Linvv}. By block column sparsity pattern of $L$, for every block $j\in [m]$ with $x_j \ne 0$, it takes $O(\eta m_{\max}^2)$ time to compute $x_j$ and update $v \gets v-L_{*,j} x_j$.
    So overall running time is $O(\|L^{-1} v\|_{2,0} \eta m_{\max}^2)$.
    \item Corollary of \ref{item:lem_mat_vec_mult_time_Linvv_sparse}.
    \item It takes $\sum_{i\in [n]} \Tmat(n_i) = T_n$ time to compute an eigendecomposition of $H$. Afterwards, it takes $\sum_{i\in [n]} \Tmat(n_i) = O(T_n)$ time to compute the product.
    \item It is obvious to see that computing $(Av)_i$ takes $\nnz(A_{i,*})$ time. Summation over all the indices $i$, we know the running time is $\nnz(A)$.
    \item It is easy to see that computing $(A^\top v)_j$ takes $\nnz(A_{*,j})$ time. Summation over all the indices $j$, we know the running time is $\nnz(A)$.
    \item $\cW^\top v = H^{-1/2} A^\top L^{-\top} v$.
    So the result follows from combining \ref{item:lem_mat_vec_mult_time_Linvv}\ref{item:lem_mat_vec_mult_time_ATv}\ref{item:lem_mat_vec_mult_time_Hinvv}.
\end{enumerate}
\end{proof}
\end{lemma}
\begin{lemma} \label{lem:mat_vec_mult_coord_time}
Assume that we are given a block elimination tree with block structure $(m_1,\ldots, m_m)$ and block depth $\eta$.
Assume that we are given the Cholesky factorization $A H A^\top = LL^{\top}$ together with inverses of the diagonal blocks of $L$, i.e., $L_{i,i}^{-1}$ for all $i\in [m]$.

Then we have the following running time for matrix-vector multiplications, when we only need result for a subset of coordinates.
\begin{enumerate}[label=(\roman*)]
    \item Let $S$ be a path in the block elimination tree whose one endpoint is the root. For $v\in \R^{m_\lp}$, computing $(L^{-\top}v)_S$ takes $O(\eta^2 m_{\max}^2)$ time.
    \label{item:mat_vec_mult_coord_time_LinvtvS}
    \item For $v\in \R^{m_\lp}$, for $i\in [n]$, computing $(\cW^\top v)_i$ takes $O(n_{\max}^2 + \eta^2 m_{\max}^2)$ time, where $\cW = L^{-1} A H^{-1/2}$.
    \label{item:mat_vec_mult_coord_time_Wtvi}
\end{enumerate}
\end{lemma}
\begin{proof}
\begin{enumerate}[label=(\roman*)]
    \item We have $(L^{-\top} v)_S = e_S^\top L^{-\top} v = (v^\top L^{-1} e_S)^\top$.
    By column sparsity pattern of $L$, $L^{-1} e_S$ is supported on $S$.
    So $(L^{-\top}v)_S$ depends only on $v_S$. So we only need to compute $L_{S,S}^{-\top} v_S$, which takes $O(\eta^2 m_{\max}^2)$ time By Lemma~\ref{lem:mat_vec_mult_time}\ref{item:lem_mat_vec_mult_time_Linvv}.
    \item 
    By definition of the block elimination tree, for $j,k\in [m_\lp]$, if there is an edge $(j,k)$ in the LP dual graph, then either $j$ and $k$ are in the same block, or they are in two different blocks, one is an ancestor of the other.
    Therefore, all constraints containing $i$ lie on a single path $S$.
    So $(\cW^\top v)_i = H^{-1/2}_{i,i} A_{*,i}^\top L^{-\top} v = H^{-1/2}_{i,i} A_{S,i}^\top (L^{-\top} v)_S$.
    So we can first compute $(L^{-\top}v)_S$, which takes $O(\eta^2 m_{\max}^2)$ time by \ref{item:mat_vec_mult_coord_time_LinvtvS}.
    Afterwards, it takes $O(n_{\max}^2 + n_{\max} \eta m_{\max}) = O(n_{\max}^2 + \eta^2 m_{\max}^2)$ time to compute $H_{i,i} A^{\top}_{S,i} (L^{-\top}v)_S$.
\end{enumerate}
\end{proof}

We state a generic bound on $T_Z$ (recall Definition~\ref{defn:sdp-general-assumptions}).
\begin{lemma} \label{lem:generic_bound_TZ}
$T_Z = O(\eta^2 m m_{\max}^2)$.
\end{lemma}
\begin{proof}
Computing each $L^{-1}v_i$ takes $O(\eta^2 m_{\max}^2)$ time by Lemma~\ref{lem:mat_vec_mult_time}\ref{item:lem_mat_vec_mult_time_Linvv_path}.
So computing $m$ of them takes $O(\eta^2 m m_{\max}^2)$ time.
\end{proof}

\subsection{Data structures being used in \textsc{CentralPathMaintenance}}\label{sec:framework:being_used_in_cpm}
 In Section~\ref{sec:framework:being_used_in_cpm}, we present several data structures that are being used in central path maintenance, including:
 \begin{itemize}
     \item \textsc{ExactDS} (Section~\ref{sec:framework:exact_ds}). This data structure implicitly maintains the primal-dual solution pair $(x,s)$. This data structure is directly used by \textsc{CentralPathMaintenance}.
     \item \textsc{ApproxDS} (Section~\ref{sec:framework:approx_ds}). This data structure explicitly maintains the approximate primal-dual solution pair $(\ov x, \ov s)$.
     This data structure is directly used by \textsc{CentralPathMaintenance}.
     \item \textsc{BatchSketch} (Section~\ref{sec:framework:batch_sketch}). This data structure maintains a sketch of $x$ and $s$, using \textsc{BlockVectorSketch} and \textsc{BlockBalancedSketch}. This data structure is used by \textsc{ApproxDS}.
     \item \textsc{BlockVectorSketch}(Section~\ref{sec:framework:block_vector_sketch}).
     This data structure maintains a sketch of a vector (with block pattern) under single-point updates. This data structured is used by \textsc{BatchSketch}.
     \item \textsc{BlockBalancedSketch}(Section~\ref{sec:framework:block_balanced_sketch}). This data structure maintains a sketch of a vector of form $\cW^\top v$, under updates of $(\ov x, \ov s)$. This data structure is used by \textsc{BatchSketch}.
 \end{itemize}

\subsubsection{\textsc{ExactDS}}\label{sec:framework:exact_ds}
In this section, we present our \textsc{ExactDS} (Algorithm~\ref{alg:exact_ds_part_1} and Algorithm~\ref{alg:exact_ds_part_2}). In Theorem~\ref{thm:exactds}, we provide our theoretical statement for Algorithm~\ref{alg:exact_ds_part_1} and Algorithm~\ref{alg:exact_ds_part_2}. This is a block-based generalization of \textsc{MultiscaleRepresentation} data structure of \cite{dly21}.

\begin{definition}[Multiscale coefficients, Definition 6.4 in \cite{dly21}]
At any step the of the robust central path with approximate primal-dual solution pair $(\ov{x}, \ov{s}) \in \R^{n_{\lp}} \times \R^{n_{\lp}}$, we define
\begin{align*}
    \cW &:= L_{\ov x}^{-1} A H_{\ov x}^{-1/2},\\
    h &:= L_{\ov{x}}^{-1} A H_{\ov{x}}^{-1} \cdot \delta_\mu(\ov{x}, \ov{s}, \ov{t})
\end{align*}
where $H_{\ov{x}} = \nabla^2 \phi(\ov{x}) \in \R^{n_{\lp} \times n_{\lp}}$, and $L_{\ov{x}} \in \R^{m_{\lp} \times m_{\lp}}$ 
is the lower Cholesky factor of $A H_{\ov{x}}^{-1} A^\top \in \R^{m_{\lp} \times m_{\lp}}$.
\end{definition}

We can write the central path update using multiscale representation
\begin{align*}
x &\leftarrow x + H_{\ov{x}}^{-1} \delta_\mu(\ov{x}, \ov{s}, \ov{t}) - H_{\ov{x}}^{-1/2} \cW^\top h \\
s &\leftarrow s + t \cdot H_{\ov{x}}^{1/2} \cW^\top h
\end{align*}

\begin{theorem}[
Block-based exact data structure] \label{thm:exactds}
Data structure \textsc{ExactDS} (Algorithm~\ref{alg:exact_ds_part_1}, \ref{alg:exact_ds_part_2}) implicitly maintains the primal-dual pair $(x,s) \in \R^{n_{\lp}} \times \R^{ n_{\lp} }$, computable via the expression
\begin{align*}
x &= \wh x + H_{\ov{x}}^{-1/2} \beta_x c_x - H_{\ov{x}}^{-1/2} \cW^\top (\beta_x h + \epsilon_x) \\
s &= \wh s + H_{\ov{x}}^{1/2} \cW^\top(\beta_s h + \epsilon_s)
\end{align*}
by maintaining the variables $\ov x, \ov s\in \R^{n_\lp}$, $H_{\ov x}\in \R^{n_\lp \times n_\lp}$, $L_{\ov x} \in \R^{m_\lp\times m_\lp}$, $\wh x, \wh s, c_x\in \R^{n_\lp}$, $\epsilon_x,\epsilon_s, h\in \R^{m_\lp}$, $\beta_x, \beta_s\in \R$, $\ov \alpha\in \R$, $\ov \delta_\mu \in \R^n$.

The data structure supports the following functions:
\begin{itemize}
    \item \textsc{Initialize}$(x \in \R^{n_{\lp}}, s \in \R^{n_{\lp}}, \ov{x} \in \R^{n_{\lp}}, \ov{s} \in \R^{n_{\lp}}, \ov{t})$: Initializes the data structure in 
    \begin{align*} 
        \wt O(T_H + T_L + \nnz(A) + T_n + \eta m_\lp m_{\max})
    \end{align*}
    time, with initial value of the primal-dual pair $(x, s)$, its initial approximation $(\ov{x}, \ov{s})$, and initial approximate timestamp $\ov{t}$.
    \item \textsc{Move}$()$: Moves $(x \in \R^{n_{\lp}},s \in \R^{n_{\lp}} )$ in $O(1)$ time by updating its implicit representation.
    \item \textsc{Update}$(\delta_{\ov{x}} \in \R^{n_{\lp}}, \delta_{\ov{s}} \in \R^{n_{\lp}})$: Updates the approximation pair $(\ov{x}, \ov{s})$ to $(\ov{x}^{\new} = \ov{x} + \delta_{\ov{x}} \in \R^{ n_{\lp} }, \ov{s}^{\new} = \ov{s} + \delta_{ \ov{s}} \in \R^{n_{\lp}} )$ in \begin{align*}
    \wt O((\Tmat(n_{\max}) + T_{\Delta_L,\max} + T_{H,\max} + \eta^2 m_{\max}^2) (\|\delta_{\ov x}\|_{2,0} + \|\delta_{\ov s}\|_{2,0}))
    \end{align*}
    time, and output the changes in variables $\delta_h, \delta_{\epsilon_x}, \delta_{\epsilon_s}, \delta_{H_{\ov x}^{1/2} \wh x}, \delta_{H_{\ov x}^{-1/2} \wh s}, \delta_{\ov H_{\ov x}^{-1/2} c_x}$.
    
    Furthermore, $\delta_h, \delta_{\epsilon_x}, \delta_{\epsilon_s}$ changes in $O(\eta(\|\delta_{\ov x}\|_{2,0}+\|\delta_{\ov s}\|_{2,0}))$ blocks, and
    $\delta_{H_{\ov x}^{1/2} \wh x}, \delta_{H_{\ov x}^{-1/2} \wh s}, \delta_{\ov H_{\ov x}^{-1/2} c_x}$ changes in $O(\|\delta_{\ov x}\|_{2,0}+\|\delta_{\ov s}\|_{2,0})$ blocks.
    \item \textsc{Output}$()$: Output $x$ and $s$ in $\wt O(T_H + \nnz(A) + \eta m_\lp m_{\max})$ time.

    \item \textsc{Query}$x(i\in [n])$: Output $x_i$ in $\wt O(n_{\max}^2 + \eta^2 m_{\max}^2)$ time.
    This function is used by \textsc{ApproxDS}.
    \item \textsc{Query}$s(i\in [n])$: Output $s_i$ in $\wt O(n_{\max}^2 + \eta^2 m_{\max}^2)$ time.
    This function is used by \textsc{ApproxDS}.
\end{itemize}
\end{theorem}

\begin{algorithm}[!ht]\caption{ Variation of Algorithm 5 in page 29 in \cite{dly21}. 
This is used in Algorithm~\ref{alg:cpm}. }\label{alg:exact_ds_part_1}
\begin{algorithmic}[1]
\State {\bf data structure} \textsc{ExactDS} \Comment{Theorem~\ref{thm:exactds} 
}
\State {\bf members}
    \State \hspace{4mm} $\ov{x} , \ov{s} \in \R^{n_{\lp}}$, $H_{\ov{x}} \in \R^{n_{\lp} \times n_{\lp}}$, $L_{\ov{x}}\in \R^{m_{\lp} \times m_{\lp}}$
    \State \hspace{4mm} $\wh x, \wh s, c_x \in \R^{n_{\lp}}$, $\epsilon_x, \epsilon_s, h\in \R^{m_{\lp}}, \beta_x, \beta_s\in \R$
    \State \hspace{4mm} $\ov{\alpha} \in \R, \ov{\delta}_\mu \in \R^{n}$
    \State \hspace{4mm} $\ov{t} \in \R_+$
    \State \hspace{4mm} $k\in \bN$
\State {\bf end members}
\Procedure{Initialize}{$x,s, \ov{x}, \ov{s} \in \R^{n_\lp}, \ov{t} \in \R_+$} \Comment{Lemma~\ref{lem:exactds_init_time}}
    \State $\ov{x} \gets \ov{x}$, $\ov{x} \gets \ov{s}$, $\ov{t} \gets \ov{t}$
    \State $\wh{x} \gets x$, $\wh{s} \gets s$, $\epsilon_x \gets 0$, $\epsilon_s \gets 0$, $\beta_x \gets 0$, $\beta_s \gets 0$
    \State $H_{\ov{x}} \gets \nabla^2 \phi( \ov{x} )$
    \State Find lower Cholesky factor $L_{\ov{x}}$ where $L_{\ov{x}} L_{\ov{x}}^\top = A H_{\ov{x}}^{-1} A^\top$
    \State \textsc{Initialize$H$}($\ov{x}, \ov{s}, H_{\ov{x}}, L_{\ov{x}}$)
\EndProcedure
\Procedure{Initialize$H$}{$\ov{x}, \ov{s}, H_{\ov{x}}, L_{\ov{x}}$}
    \For{$i \in [n]$}
        \State $( \ov{\delta}_{\mu})_i \gets - \frac{ \alpha \sinh( \frac{\lambda}{w_i} \gamma_i( \ov{x}, \ov{s}, \ov{t} ) ) }{ \gamma_i( \ov{x}, \ov{s}, \ov{t} ) } \cdot \mu_i( \ov{x}, \ov{s}, \ov{t} )$
        \State $\ov{\alpha} \gets \ov{\alpha} \cdot w_i^{-1} \cosh^2( \frac{\lambda}{w_i} \gamma_i( \ov{x}, \ov{s}, \ov{t} ) )$
    \EndFor 
    \State $c_x \gets H_{\ov{x}}^{-1/2} \ov{\delta}_{\mu}$
    \State $h \gets L_{\ov{x}}^{-1} A H_{\ov{x}}^{-1} \ov{\delta}_{\mu}$
\EndProcedure 
\Procedure{Move}{$ $} \Comment{Lemma~\ref{lem:exactds_move_time}}
    \State $\beta_x \gets \beta_x + (\ov{\alpha})^{-1/2}$
    \State $\beta_s \gets \beta_s + \ov{t} \cdot (\ov{\alpha})^{-1/2}$
    \State \Return $\beta_x, \beta_s$
\EndProcedure 
\Procedure{Output}{$ $} \Comment{Lemma~\ref{lem:exactds_output_time}}
    \State \Return $\wh x + H_{\ov{x}}^{-1/2} \beta_x c_x - H_{\ov x}^{-1/2} \cW^\top (\beta_x h + \epsilon_x), \wh s + H_{\ov x}^{1/2} \cW^\top (\beta_s h + \epsilon_s)$
\EndProcedure
\Procedure{Query$x$}{$i\in [n]$} \Comment{Lemma~\ref{lem:exactds_query_time}}
    \State \Return $\wh x_i + H_{\ov x,(i,i)}^{-1/2} \beta_x c_{x,i} - H_{\ov x}^{-1/2} (\cW^\top (\beta_x h + \epsilon_x))_i$
\EndProcedure
\Procedure{Query$s$}{$i\in [n]$} \Comment{Lemma~\ref{lem:exactds_query_time}}
    \State \Return $\wh s_i + H_{\ov x,(i,i)}^{1/2} (\cW^\top (\beta_s h + \epsilon_s))_i$.
\EndProcedure
\State {\bf end data structure}
\end{algorithmic}
\end{algorithm}

\begin{algorithm}[!ht]\caption{\textsc{ExactDS} Algorithm~\ref{alg:exact_ds_part_1} continued.
}\label{alg:exact_ds_part_2}
\begin{algorithmic}[1]
\State {\bf data structure} \textsc{ExactDS} \Comment{Theorem~\ref{thm:exactds}}
\Procedure{Update}{$\delta_{\ov{x}}, \delta_{\ov{s}}$} \Comment{Lemma~\ref{lem:exactds_update_time}} 
    \State $\Delta_{H_{\ov x}} \gets \nabla^2 \phi( \ov x + \delta_{\ov x}) - H_{\ov x}$
    \Comment{$\Delta_{H_{\ov x}}$ is non-zero only for diagonal blocks $i$ for which $\delta_{\ov{x},_i} \ne 0$}
    \State \textsc{Update$H$}$( \delta_{\ov{x}} , \delta_{\ov{s}}, \Delta_{H_{\ov x}} )$
    \State Find $\Delta_{L_{\ov x}}$ where $(L_{\ov x}+\Delta_{L_{\ov x}})(L_{\ov x}+\Delta_{L_{\ov x}})^\top = A (H_{\ov x}+\Delta_{H_{\ov x}}) A^\top$
    \State \textsc{Update$\cW$}$(\Delta_{L_{\ov x}}, \Delta_{H_{\ov x}})$
    \State $\ov{x} \gets \ov x + \delta_{ \ov{x}}$, $\ov{s} \gets \ov{s} + \delta_{\ov s}$
    \State $H_{\ov{x}} \gets H_{\ov x} + \Delta_{H_{\ov x}}$, $L_{\ov{x}} \gets L_{\ov x} + \Delta_{L_{\ov x}}$
    \State \Return $\delta_{h}, \delta_{\epsilon_x}, \delta_{\epsilon_s}, \delta_{H_{\ov x}^{1/2} \wh x}, \delta_{H_{\ov x}^{-1/2} \wh s}, \delta_{c_x}$
\EndProcedure 
\Procedure{Update$H$}{$\delta_{\ov{x}}$, $\delta_{\ov{s}}$, $\Delta_H$} 
    \State $S \gets \{ i \in [n] ~|~ \delta_{\ov{x},i} \ne 0 \mathrm{~or~} \delta_{\ov{s},i} \ne 0 \}$
    \State $\delta_{ \ov{\delta}_{\mu} } \gets 0$
    \For{$i \in S$}
        \State Let $\gamma_i = \gamma_i(\ov{x}, \ov{s}, \ov{t})$, $\gamma_i^{\new} = \gamma_i( \ov{x} + \delta_{\ov x}, \ov{s} + \delta_{\ov s}, \ov{t} )$, $\mu_i^{\new} = \mu_i ( \ov{x} + \delta_{\ov x}, \ov{s} + \delta_{\ov s}, \ov{t} )$
        \State $\ov{\alpha} \gets \ov{\alpha} - \alpha^2 \cdot w_i^{-1} \cosh^2( \frac{\lambda}{w_i} \gamma_i ) + \alpha^2 w_i^{-1} \cosh^2( \frac{\lambda}{ w_i } \gamma_i^{\new} )$
        \State $\delta_{\ov{\delta}_{\mu},i} \gets - \alpha \sinh( \frac{\lambda}{w_i} {\gamma}_i^{\new} ) \cdot \frac{ 1 }{ \gamma^{\new}_i } \cdot \mu^{\new}_i - \ov{\delta}_{\mu,i}$
    \EndFor 
    \State $\delta_{c_x} \gets \Delta_{H_{\ov x}^{-1/2}} \cdot (\ov{\delta}_{\mu} + \delta_{ \ov{\delta}_{\mu}}) + H_{\ov x}^{-1/2} \cdot \delta_{ \ov{\delta}_{\mu} }$
    \State $\delta_{h} \gets L_{\ov{x}}^{-1} A ( \Delta_{H_{\ov x}^{-1}} \cdot (\ov{\delta}_{\mu} + \delta_{ \ov{\delta}_{\mu}}) + H_{\ov x}^{-1} \cdot \delta_{\ov{\delta}_{\mu} } )$
    \State $\delta_{\epsilon_x} \gets -\beta_x \delta_{h}$
    \State $\delta_{\epsilon_s} \gets -\beta_s \delta_{h}$
    \State $\delta_{\wh{x}} \gets \beta_x ( -\Delta_{H_{\ov x}^{-1/2}} \cdot (c_x + \delta_{c_x}) - H_{\ov x}^{-1/2} \delta_{c_x} ) + \Delta_{H_{\ov x}^{-1/2}} \cdot \cW^\top \cdot (\beta_x h + \epsilon_x)$ 
    \State $\delta_{\wh{s}} \gets -\Delta_{ H_{\ov x}^{1/2} } \cdot \cW^\top (\beta_s h + \epsilon_s)$ \label{line:exactds_updateh_breakpoint}
    \State $\ov \delta_\mu \gets \ov \delta_\mu + \delta_{\ov \delta_\mu}$, $c_x \gets c_x + \delta_{c_x}$, $h \gets h + \delta_{h}$
    \State $\epsilon_x \gets \epsilon_x + \delta_{\epsilon_x}$, $\epsilon_s \gets \epsilon_s + \delta_{\epsilon_s}$
    \State $\wh{x} \gets \wh{x} + \delta_{ \wh{x}}$, $\wh{s} \gets \wh{s} + \delta_{ \wh{s}}$
\EndProcedure 
\Procedure{Update$\cW$}{$\Delta_H, \Delta_{L}$}
    \State $\delta_{\wh{x}} \gets ( H_{\ov x} + \Delta_{H_{\ov x}} )^{-1/2} \cdot \Delta_{H_{\ov x}^{-1/2}} \cdot A^\top L_{\ov x}^{-\top} (\beta_x h + \epsilon_x)$ 
    \State $\delta_{\wh{s}} \gets -( H_{\ov x} + \Delta_{H_{\ov x}} )^{1/2} \cdot \Delta_{H_{\ov x}^{-1/2}} \cdot A^\top L_{\ov x}^{-\top} (\beta_s h + \epsilon_s)$
    \State $\delta_{\epsilon_x} \gets \Delta_{L_{\ov x}}^\top \cdot L_{\ov x}^{-\top} \cdot (\beta_x h + \epsilon_x)$
    \State $\delta_{\epsilon_s} \gets \Delta_{L_{\ov x}}^\top \cdot L_{\ov x}^{-\top} \cdot (\beta_s h + \epsilon_s)$ \label{line:exactds_updatew_breakpoint}
    \State $\wh{x} \gets \wh{x} + \delta_{ \wh{x}}$, $\wh{s} \gets \wh{s} + \delta_{ \wh{s}}$
    \State $\epsilon_x \gets \epsilon_x + \delta_{\epsilon_x}$, $\epsilon_s \gets \epsilon_s + \delta_{\epsilon_s}$
\EndProcedure 
\State {\bf end data structure}
\end{algorithmic}
\end{algorithm}

\begin{proof}[Proof of Theorem~\ref{thm:exactds}]
By combining Lemma~\ref{lem:exactds_correctness}, \ref{lem:exactds_init_time}, \ref{lem:exactds_update_time}, \ref{lem:exactds_output_time}.
\end{proof}

\begin{lemma} \label{lem:exactds_correctness}
\textsc{ExactDS} correctly maintains an implicit representation of $(x, s)$, i.e., invariant
\begin{align*}
    x &= \wh x + H_{\ov{x}}^{-1/2} \beta_x c_x - H_{\ov{x}}^{-1/2} \cW^\top (\beta_x h + \epsilon_x), \\
s &= \wh s + H_{\ov{x}}^{1/2} \cW^\top(\beta_s h + \epsilon_s)
\end{align*}
always holds.
\end{lemma}
\begin{proof}
\textsc{Initialize}:
Initialization satisfies the invariant because $\wh x = x$, $\wh s = s$, $\epsilon_x = \epsilon_s = 0$, $\beta_x = \beta_s = 0$.
Furthermore, we correctly initialize $H_{\ov x}$, $L_{\ov x}$, $\ov \alpha$, $\ov \delta_\mu$, $c_x = H_{\ov x}^{-1/2} \ov \delta_\mu$, $h = L_{\ov x}^{-1} A H_{\ov x}^{-1} \ov \delta_\mu$.

\textsc{Move}:
By inspecting terms with coefficient $\beta_x, \beta_s$, we see that we move in the correct direction and step size.

\textsc{Update}:
We would like to prove that \textsc{Update} does not change the value of $(x,s)$.
First note that $H_{\ov x}$ and $L_{\ov x}$, $\ov \alpha$, $\ov \delta_\mu$ are update correctly.
The remaining updates are separated into two steps: $\textsc{Update$H$}$ and $\textsc{Update$\cW$}$.

\textbf{Step $\textsc{Update$H$}$:}
Write $H_{\ov x}^{\new} := H_{\ov x} + \Delta_{H_{\ov x}}$.
Immediately after Algorithm~\ref{alg:exact_ds_part_2}, Line~\ref{line:exactds_updateh_breakpoint}, we have
\begin{align*}
    c_x + \delta_{c_x} &= H_{\ov x}^{-1/2} \ov \delta_\mu + \Delta_{H_{\ov x}^{-1/2}} (\ov{\delta}_{\mu} + \delta_{ \ov{\delta}_{\mu}}) + H_{\ov x}^{-1/2} \cdot \delta_{ \ov{\delta}_{\mu} }, \\
    &= (H_{\ov x}^{\new})^{-1/2} (\ov{\delta}_{\mu} + \delta_{ \ov{\delta}_{\mu}})\\
    h+\delta_h &= L_{\ov{x}}^{-1} A H_{\ov x}^{-1} \ov \delta_\mu  + L_{\ov{x}}^{-1} A ( \Delta_{H_{\ov x}^{-1}} \cdot (\ov{\delta}_{\mu} + \delta_{ \ov{\delta}_{\mu}}) + H_{\ov x}^{-1} \cdot \delta_{\ov{\delta}_{\mu} } )\\
    &= L_{\ov x}^{-1} A (H_{\ov x}^{\new})^{-1}(\ov \delta_\mu + \delta_{\ov \delta_\mu}).
\end{align*}
So $c_x$ and $h$ are updated correctly.
Furthermore, immediately after Algorithm~\ref{alg:exact_ds_part_2}, Line~\ref{line:exactds_updateh_breakpoint}, we have
\begin{align*}
    &~\delta_{\wh x} + (H_{\ov x}^{\new})^{-1/2} \beta_x (c_x + \delta_{c_x}) - H_{\ov x}^{-1/2} c_x \\
    &~ - ((H_{\ov x}^{\new})^{-1/2}\cW^\top (\beta_x (h+\delta_h) + (\epsilon_x + \delta_{\epsilon_x})) - H_{\ov x}^{-1/2} \cW^\top (\beta_x h + \epsilon_x)) \\
    =&~\beta_x ( -\Delta_{H_{\ov x}^{-1/2}} \cdot (c_x + \delta_{c_x}) - H_{\ov x}^{-1/2} \delta_{c_x} ) + \Delta_{H_{\ov x}^{-1/2}} \cdot \cW^\top \cdot (\beta_x h + \epsilon_x) \\
    &~+ (H_{\ov x}^{\new})^{-1/2} \beta_x (c_x + \delta_{c_x}) - H_{\ov x}^{-1/2} c_x \\
    &~ - ((H_{\ov x}^{\new})^{-1/2} \cW^\top (\beta_x (h+\delta_h) + (\epsilon_x + \delta_{\epsilon_x})) - H_{\ov x}^{-1/2} \cW^\top (\beta_x h + \epsilon_x)) \\
    =&~ \Delta_{H_{\ov x}^{-1/2}} \cdot \cW^\top \cdot (\beta_x h + \epsilon_x) \\
    &~ - ((H_{\ov x}^{\new})^{-1/2} \cW^\top (\beta_x (h+\delta_h) + (\epsilon_x + \delta_{\epsilon_x})) - H_{\ov x}^{-1/2} \cW^\top (\beta_x h + \epsilon_x)) \\
    =&~ - (H_{\ov x}^{\new})^{-1/2}\cW^\top (\beta_x \delta_h + \delta_{\epsilon_x}) \\
    =&~ 0.
\end{align*}
So $x$ is updated correctly, i.e., after \textsc{Update$H$} finishes, we have
\begin{align*}
    x = \wh x + (H_{\ov x}^{\new})^{-1/2} \beta_x c_x - (H_{\ov x}^{\new})^{-1/2} \cW^\top (\beta_x h+\epsilon_x).
\end{align*}

Immediately after Algorithm~\ref{alg:exact_ds_part_2}, Line~\ref{line:exactds_updateh_breakpoint}, we have
\begin{align*}
    &~\delta_{\wh s} + (H_{\ov x}^{\new})^{-1/2}\cW^\top (\beta_s (h+\delta_h) + (\epsilon_s + \delta_{\epsilon_s})) - H_{\ov x}^{1/2} \cW^\top (\beta_s h + \epsilon_s)  \\
    =&~ -\Delta_{ H_{\ov x}^{1/2} } \cdot \cW^\top (\beta_s h + \epsilon_s) \\
    & + (H_{\ov x}^{\new})^{-1/2}\cW^\top (\beta_s (h+\delta_h) + (\epsilon_s + \delta_{\epsilon_s})) - H_{\ov x}^{1/2} \cW^\top (\beta_s h + \epsilon_s) \\
    =&~ (H_{\ov x}^{\new})^{-1/2}\cW^\top (\beta_s \delta_h \delta_{\epsilon_s}) \\
    =&~0.
\end{align*}
So $s$ is updated correctly, i.e., after \textsc{Update$H$} finishes, we have
\begin{align*}
    s = \wh s + (H_{\ov x}^{\new})^{-1/2}\cW^\top (\beta_s h + \epsilon_s).
\end{align*}
This proves correctness of \textsc{Update$H$}.

\textbf{Step $\textsc{Update$\cW$}$:}
Let $H_{\ov x}^\new := H_{\ov x} + \Delta_{H_{\ov x}}$ and $L_{\ov x}^\new := L_{\ov x} + \Delta_{L_{\ov x}}$.
Immediately after Algorithm~\ref{alg:exact_ds_part_2}, Line~\ref{line:exactds_updatew_breakpoint}, we have
\begin{align*}
&~ \delta_{\wh x} - ((H_{\ov x}^\new)^{-1} A^\top (L_{\ov x}^\new)^{-\top} (\beta_x h + (\epsilon_x + \delta_{\epsilon_x}))- ((H_{\ov x}^\new)^{-1/2} H_{\ov x}^{-1/2} A^\top L_{\ov x}^{-\top} (\beta_x h + \epsilon_x))) \\
=&~ (H_{\ov x}^{\new})^{-1/2}  \Delta_{H_{\ov x}^{-1/2}}  A^\top L_{\ov x}^{-\top} (\beta_x h + \epsilon_x) \\
&~- ((H_{\ov x}^\new)^{-1} A^\top (L_{\ov x}^\new)^{-\top} (\beta_x h + (\epsilon_x + \delta_{\epsilon_x}))- (H_{\ov x}^\new)^{-1/2} H_{\ov x}^{-1/2} A^\top L_{\ov x}^{-\top} (\beta_x h + \epsilon_x))) \\
=&~ (H_{\ov x}^{\new})^{-1} A^\top (L_{\ov x}^{-\top} - (L_{\ov x}^{\new})^{-\top})(\beta_x h + \epsilon_x)
- (H_{\ov x}^\new)^{-1} A^\top (L_{\ov x}^\new)^{-\top} \delta_{\epsilon_x} \\
=&~0.
\end{align*}
So $x$ is updated correctly, i.e., after \textsc{Update$\cW$} finishes, we have
\begin{align*}
   x = \wh x + (H_{\ov x}^{\new})^{-1/2} \beta_x c_x - (H_{\ov x}^{\new})^{-1} A^\top (L_{\ov x}^{\new})^{-\top} (\beta_x h+\epsilon_x).
\end{align*}

Immediately after Algorithm~\ref{alg:exact_ds_part_2}, Line~\ref{line:exactds_updatew_breakpoint}, we have
\begin{align*}
&~ \delta_{\wh s} + (A^\top (L_{\ov x}^\new)^{-\top} (\beta_s h + (\epsilon_s + \delta_{\epsilon_s}))- ((H_{\ov x}^\new)^{1/2} H_{\ov x}^{-1/2} A^\top L_{\ov x}^{-\top} (\beta_s h + \epsilon_s))) \\
=&~- (H_{\ov x}^{\new})^{1/2}  \Delta_{H_{\ov x}^{-1/2}}  A^\top L_{\ov x}^{-\top} (\beta_s h + \epsilon_s) \\
&~+ (A^\top (L_{\ov x}^\new)^{-\top} (\beta_s h + (\epsilon_s + \delta_{\epsilon_s}))- (H_{\ov x}^\new)^{1/2} H_{\ov x}^{-1/2} A^\top L_{\ov x}^{-\top} (\beta_s h + \epsilon_s)) \\
=&~ (A^\top ((L_{\ov x}^{\new})^{-\top}-L_{\ov x}^{-\top})(\beta_s h + \epsilon_s)
+ A^\top (L_{\ov x}^\new)^{-\top} \delta_{\epsilon_s} \\
=&~0.
\end{align*}
So $s$ is updated correctly, i.e., after \textsc{Update$\cW$} finishes, we have
\begin{align*}
s = \wh s + A^\top (L_{\ov x}^{\new})^{-\top} (\beta_s h+\epsilon_s).
\end{align*}
\end{proof}

\begin{lemma} \label{lem:exactds_init_time}
\textsc{ExactDS.Initialize} (Algorithm~\ref{alg:exact_ds_part_1}) runs in $O(T_H + T_L + \nnz(A) + T_n + \eta m_\lp m_{\max})$ time.
\end{lemma}
\begin{proof}

Computing $H_{\ov x}$ takes $T_H$ time.
Computing $L_{\ov x}$ takes $T_L$ time.

It remains to analyze the running time of $\textsc{Initialize$H$}$.
Computing $\mu_i(\ov x, \ov s, \ov t)$ takes $T_{H,i}$ time.
Computing $\gamma_i(\ov x, \ov s, \ov t)$ takes $\Tmat(n_i)$ time because it involves computing inverse of $\nabla^2 \phi_i(x_i)$.
So computing $\ov \delta_\mu$ and $\ov \alpha$ takes $T_H + T_n$ time.

Computing $c_x$ takes $T_n$ time by Lemma~\ref{lem:mat_vec_mult_time}\ref{item:lem_mat_vec_mult_time_Hinvv}.
Computing $h$ takes $O(T_n + \eta m_\lp m_{\max} + \nnz(A))$ time by Lemma~\ref{lem:mat_vec_mult_time}\ref{item:lem_mat_vec_mult_time_Hinvv}\ref{item:lem_mat_vec_mult_time_Av}\ref{item:lem_mat_vec_mult_time_Linvv}.
\end{proof}

\begin{lemma} \label{lem:exactds_move_time}
\textsc{ExactDS.Move} (Algorithm~\ref{alg:exact_ds_part_1}) runs in $O(1)$ time.
\end{lemma}
\begin{proof}
All steps in this procedure can be done in $O(1)$ time.
\end{proof}

\begin{lemma} \label{lem:exactds_update_time}
Each call of \textsc{ExactDS.Update} (Algorithm~\ref{alg:exact_ds_part_2}) runs in
\begin{align*}
O((\Tmat(n_{\max}) + T_{\Delta_L,\max} + T_{H,\max} + \eta^2 m_{\max}^2) (\|\delta_{\ov x}\|_{2,0} + \|\delta_{\ov s}\|_{2,0}))
\end{align*}
time.
Furthermore, $\delta_h, \delta_{\epsilon_x}, \delta_{\epsilon_s}$ changes in $O(\eta(\|\delta_{\ov x}\|_{2,0}+\|\delta_{\ov s}\|_{2,0}))$ blocks, and
$\delta_{H_{\ov x}^{1/2} \wh x}, \delta_{H_{\ov x}^{-1/2} \wh s}, \delta_{\ov H_{\ov x}^{-1/2} c_x}$ changes in $O(\|\delta_{\ov x}\|_{2,0}+\|\delta_{\ov s}\|_{2,0})$ blocks.
\end{lemma}
\begin{proof}
Computing $\Delta_{H_{\ov x}}$ takes $O(T_{H,\max} \|\delta_{\ov x}\|_{2,0})$ time, and $\nnz( \Delta_{H_{\ov x}} ) = O(n_{\max}^2 \|\delta_{\ov x}\|_{2,0})$.

Computing $\Delta_{L_{\ov x}}$ takes $O(T_{\Delta_L,\max} \|\delta_{\ov x}\|_{2,0})$ time, and $\nnz(\Delta_{L_{\ov x}}) = O(T_{\Delta_L,\max} \|\delta_{\ov x}\|_{2,0})$.

It remains to analyze two parts \textsc{Update$H$} and \textsc{Update$\cW$}. We will analyze these two parts separately in the next a few paragraphs.

{\bf Part 1.} 

\textsc{Update$H$}: 

Computing $\ov \alpha^{\new}$ and
$\delta_{\ov \delta_\mu}$ takes
\begin{align*} 
    O(\Tmat(n_{\max}) \cdot (\|\delta_{\ov x}\|_{2,0} + \|\delta_{\ov s}\|_{2,0}))
\end{align*}
time, and we have
\begin{align*} 
\|\delta_{\ov \delta_\mu} \|_0 \le \|\delta_{\ov x}\|_{2,0} + \|\delta_{\ov s}\|_{2,0}.
\end{align*}

Computing $\Delta_{ H_{\ov x}^{-1/2}} $, $\Delta_{ H_{\ov x}^{1/2} }$, $\Delta_{ H_{\ov x}^{-1} }$
takes $O(\Tmat(n_{\max}) \|\delta_{\ov x}\|_{2,0})$ time, and they all have $\nnz$ at most $O(n_{\max}^2 \|\delta_{\ov x}\|_{2,0})$.

Computing $\delta_{c_x}$ takes $\nnz(\Delta_{ H_{\ov x}^{-1/2} }) +  n_{\max}^2\|\delta_{\ov \delta_\mu}\|_0$ time.

Computing $\delta_{h}$ takes 
\begin{align*} 
& ~ O(\nnz(\Delta_{ H_{\ov x}^{-1/2} }) + n_{\max}^2 \|\delta_{\ov \delta_\mu} \|_0 + \eta^2 m_{\max}^2 \|\delta_{\ov x}\|_{2,0}) \\
= & ~ O((n_{\max}^2 + \eta^2 m_{\max}^2) (\|\delta_{\ov x}\|_{2,0} + \|\delta_{\ov s}\|_{2,0}))
\end{align*}
time by Lemma~\ref{lem:mat_vec_mult_coord_time}\ref{item:mat_vec_mult_coord_time_LinvtvS}.
Also, $\delta_{h}$ is supported on $O(\|\delta_{x}\|_{2,0})$ paths in the block elimination tree, thus $O(\eta \|\delta_x\|_{2,0})$ blocks, or $O(\eta n_{\max} \|\delta_x\|_{2,0})$ dimension.

We compute $\delta_{\wh x}$ and $\delta_{\wh s}$ from left to right.
This takes 
\begin{align*} 
 & ~ O(\|\delta_{c_x}\|_0 + \nnz(\Delta_{ H_{\ov x}^{1/2} } )+\nnz(\Delta_{H_{\ov x}^{-1/2}}) + \eta^2 m_{\max}^2 \|\delta_x\|_{2,0}) \\
= & ~ O((n_{\max}^2 + \eta^2 m_{\max}^2) (\|\delta_{\ov x}\|_{2,0} + \|\delta_{\ov s}\|_{2,0}))
\end{align*}
time by Lemma~\ref{lem:mat_vec_mult_coord_time}\ref{item:mat_vec_mult_coord_time_Wtvi}.

Remaining computations are all cheap.

{\bf Part 2.}

\textsc{Update$\cW$}:

Computing $\delta_{\wh x}$ and $\delta_{\wh s}$ takes 
\begin{align*} 
O((\eta^2 m_{\max}^2 + \Tmat(n_{\max}))\|\delta_{\ov x}\|_{2,0})
\end{align*}
time by Lemma~\ref{lem:mat_vec_mult_time}.

To compute $\delta_{\epsilon_x}$ and $\delta_{\epsilon_s}$, we first compute $(L^{-\top} (\beta_x h + \epsilon_x))_S$, where $S \in [m]$ is the row support of $\Delta_{L}$, which can be decomposed into at most $\|\delta_{\ov x}\|_{2,0}$ paths.
This takes 
\begin{align*} 
O(\eta^2 m_{\max}^2 \|\delta_{\ov x}\|_{2,0}) 
\end{align*}
time by Lemma~\ref{lem:mat_vec_mult_coord_time}\ref{item:mat_vec_mult_coord_time_LinvtvS}.
So computing $\delta_{\epsilon_x}$ and $\delta_{\epsilon_s}$ takes 
\begin{align*} 
O(\nnz(\Delta_L) + \eta^2 m_{\max}^2 \|\delta_{\ov x}\|_{2,0})
\end{align*}
time.

Combining everything finishes the proof of running time.

For the claim on output sparsity, note that $\delta_h, \delta_{\epsilon_x}, \delta_{\epsilon_s}$ change in $O(\|\delta_{\ov x}\|_{2,0} + \|\delta_{\ov s}\|_{2,0})$ paths, and  that
$\delta_{H_{\ov x}^{1/2} \wh x}, \delta_{H_{\ov x}^{-1/2} \wh s}, \delta_{\ov H_{\ov x}^{-1/2} c_x}$ change only in the support of $\delta_{\ov x}$ and support of $\delta_{\ov s}$.
\end{proof}

\begin{lemma} \label{lem:exactds_output_time}
\textsc{ExactDS.Output} (Algorithm~\ref{alg:exact_ds_part_1}) runs in $O(\nnz(A) + T_n + \eta m_\lp m_{\max})$ time and correctly outputs $(x,s)$.
\end{lemma}
\begin{proof}
Correctness is by Lemma~\ref{lem:exactds_correctness}.
By Lemma~\ref{lem:mat_vec_mult_time}\ref{item:lem_mat_vec_mult_time_WTv}.
\end{proof}

\begin{lemma} \label{lem:exactds_query_time}
\textsc{ExactDS.Query$x$} (Algorithm~\ref{alg:exact_ds_part_1}) and \textsc{ExactDS.Query$s$} (Algorithm~\ref{alg:exact_ds_part_1}) runs in $O(n_{\max}^2 + \eta^2 m_{\max}^2)$ time and returns the correct answer.
\end{lemma}
\begin{proof}
Correctness is obvious. Running time follows from Lemma~\ref{lem:mat_vec_mult_coord_time}\ref{item:mat_vec_mult_coord_time_Wtvi}.
\end{proof}

\subsubsection{\textsc{ApproxDS}}\label{sec:framework:approx_ds}
In this section we introduce \textsc{ApproxDS}, data structure for maintaining a sparsely-changing approximation of $(x,s)$.

\begin{algorithm}[!ht] \caption{This is used in Algorithm~\ref{alg:cpm}.}\label{alg:approxds_1}
\begin{algorithmic}[1]
\small
\State {\bf data structure} \textsc{ApproxDS} \Comment{Theorem~\ref{thm:approxds}}
\State {\bf private : members}
\State \hspace{4mm} $\epsilon_{\apx,x}, \epsilon_{\apx,s}, \zeta_x, \zeta_s\in \R$
\State \hspace{4mm} $\ell \in \bN$
\State \hspace{4mm} \textsc{BatchSketch} $\mathsf{bs}$ \Comment{This maintains a sketch of $H_{\ov x}^{1/2} x$ and $H_{\ov x}^{-1/2} s$. See Algorithm~\ref{alg:batch_sketch_init_update} and \ref{alg:batch_sketch_query}.}

\State \hspace{4mm} \textsc{ExactDS*} $\mathsf{exact}$ \Comment{This is a pointer to the \textsc{ExactDS} (Algorithm~\ref{alg:exact_ds_part_1}, \ref{alg:exact_ds_part_2}) we maintain in parallel to \textsc{ApproxDS}.}

\State \hspace{4mm} $\wt x, \wt s \in \R^{n_\lp}$
\Comment{$(\wt x, \wt s)$ is a sparsely-changing approximation of $(x, s)$. They have the same value as $(\ov x, \ov s)$, but for these local variables we use $(\wt x, \wt s)$ to avoid confusion.}

\State {\bf end members}

\Procedure{\textsc{Initialize}}{$x,s \in \R^{n_\lp}, h, \epsilon_x, \epsilon_s \in \R^{m_\lp}, H_{\ov{x}}^{1/2} \wh{x}, H_{\ov{x}}^{-1/2} \wh{s}, c_x \in \R^{n_\lp}, \beta_x, \beta_s\in \R, q\in \bN,
\textsc{ExactDS*}~\mathsf{exact}, \epsilon_{\apx,x}, \epsilon_{\apx,s}, \delta_{\apx}\in \R$}
    \State $\ell \gets 0$, $q \gets q$
    \State $\epsilon_{\apx,x} \gets \epsilon_{\apx,x}, \epsilon_{\apx,s} \gets \epsilon_{\apx,s}$
    
    \State $\mathsf{bs}.\textsc{Initialize}(x, h, \epsilon_x, \epsilon_s, H_{\ov x}^{1/2} \wh x, H_{\ov x}^{-1/2} \wh s, c_x, \beta_x, \beta_s, \delta_\apx/q)$
    \Comment{Algorithm~\ref{alg:batch_sketch_init_update}}
    \State $\wt x \gets x, \wt s \gets s$
    \State $\mathsf{exact} \gets \mathsf{exact}$
\EndProcedure
\Procedure{\textsc{Update}}{$\delta_{\ov x} \in \R^{n_\lp}, \delta_{h}, \delta_{\epsilon_x}, \delta_{\epsilon_s} \in \R^{m_\lp}, \delta_{H_{\ov x}^{1/2} \wh x}, \delta_{H_{\ov x}^{-1/2} \wh s}, \delta_{H_{\ov x}^{-1/2} c_x} \in \R^{n_\lp}$}
    \State {$\mathsf{bs}.\textsc{Update}(\delta_{\ov x}, \delta_{h}, \delta_{\epsilon_x}, \delta_{\epsilon_s}, \delta_{H_{\ov x}^{1/2} \wh x}, \delta_{H_{\ov x}^{-1/2} \wh s}, \delta_{H_{\ov x}^{-1/2} c_x})$}
    \Comment{Algorithm~\ref{alg:batch_sketch_init_update}}
    \State $\ell \gets \ell+1$
\EndProcedure 
\Procedure{MoveAndQuery}{$\beta_x, \beta_s\in \R$}
    \State $\mathsf{bs}.\textsc{Move}(\beta_x, \beta_s)$
    \Comment{Do not update $\ell$ yet}
    \State $\delta_{\wt x} \gets \textsc{Query$x$}(\epsilon_{\apx,x}/(2\log q+1))$ \Comment{Algorithm~\ref{alg:approxds_2}}
    \State $\delta_{\wt s} \gets \textsc{Query$s$}(\epsilon_{\apx,s}/(2\log q+1))$ \Comment{Algorithm~\ref{alg:approxds_2}}
    \State $\wt x \gets \wt x + \delta_{\wt x}$, $\wt s \gets \wt s + \delta_{\wt s}$
    \State \Return $(\delta_{\wt x}, \delta_{\wt s})$
\EndProcedure
\State {\bf end data structure}
\end{algorithmic}
\end{algorithm}

\begin{algorithm}[!ht] \caption{\textsc{ApproxDS} Algorithm~\ref{alg:approxds_1} continued.}\label{alg:approxds_2}
\begin{algorithmic}[1]
\State {\bf data structure} \textsc{ApproxDS} \Comment{Theorem~\ref{thm:approxds}}
\State {\bf private:}
\Procedure{\textsc{Query$x$}}{$\epsilon \in \R$}
    \State $\cI \gets 0$
    \For{$j=0\to \lfloor \log_2 \ell\rfloor$}
    \If{$\ell \bmod {2^j}=0$}
        \State $\cI \gets \cI \cup \mathsf{bs}.\textsc{Query$x$}(\ell-2^j+1, \epsilon)$
        \Comment{Algorithm~\ref{alg:batch_sketch_query}}
    \EndIf
    \EndFor
    \State $\delta_{\wt x} \gets 0$
    \For{all $i\in \cI$} \label{line:algo_approxds_queryx_xinI}
        \State $x_i \gets \mathsf{exact}.\textsc{Query}x(i)$
        \Comment{Algorithm~\ref{alg:exact_ds_part_1}}
        \If{$\|\wt x_i - x_i\|_{\wt x_i} > \epsilon$} \label{line:algo_approxds_queryx_checkxi}
            \State $\delta_{\wt x,i} \gets x_i-\wt x_i$
        \EndIf
    \EndFor
    \State \Return $\delta_{\wt x}$
\EndProcedure
\Procedure{\textsc{Query$s$}}{$\epsilon \in \R$}
    \State $\cI \gets 0$
    \For{$j=0\to \lfloor \log_2 \ell\rfloor$}
    \If{$\ell \bmod {2^j}=0$}
        \State $\cI \gets \cI \cup \mathsf{bs}.\textsc{Query$s$}(\ell-2^j+1, \epsilon)$
        \Comment{Algorithm~\ref{alg:batch_sketch_query}}
    \EndIf
    \EndFor
    \State $\delta_{\wt s} \gets 0$
    \For{all $i\in \cI$}
        \State $s_i \gets \mathsf{exact}.\textsc{Query}s(i)$
        \Comment{Algorithm~\ref{alg:exact_ds_part_1}}
        \If{$\|\wt s_i - s_i\|_{\wt x_i}^* > \epsilon$}
            \State $\delta_{\wt s,i} \gets s_i-\wt s_i$
        \EndIf
    \EndFor
    \State \Return $\delta_{\wt s}$
\EndProcedure
\State {\bf end data structure}
\end{algorithmic}
\end{algorithm}

\begin{theorem}\label{thm:approxds}
Given parameters $\epsilon_{\apx,x}, \epsilon_{\apx,s} \in (0, 1), \delta_\apx \in (0,1)$, $\zeta_{x}, \zeta_{s}\in \R$ such that 
\begin{align*}
\|H_{\ov x^{(\ell)}}^{1/2} x^{(\ell)}-H_{\ov x^{(\ell)}}^{1/2} x^{(\ell+1)}\|_2 \le \zeta_{x}, ~~~ \|H_{\ov x^{(\ell)}}^{-1/2} s^{(\ell)}-H_{\ov x^{(\ell)}}^{-1/2} s^{(\ell+1)}\|_2 \le \zeta_{s}
\end{align*}
for all $\ell \in \{0,\ldots,q-1\}$,
data structure \textsc{ApproxDS} (Algorithm~\ref{alg:approxds_1} and Algorithm~\ref{alg:approxds_2}) supports the following operations:
\begin{itemize}
\item $\textsc{Initialize}(x,s \in \R^{n_\lp}, h, \epsilon_x, \epsilon_s \in \R^{m_\lp}, H_{\ov{x}}^{1/2} \wh{x}, H_{\ov{x}}^{-1/2} \wh{s}, c_x \in \R^{n_\lp}, \beta_x, \beta_s \in \R, q\in \bN$, $\textsc{ExactDS*}~\mathsf{exact}$, $\epsilon_{\apx,x}, \epsilon_{\apx,s}, \delta_{\apx}\in \R)$: Initialize the data structure in $$\wt O(T_H + T_n + T_L + r\cdot (n \cdot n_{\max}^2 + T_Z + \nnz(A)))$$ time.

\item $\textsc{MoveAndQuery}(\beta_x, \beta_s\in \R)$:
Update values of $\beta_x, \beta_s$ by calling $\textsc{BatchSketch}.\textsc{Move}$.
This effectively moves $(x^{(\ell)}, s^{(\ell)})$ to $(x^{(\ell+1)}, s^{(\ell+1)})$ while keeping $\ov x^{(\ell)}$ unchanged.

Then return two sets $L_x^{(\ell)}, L_s^{(\ell)} \subset [n]$ where
\begin{align*}
L_x^{(\ell)} &\supseteq \{i \in [n] : \|H_{\ov x^{(\ell)}}^{1/2} x^{(\ell)}_i - H_{\ov x^{(\ell)}}^{1/2} x^{(\ell+1)}_i\|_2 \ge \epsilon_{\apx,x}\},\\
L_s^{(\ell)} &\supseteq \{i \in [n] : \|H_{\ov x^{(\ell)}}^{-1/2} s^{(\ell)}_i - H_{\ov x^{(\ell)}}^{-1/2} s^{(\ell+1)}_i\|_2 \ge \epsilon_{\apx,s}\},
\end{align*}
satisfying
\begin{align*}
\sum_{0\le \ell \le q-1} |L_x^{(\ell)}| = \wt O(\epsilon_{\apx,x}^{-2} \zeta_x^2 q^2), \\
\sum_{0\le \ell \le q-1} |L_s^{(\ell)}| = \wt O(\epsilon_{\apx,s}^{-2} \zeta_s^2 q^2).
\end{align*}

For every query, with probability at least $1-\delta_\apx / q$, the return values are correct.

Furthermore, total time cost over all queries is at most
\begin{align*}
\wt O\left( (\epsilon_{\apx,x}^{-2} \zeta_{x}^2 + \epsilon_{\apx,s}^{-2} \zeta_{s}^2)q^2 r (n_{\max}^2 + \eta^2 m_{\max}^2)\right).
\end{align*}

\item $\textsc{Update}(\delta_{\ov x}, \delta_{h}, \delta_{\epsilon_x}, \delta_{\epsilon_s}, \delta_{H_{\ov x}^{1/2} \wh x}, \delta_{H_{\ov x}^{-1/2} \wh s}, \delta_{H_{\ov x}^{-1/2} c_x})$: Update sketches of $H_{\ov x^{(\ell)}}^{1/2} x^{(\ell+1)}$ and $H_{\ov x^{(\ell)}}^{-1/2} s^{(\ell+1)}$ by calling \textsc{BatchSketch}.\textsc{Update}.
This effectively moves $\ov x^{(\ell)}$ to $\ov x^{(\ell+1)}$ while keeping $(x^{(\ell+1)}, s^{(\ell+1)})$ unchanged. Then advance timestamp $\ell$.

Each update costs
\begin{align*}
\wt O(& (T_{H,\max} + T_{\Delta L,\max} + \Tmat(n_{\max}) + r \eta^2 m_{\max}^2 ) \|\delta_{\ov x}\|_{2,0} \\
& + r m_{\max}  \cdot (\|\delta_{h}\|_0 + \|\delta_{\epsilon_x}\|_0 + \|\delta_{\epsilon_s}\|_0) \\
& + r n_{\max} \cdot (\|\delta_{H_{\ov x}^{1/2} \wh x}\|_0+\| \delta_{H_{\ov x}^{-1/2} \wh s}\|_0+\|\delta_{H_{\ov x}^{-1/2} c_x}\|_0))
\end{align*}
time.
\end{itemize}

\end{theorem}
\begin{remark} \label{rmk:dly_linf_approx_error}
In \cite{dly21}, Theorem 6.14 (their counterpart of Theorem~\ref{thm:approxds}) is not used properly, making their algorithm incorrect.
For \cite[Theorem 6.14]{dly21} to apply, they need $\|H_{\ov x^{(\ell)}}^{1/2} x^{(\ell)} - H_{\ov x^{(\ell+1)}}^{1/2} x^{(\ell+1)}\|_2 \le \zeta_x$.
However, in the proof of \cite[Theorem 6.1]{dly21}, they only prove
$\|H_{\ov x^{(\ell)}}^{1/2} x^{(\ell)} - H_{\ov x^{(\ell)}}^{1/2} x^{(\ell+1)}\|_2 \le \zeta_x$, and the change in $\ov x$ is ignored.
Because \cite[Algorithm 8]{dly21} uses sampling instead of top-down BFS as in Algorithm~\ref{alg:batch_sketch_query}, the change in $\ov x$ could potentially make $\|H_{\ov x^{(\ell)}}^{1/2} x^{(\ell)} - H_{\ov x^{(\ell+1)}}^{1/2} x^{(\ell+1)}\|_2$ very large, and make their algorithm unable to find large changes in $x_i$.
Similar issue also happens for $H_{\ov x}^{-1/2} s$.
We fix these issues by replacing their sampling step with a top-down BFS in \textsc{BatchSketch}, and considering change in $\ov x$ in the analysis of \textsc{ApproxDS}.
\end{remark}

\begin{proof}[Proof of Theorem~\ref{thm:approxds}]
\textbf{Correctness:}
We prove that with high probability ($1-1/\poly(n_\lp)$), for all blocks $i\in [n]$ not in $\supp \delta_{\wt x}$, we have $\|\wt x^{(\ell)} - x^{(\ell+1)}\|_{\wt x^{(\ell)}} \le \epsilon_{\apx, x}$. Recall that $\wt x^{(t)}$ and $\ov x^{(t)}$ are equal for all $t$ because updates $\delta_{\wt x}$ are always applied to $\ov x$.

Fix current timestamp $\ell$ and a block $i\in [n]$.
Suppose that $\ov x_i$ is not updated at timestamp $\ell$, i.e., $\ov x^{(\ell)}_i \ne \ov x^{(\ell+1)}_i$. If $i\in \cI$ then we are done because it is checked that $\|\ov x^{(\ell)}_i - x^{(\ell+1)}_i\|_{\ov x^{(\ell)}_i} \le \epsilon \le \epsilon_{\apx,x}$ at Algorithm~\ref{alg:approxds_2} Line~\ref{line:algo_approxds_queryx_checkxi}.
Let $\ell'$ be the last time where block $i$ was in $\cI$ in Algorithm~\ref{alg:approxds_2} Line~\ref{line:algo_approxds_queryx_xinI}, i.e., $\ell' = \max \{t: t \le \ell-1, i\in \cI \text{ in Algorithm~\ref{alg:approxds_2} Line~\ref{line:algo_approxds_queryx_xinI}}\}$.
Then we can divide the interval $[\ell', \ell]$ into $k\le 2\log q$ intervals, i.e., we can find $\ell' \le t_0 < \cdots < t_k = \ell$ such $k\le 2\log q$, and for every $o\in [k]$, there exists $j$ such that $t_{o-1} = t_o-2^j$ and $t_o \bmod 2^j = 0$.
Then by our assumption on $\ell'$, we have
\begin{align*}
    \|\ov x_i^{(\ell'+1)} - x_i^{(\ell'+1)}\|_{\ov x_i^{(\ell'+1)}} \le \epsilon
\end{align*}
where $\epsilon = \epsilon_{\apx,x}/(2\log q+1)$, because whether or not $\ov x^{(\ell')}_i \ne \ov x^{(\ell'+1)}_i$, this bound always holds.
Furthermore, $i\not \in \cI$ at time $t_1,\ldots, t_k$, so
\begin{align*}
    \|H_{\ov x^{(t_o+1)}}^{1/2} x^{(t_o+1)} - H_{\ov x^{(t_{o+1})}}^{1/2} x^{(t_{o+1}+1)}\| \le \epsilon,
\end{align*}
which implies 
\begin{align*}
    \|x^{(t_o+1)}_i - x^{(t_{o+1}+1)}_i\|_{\ov x_i^{(\ell'+1)}} \le \epsilon
\end{align*}
because $\ov x_i^{(\ell'+1)} = \cdots = \ov x_i^{(\ell)}$.
So
\begin{align*}
    &\|\ov x^{(\ell)}_i - x^{(\ell+1)}_i\|_{\ov x^{(\ell)}_i} \\
    &\le~
    \|\ov x_i^{(\ell'+1)} - x_i^{(\ell'+1)}\|_{\ov x_i^{(\ell'+1)}}
    + \sum_{o\in [k]} \|x^{(t_o+1)}_i - x^{(t_{o+1}+1)}_i\|_{\ov x_i^{(\ell'+1)}}\\
    &\le~ (k+1) \epsilon \le \epsilon_{\apx, x}.
\end{align*}
This proves that
\begin{align*}
\|\ov x^{(\ell+1)}_i - x^{(\ell+1)}_i\|_{\ov x^{(\ell+1)}_i} \le \epsilon_{\apx, x}.
\end{align*}
Therefore
\begin{align*}
    L_x^{(\ell)} &\supseteq \{i \in [n] : \|H_{\ov x^{(\ell)}}^{1/2} x^{(\ell)}_i - H_{\ov x^{(\ell)}}^{1/2} x^{(\ell+1)}_i\|_2 \ge \epsilon_{\apx,x}\}.
\end{align*}

On the other hand, because Algorithm~\ref{alg:approxds_2} Line~\ref{line:algo_approxds_queryx_checkxi}, all elements $i\in L_x^{(\ell)}$ satisfy
\begin{align*}
\|\ov x^{(\ell)}_i - x^{(\ell+1)}_i\|_{\ov x^{(\ell)}} \ge \epsilon.
\end{align*}
So
\begin{align*}
    &~\epsilon^2 \sum_{0\le \ell \le q-1} |L_x^{\ell)}| \\
    \le &~ \sum_{0\le \ell \le q-1} \sum_{i\in L_x^{(\ell)}} \|\ov x^{(\ell)}_i - x^{(\ell+1)}_i\|_{\ov x^{(\ell)}_i}^2 \\
    \le &~ \sum_{i\in [n]} q \sum_{0\le \ell \le q-1} \|x^{(\ell)}_i - x^{(\ell+1)}_i\|_{\ov x^{(\ell)}_i}^2 \\
    = &~ q \sum_{0\le \ell \le q-1} \|H_{\ov x^{(\ell)}}^{1/2} x^{(\ell)} - H_{\ov x^{(\ell)}}^{1/2} x^{(\ell+1)}\|_2\\
    \le &~ q^2 \zeta_x^2.
\end{align*}
This proves that
\begin{align*}
    \sum_{0\le \ell \le q-1} |L_x^{(\ell)}| = \wt O(\epsilon_{\apx,x}^{-2} \zeta_x^2 q^2).
\end{align*}

The proof for $s$ is similar and omitted.

\textbf{Running time:}

\textsc{Initialize}: By \textsc{Initialize} part of Theorem~\ref{thm:batch_sketch}.

\textsc{Update}: By \textsc{Update} part of Theorem~\ref{thm:batch_sketch}.

\textsc{Query$x$} and \textsc{Query$s$}:
Let $\cI^{(\ell)}_x$ be the set $\cI$ in $\textsc{Query}x$ at timestamp $\ell$.
Then by Theorem~\ref{thm:batch_sketch}, we have
\begin{align*}
    |\cI^{(\ell)}_x| = \sum_{0\le j\le \lfloor \log_2 \ell\rfloor, 2^j|\ell} O(&\epsilon^{-2} 2^j \sum_{\ell -2^j+1 \le t \le \ell} \|H_{\ov x^{(t)}}^{1/2} x^{(t)} - H_{\ov x^{(t)}}^{1/2} x^{(t+1)}\|_2^2 \\
    & + \sum_{\ell-2^j+1 \le t \le \ell-1} \|\ov x^{(t)} - \ov x^{(t+1)}\|_{2,0}).
\end{align*}
We have 
\begin{align*}
    \sum_{0\le \ell \le q-1}  \sum_{0\le j\le \lfloor \log_2 \ell\rfloor, 2^j|\ell} O(\epsilon^{-2} 2^j \sum_{\ell -2^j+1 \le t \le \ell} \|H_{\ov x^{(t)}}^{1/2} x^{(t)} - H_{\ov x^{(t)}}^{1/2} x^{(t+1)}\|_2^2)
    = \wt O(\epsilon^{-2} \zeta_x^2 q^2)
\end{align*}
and
\begin{align*}
    &~ \sum_{0\le \ell \le q-1} \sum_{0\le j\le \lfloor \log_2 \ell\rfloor, 2^j|\ell} O(\sum_{\ell-2^j+1 \le t \le \ell-1} \|\ov x^{(t)} - \ov x^{(t+1)}\|_{2,0}) \\
    =&~ \sum_{0\le \ell \le q-1}  \sum_{0\le j\le \lfloor \log_2 \ell\rfloor, 2^j|\ell} O(\sum_{\ell-2^j+1 \le t \le \ell-1} \|\ov x^{(t)} - \ov x^{(t+1)}\|_{2,0})\\
    =&~ \wt O(\sum_{0\le \ell \le q-1} L_x^{(\ell)})\\
    =&~ \wt O(\epsilon^{-2} \zeta_x^2 q^2)
\end{align*}
because the dyadic intervals $[\ell-2^j+1,\ell]$ together cover the whole inteval $[0, q-1]$ at most $O(\log q) = \wt O(1)$ times.

Therefore we have proved that
\begin{align*}
    \sum_{0\le \ell \le q-1}  \|\cI_x^{(\ell)}\| = \wt O(\epsilon^2 \zeta_x^2 q^2) = \wt O(\epsilon_{\apx,x}^2 \zeta_x^2 q^2).
\end{align*}

Then the running time bound for \textsc{Query$x$} follows from \textsc{Query} part of Theorem~\ref{thm:exactds} and \textsc{Query$x$} part of Theorem~\ref{thm:batch_sketch}.

The proof for $s$ is similar and omitted.
\end{proof}

\subsubsection{\textsc{BatchSketch}}\label{sec:framework:batch_sketch}

In this section we introduce \textsc{BatchSketch}, our data structure for maintaining a sketch of $H_{\ov x}^{1/2} x$ and $H_{\ov x}^{-1/2} s$.

For this task, we need another tree structure on the set of variable blocks.
\begin{definition}[Partition tree\footnote{This is called a sampling tree in \cite{dly21}. We changed the name because this data structure has little to do with sampling.}]
A partition tree $(\cS, \chi)$ of $\R^{n}$ consists of a constant degree rooted tree $\cS = (V, E)$ and a labeling of the vertices $\chi: V\to 2^{[n]}$, such that
\begin{itemize}
    \item $\chi(\mathrm{root}) = [n]$
    \item If $v$ is a leaf node of $\cS$, then $|\chi(v)|=1$
    \item For any node $v$ of $\cS$, the set $\{\chi(c): \text{$c$ is a child of $v$}\}$ forms a partition of $\chi(v)$.
\end{itemize}
\end{definition}
The partition tree does not need to (but can) have any relationship with the block elimination tree.
The partition tree will need to satisfy $O(1)$ maximum degree and $\wt O(1)$ depth.
We will choose the partition tree in Section~\ref{sec:framework:block_balanced_sketch} so that the \textsc{BlockBalancedSketch} data structure can be efficiently maintained.

\begin{algorithm}[!ht]\caption{This is used by Algorithm~\ref{alg:approxds_1} and \ref{alg:approxds_2}.}\label{alg:batch_sketch_init_update}
\begin{algorithmic}[1]
\State {\bf data structure} \textsc{BatchSketch} \Comment{Theorem~\ref{thm:batch_sketch}}
\State {\bf memebers}
\State \hspace{4mm} $\Phi\in \R^{r\times n_\lp} $ \Comment{All sketches need to share the same sketching matrix}
\State \hspace{4mm} $\cS,\chi$ partition tree
\State \hspace{4mm} $\ell \in \bN$ \Comment{Current timestamp}
\State \hspace{4mm} \textsc{BlockBalancedSketch} $\mathsf{sketch}\cW^\top h$, $\mathsf{sketch}\cW^\top \epsilon_x$, $\mathsf{sketch}\cW^\top \epsilon_s$ 
\Comment{Algorithm \ref{alg:block_balanced_sketch_1}}
\State \hspace{4mm} \textsc{BlockVectorSketch} $\mathsf{sketch}H_{\ov x}^{1/2} \wh x$, $\mathsf{sketch}H_{\ov x}^{-1/2} \wh s$, $\mathsf{sketch} c_x$
\Comment{Algorithm \ref{alg:block_vector_sketch}}
\State \hspace{4mm} $\beta_x, \beta_s \in \R$
\State \hspace{4mm} $(\mathsf{history}[t])_{t\ge 0}$
\Comment{Snapshot of data at timestamp $t$. See Remark~\ref{rmk:snapshot}.}
\State {\bf end memebers}
\Procedure{Initialize}{$x \in \R^{n_\lp}, h, \epsilon_x, \epsilon_s \in \R^{m_\lp}, H_{\ov{x}}^{1/2} \wh{x}, H_{\ov{x}}^{-1/2} \wh{s}, c_x\in \R^{n_\lp}, \beta_x, \beta_s \in \R, \delta_\apx \in \R$}
    \State Construct partition tree $(\cS, \chi)$ as in Definition~\ref{defn:partition_tree_construction}
    \State $r \gets \Theta(\log^3(n_\lp)\log(1/\delta_\apx))$
    \State Initialize $\Phi\in \R^{r \times n_\lp}$ with iid $\cN(0, \frac 1{r})$
    \State $\beta_x \gets \beta_x$, $\beta_s \gets \beta_s$
    \State $\mathsf{sketch}\cW^\top h.\textsc{Initialize}( {\cal S} , {\chi}, \Phi, x, h )$ 
    \Comment{Algorithm~\ref{alg:block_balanced_sketch_1}}
    \State $\mathsf{sketch}\cW^\top \epsilon_x.\textsc{Initialize}( {\cal S} , {\chi}, \Phi, x, \epsilon_x )$ 
    \Comment{Algorithm~\ref{alg:block_balanced_sketch_1}}
    \State $\mathsf{sketch}\cW^\top \epsilon_s.\textsc{Initialize}( {\cal S} , {\chi}, \Phi, x, \epsilon_s )$ 
    \Comment{Algorithm~\ref{alg:block_balanced_sketch_1}}
    \State $\mathsf{sketch}H_{\ov x}^{1/2} \wh x.\textsc{Initialize}( {\cal S} , {\chi}, \Phi, H_{\ov{x}}^{1/2} \wh{x} )$
    \Comment{Algorithm~\ref{alg:block_vector_sketch}}
    \State $\mathsf{sketch}H_{\ov x}^{-1/2} \wh s.\textsc{Initialize}( {\cal S} , {\chi}, \Phi, H_{\ov{x}}^{-1/2} \wh{s} )$
    \Comment{Algorithm~\ref{alg:block_vector_sketch}}
    \State $\mathsf{sketch} c_x.\textsc{Initialize}( {\cal S} , {\chi}, \Phi, c_x )$
    \Comment{Algorithm~\ref{alg:block_vector_sketch}}
    \State $\ell \gets 0$
    \State Make snapshot $\mathsf{history}[\ell]$ \Comment{Remark~\ref{rmk:snapshot}}
\EndProcedure
\Procedure{Move}{$\beta_x, \beta_s\in \R$}
    \State $\beta_x \gets \beta_x$, $\beta_s \gets \beta_s$
    \Comment{Do not update $\ell$ yet}
\EndProcedure
\Procedure{Update}{$\delta_{\ov x} \in \R^{n_\lp}, \delta_{h}, \delta_{\epsilon_x}, \delta_{\epsilon_s} \in \R^{m_{\lp}}, \delta_{H_{\ov x}^{1/2} \wh{x}}, \delta_{ H_{\ov x}^{-1/2} \wh{s} } , \delta_{c_x} \in \R^{n_\lp}$}
    \State $\mathsf{sketch}\cW^\top h.\textsc{Update}(\delta_{\ov x}, \delta_{h})$
    \Comment{Algorithm~\ref{alg:block_balanced_sketch_2}}
    \State $\mathsf{sketch}\cW^\top \epsilon_x.\textsc{Update}(\delta_{\ov x}, \delta_{\epsilon_x})$
    \Comment{Algorithm~\ref{alg:block_balanced_sketch_2}}
    \State $\mathsf{sketch}\cW^\top \epsilon_s.\textsc{Update}(\delta_{\ov x}, \delta_{\epsilon_s})$
    \Comment{Algorithm~\ref{alg:block_balanced_sketch_2}}
    \State $\mathsf{sketch}H_{\ov x}^{1/2} \wh x.\textsc{Update}( \delta_{ H_{\ov{x}}^{1/2} \wh{x} } )$
    \Comment{Algorithm~\ref{alg:block_vector_sketch}}
    \State $\mathsf{sketch}H_{\ov x}^{-1/2} \wh s.\textsc{Update}(  \delta_{ H_{\ov{x}}^{-1/2} \wh{s}} )$ 
    \Comment{Algorithm~\ref{alg:block_vector_sketch}}
    \State $\mathsf{sketch}c_x.\textsc{Update}( \delta_{ c_x } )$
    \Comment{Algorithm~\ref{alg:block_vector_sketch}}
    \State $\ell \gets \ell+1$
    \State Make snapshot $\mathsf{history}[\ell]$ \Comment{Remark~\ref{rmk:snapshot}}
\EndProcedure
\State {\bf end data structure}
\end{algorithmic}
\end{algorithm}

\begin{algorithm}[!ht]\caption{\textsc{BatchSketch} Algorithm~\ref{alg:batch_sketch_init_update} continued. This is used by Algorithm~\ref{alg:approxds_1} and \ref{alg:approxds_2}.}\label{alg:batch_sketch_query}
\begin{algorithmic}[1]
\State {\bf data structure} \textsc{BatchSketch} \Comment{Theorem~\ref{thm:batch_sketch}}
\State {\bf private:}
    \Procedure{Query$x$Sketch}{$v\in \cS$}
        \Comment{Return the value of $\Phi_{\chi(v)} (H_{\ov x}^{1/2} x)_{\chi(v)}$}
        \State \Return $-\beta_x \mathsf{sketch}\cW^\top h.\textsc{Query}(v) - \mathsf{sketch}\cW^\top \epsilon_x.\textsc{Query}(v)
        + \mathsf{sketch} H_{\ov x}^{1/2} \wh x.\textsc{Query}(v) +
        \beta_x \mathsf{sketch} c_x.\textsc{Query}(v)$
        \Comment{Algorithm~\ref{alg:block_vector_sketch}, \ref{alg:block_balanced_sketch_1}}
    \EndProcedure
    \Procedure{Query$s$Sketch}{$v\in \cS$}
        \Comment{Return the value of $\Phi_{\chi(v)} (H_{\ov x}^{-1/2} s)_{\chi(v)}$}
        \State \Return $\beta_s \mathsf{sketch}\cW^\top h.\textsc{Query}(v) + \mathsf{sketch}\cW^\top \epsilon_s.\textsc{Query}(v)
        + \mathsf{sketch} H_{\ov x}^{-1/2} \wh s.\textsc{Query}(v)$
        \Comment{Algorithm~\ref{alg:block_vector_sketch}, \ref{alg:block_balanced_sketch_1}}
    \EndProcedure
\State {\bf public:}
\Procedure{Query$x$}{$\ell' \in \bN, \epsilon \in \R$}
    \State $L_0=\{\text{root}(\cS)\}$
    \State $S \gets \emptyset$
    \For{$d = 0 \to \infty$}
        \If{$L_d=\emptyset$}
        \State \Return $S$
        \EndIf
        \State $L_{d+1} \gets \emptyset$
        \For{$v \in L_d$}
            \If{$v$ is a leaf node}
            \State $S \gets S \cup \{v\}$
            \Else
            \For{$u$ child of $v$}
            \If{$\|\textsc{Query$x$Sketch}(u) - \mathsf{history}[\ell'].\textsc{Query$x$Sketch}(u)\|_2 > 0.9\epsilon$} \label{line:batch_sketch_queryx}
                \State $L_{d+1} \gets L_{d+1} \cup \{ u \}$
            \EndIf
            \EndFor
            \EndIf
        \EndFor 
    \EndFor 
\EndProcedure
\Procedure{Query$s$}{$\ell' \in \bN, \epsilon \in \R$}
    \State Same as \textsc{Query}$x$, except for replacing $\textsc{Query$x$Sketch}$ in Line~\ref{line:batch_sketch_queryx} with $\textsc{Query$s$Sketch}$.
\EndProcedure
\State {\bf end structure}
\end{algorithmic}
\end{algorithm}

\begin{theorem}\label{thm:batch_sketch}
Data structure \textsc{BatchSketch} (Algorithm~\ref{alg:batch_sketch_init_update}, \ref{alg:batch_sketch_query}) supports the following operations:
\begin{itemize}
\item $\textsc{Initialize}(x \in \R^{n_\lp}, h, \epsilon_x, \epsilon_s \in \R^{m_\lp}, H_{\ov{x}}^{1/2} \wh{x}, H_{\ov{x}}^{-1/2} \wh{s}, c_x \in \R^{n_\lp}, \beta_x, \beta_s \in \R, \delta \in \R)$: Initialize the data structure in $$\wt O(T_H + T_n + T_L + r\cdot (n n_{\max}^2 + T_Z + \nnz(A)))$$ time, where $r = \Theta(\log^3(n_\lp) \log (1/\delta))$.

\item $\textsc{Move}(\beta_x, \beta_s\in \R)$: Update values of $\beta_x, \beta_s\in \R$ in $O(1)$ time. This effectively moves $(x^{(\ell)}, s^{(\ell)})$ to $(x^{(\ell+1)}, s^{(\ell+1)})$ while keeping $\ov x^{(\ell)}$ unchanged.

\item $\textsc{Update}(\delta_{\ov x}\in \R^{n_\lp}, \delta_{h}, \delta_{\epsilon_x}, \delta_{\epsilon_s} \in \R^{m_\lp}, \delta_{H_{\ov x}^{1/2} \wh x}, \delta_{H_{\ov x}^{-1/2} \wh s}, \delta_{H_{\ov x}^{-1/2} c_x}\in \R^{n_\lp})$:
Update sketches of $H_{\ov x^{(\ell)}}^{1/2} x^{(\ell+1)}$ and $H_{\ov x^{(\ell)}}^{-1/2} s^{(\ell+1)}$ by updating sketches of $\cW^\top h, \cW^\top \epsilon_x, \cW^\top \epsilon_s, H_{\ov x}^{1/2} \wh x, H_{\ov x}^{-1/2} \wh s, c_x$.
This effectively moves $\ov x^{(\ell)}$ to $\ov x^{(\ell+1)}$ while keeping $(x^{(\ell+1)}, s^{(\ell+1)})$ unchanged.
Then advance timestamp $\ell$.

Each update costs
\begin{align*}
\wt O(& (T_{H,\max} + T_{\Delta L,\max} + \Tmat(n_{\max}) + r \eta^2 m_{\max}^2 ) \|\delta_{\ov x}\|_{2,0} \\
& + r m_{\max} \cdot (\|\delta_{h}\|_0 + \|\delta_{\epsilon_x}\|_0 + \|\delta_{\epsilon_s}\|_0) \\
& + r n_{\max} \cdot (\|\delta_{H_{\ov x}^{1/2} \wh x}\|_0+\| \delta_{H_{\ov x}^{-1/2} \wh s}\|_0+\|\delta_{c_x}\|_0))
\end{align*}
time.

\item $\textsc{Query}x(\ell' \in \bN, \epsilon \in\R)$:
Given timestamp $\ell'$, return a set $S\subseteq [n]$ where
\begin{align*}
    S &\supseteq \{i \in [n]: \|H_{\ov x^{(\ell')}}^{1/2} x^{(\ell')}_i - H_{\ov x^{(\ell)}}^{1/2} x^{(\ell+1)}_i\|_2 \ge \epsilon\},
\end{align*}
and
\begin{align*}
    |S| & = O(\epsilon^{-2} (\ell-\ell'+1) \sum_{\ell' \le t \le \ell} \|H_{\ov x^{(t)}}^{1/2} x^{(t)} - H_{\ov x^{(t)}}^{1/2} x^{(t+1)}\|_2^2 + \sum_{\ell ' \le t \le \ell-1} \|\ov x^{(t)} - \ov x^{(t+1)}\|_{2,0})
\end{align*}
where $\ell$ is the current timestamp.

For every query, with probability at least $1-\delta$, the return values are correct, and costs at most
\begin{align*} 
\wt O(r \eta^2 m_{\max}^2 
\cdot (\epsilon^{-2} (\ell-\ell'+1) \sum_{\ell' \le t \le \ell} \|H_{\ov x^{(t)}}^{1/2} x^{(t)} - H_{\ov x^{(t)}}^{1/2} x^{(t+1)}\|_2^2 + \sum_{\ell ' \le t \le \ell-1} \|\ov x^{(t)} - \ov x^{(t+1)}\|_{2,0}))
\end{align*}
running time.

\item $\textsc{Query}s(\ell' \in \bN, \epsilon \in\R)$:
Given timestamp $\ell'$, return a set $S\subseteq [n]$ where
\begin{align*}
    S &\supseteq \{i \in [n]: \|H_{\ov x^{(\ell')}}^{-1/2} s^{(\ell')}_i - H_{\ov x^{(\ell)}}^{-1/2} s^{(\ell+1)}_i\|_2 \ge \epsilon\}
\end{align*}
and
\begin{align*}
    |S| & = O(\epsilon^{-2}  (\ell-\ell'+1)\sum_{\ell' \le t \le \ell} \|H_{\ov s^{(t)}}^{-1/2} x^{(t)} - H_{\ov x^{(t)}}^{-1/2} s^{(t+1)}\|_2^2 + \sum_{\ell' \le t \le \ell-1} \|\ov x^{(t)} - \ov x^{(t+1)}\|_{2,0})
\end{align*}
where $\ell$ is the current timestamp.

For every query, with probability at least $1-\delta$, the return values are correct, and costs at most
\begin{align*} 
\wt O(r \eta^2 m_{\max}^2 \cdot(\epsilon^{-2}  (\ell-\ell'+1)\sum_{\ell' \le t \le \ell} \|H_{\ov s^{(t)}}^{-1/2} x^{(t)} - H_{\ov x^{(t)}}^{-1/2} s^{(t+1)}\|_2^2 + \sum_{\ell' \le t \le \ell-1} \|\ov x^{(t)} - \ov x^{(t+1)}\|_{2,0}))
\end{align*}
running time.
\end{itemize}
\end{theorem}

\begin{remark}[Snapshot]\label{rmk:snapshot}
In our data structures, we use persistent data structure techniques (e.g., \cite{driscoll1989making}) to keep a snapshot of the data structure after every update. This allows us to support query to historical data.
This incurs an $O(\log n_\lp)=\wt O(1)$ multiplicative factor in all running times, which we ignore in our analysis.
\end{remark}

\begin{proof}[Proof of Theorem~\ref{thm:batch_sketch}]
\textbf{Correctness}:
Correctness of \textsc{Update} follows from combining Lemma~\ref{lem:block_vector_sketch} and Lemma~\ref{lem:block_sketch_correct}. In the following we focus on correctness of queries.

Let us analyze \textsc{Query}$x$.
For $0\le t\le q$, define $y^{(t)} = H_{\ov x^{(t)}}^{1/2} x^{(t)}$,
$z^{(t)} = H_{\ov x^{(t)}}^{1/2} x^{(t+1)}$.
By Johnson-Lindenstrauss (Lemma~\ref{lem:jl_matrix}), for our choice of $r$, with probability $1-\delta_\apx/\poly(n_\lp)$, all sketches are $0.01$-accurate, i.e.,
\begin{align*}
    0.99 \le \frac{\|\Phi_{\chi(u)} y^{(\ell')}_{\chi(u)}-\Phi_{\chi(u)} z^{(\ell)}_{\chi(u)}\|_2}{\|y^{(\ell')}_{\chi(u)}-z^{(\ell)}_{\chi(u)}\|_2} \le 1.01
\end{align*}
for all $u\in V(\cS)$.

If for some $i\in [n]$, $\|y^{(\ell')}_i - z^{(\ell)}_i\|_2 \ge \epsilon$, then all its ancestors $u$ in the partition tree all satisfy $\|y^{(\ell')}_{\chi(u)} - z^{(\ell)}_{\chi(u)}\|_2 \ge \epsilon$ and $\|\Phi_{\chi(u)}(y^{(\ell')}_{\chi(u)} - z^{(\ell)}_{\chi(u)})\|_2 \ge 0.9\epsilon$.
So $i\in S$ with high probability.
This proves that $S \supseteq \{i \in [n]: \|y^{(\ell)}_i - z^{(\ell')}_i\|_2 \ge \epsilon\}$ with high probability.

Now let us bound the size of $S$. 
There are two types of elements in $S$.
The first type comes from changes in $\ov x$.
Let $T = \{i : i\in S, \ov x^{(t)}_i \ne \ov x^{(t+1)} \text{ for some } \ell'\le t \le \ell-1\}$.
Then $|T| \le \sum_{\ell' \le t \le \ell-1} \|\ov x^{(t)}-\ov x^{(t+1)}\|_{2,0}$.

The second type is blocks $i\in [n]$ where $\ov x$ does not change.
Define $U = S-T$. For all $i\in [n]$ with $\|y^{(\ell')}_i - z^{(\ell)}_i\|_2 \le 0.8 \epsilon$, with high probability $i\not \in S$.
So with high probability, all $i\in S$ satisfy $\|y^{(\ell')}_i - z^{(\ell)}_i\|_2 \ge 0.8\epsilon$. Then we have
\begin{align*}
    (0.8\epsilon)^2 |U| &\le \sum_{i\in U} \|z^{(\ell)}_i - y^{(\ell')}_i\|_2^2 \\
    &\le (\ell-\ell'+1) \sum_{\ell'\le t\le \ell} \sum_{i\in U} \|z^{(t)}_i - y^{(t)}_i\|_2^2 \\
    &\le (\ell-\ell'+1) \sum_{\ell'\le t\le \ell} \|z^{(t)} - y^{(t)}\|_2^2.
\end{align*}
So $|U| = O(\epsilon^{-2} \sum_{\ell'\le t\le \ell} \|z^{(t)} - y^{(t)}\|_2^2)$.

This proves that
\begin{align*}
    |S| = |T| + |U| \le O(\epsilon^{-2} \sum_{\ell'\le t\le \ell} \|z^{(t)} - y^{(t)}\|_2^2 + \sum_{\ell' \le t \le \ell-1} \|\ov x^{(t)}-\ov x^{(t+1)}\|_{2,0}).
\end{align*}

Proof of correctness of \textsc{Query}$s$ is similar and omitted.

\textbf{Running time:}

\textsc{Initialize}: Follows from Lemma~\ref{lem:block_vector_sketch} and Lemma~\ref{lem:block_sketch_init}.

\textsc{Update}: Follows from Lemma~\ref{lem:block_vector_sketch} and Lemma~\ref{lem:block_sketch_update}.

\textsc{Query}$x$: 
Let us bound the number of calls to $\textsc{Query$x$Sketch}$.
Let $C = \epsilon^{-2} \sum_{\ell'\le t\le \ell} \|z^{(t)} - y^{(t)}\|_2^2 + \sum_{\ell' \le t \le \ell-1} \|\ov x^{(t)}-\ov x^{(t+1)}\|_{2,0}$ for simplicity.
For every level $d$, there can be at most $O(C)$ vertices $u$ in level $d$ such that $\|y^{(\ell')}_{\chi(u)} - z^{(\ell)}_{\chi(u)}\|_2 \ge \epsilon/2$ by similar discussion as the correctness analysis, because all such $u$'s have disjoint $\chi(u)$.
So $L_d = O(C)$.
By our assumption that the partition tree has $\wt O(1)$ depth,
the total size of all $L_d$ satisfies $\sum_{d\ge 0} |L_d| = \wt O(C)$.
So the number of calls to $\textsc{Query$x$Sketch}$ is at most $\wt O(C)$.

Then the result follows from Lemma~\ref{lem:block_vector_sketch} and Lemma~\ref{lem:block_sketch_query}.

\textsc{Query}$s$: Proof is similar to \textsc{Query}$x$ and is omitted.
\end{proof}

\subsubsection{\textsc{BlockVectorSketch}}\label{sec:framework:block_vector_sketch}
In this section we present \textsc{BlockVectorSketch} (Algorithm~\ref{alg:block_vector_sketch}), data structure for maintaining a vector with block structure under sparse changes. This is generalization of \textsc{VectorSketch} in \cite{dly21}.

\begin{lemma}[
]\label{lem:block_vector_sketch}
Given a partition tree $(\cS,\chi)$ of $\R^{n}$, and a JL sketching matrix $\Phi \in \R^{r \times n_\lp}$, the data structure \textsc{BlockVectorSketch} (Algorithm~\ref{alg:block_vector_sketch}) maintains $\Phi_{\chi(v)} x_{\chi(v)}$ for all nodes $v$ in the partition tree implicitly through the following functions:
\begin{itemize}
    \item \textsc{Initialize}$(\cS, \chi,\Phi)$: Initializes the data structure in $O(r n_\lp)$ time.
    \item \textsc{Update}$(\delta_{x} \in \R^{n_\lp})$: Maintains the data structure for $x\leftarrow x + \delta_{x}$ in $O(r \|\delta_{x}\|_0\log n)$ time.
    \item \textsc{Query}$(v\in V(\cS))$: Outputs $\Phi_{\chi(v)}x_{\chi(v)}$ in $O(r \log n)$ time.
\end{itemize}
\end{lemma}
Note that in Theorem~\ref{lem:block_vector_sketch}, the running time does not depend on depth of the partition tree.
\begin{proof}
Correctness follows from from guarantees of segment trees.

Running time analysis is as follows.
\begin{itemize}
    \item \textsc{Initialize}: Computing $\Phi_{i} x_i$ for all $i\in [n]$ takes $O(r n_\lp)$ time. Afterwards, building $\cT$ takes $O(r n)$ time.
    \item \textsc{Update}: For every modified coordinate, it takes $O(r \log n)$ time to update $\cT$. So total time cost is $O(r \|\delta_{x}\|_0 \log n)$.
    \item \textsc{Query}: Takes $O(r\log n)$ time because of segment tree query time.
\end{itemize}
\end{proof}

\begin{algorithm}[!ht]\caption{This is used in Algorithm~\ref{alg:batch_sketch_init_update} and \ref{alg:batch_sketch_query}.
}\label{alg:block_vector_sketch}
\begin{algorithmic}[1]
\State {\bf data structure} \textsc{BlockVectorSketch} \Comment{Lemma~\ref{lem:block_vector_sketch}}
    \State {\bf private: members}
    \State \hspace{4mm} $\Phi \in \R^{r \times n_\lp}$
    \State \hspace{4mm} partition tree $( {\cal S}, \chi )$
    \State \hspace{4mm} $x \in \R^{n_\lp}$
    \State \hspace{4mm} Segment tree $\cT$ on $[n]$ with values in $\R^r$
    \State {\bf end members}
    \Procedure{Initialize}{${\cal S}, \chi: \text{partition tree}, \Phi \in \R^{r \times n_{\lp}}, x \in \R^{n_{\lp}}$}
        \State $({\cal S}, \chi) \gets ( {\cal S} , \chi)$
        \State $\Phi \gets \Phi$
        \State $x \gets x$
        \State Order leaves of $\cS$ (variable blocks) such that every node $\chi(v)$ corresponds to a contiguous interval $\subseteq [n]$.
        \State Build a segment tree $\cT$ on $[n]$ such that each segment tree interval $I\subseteq [n]$ maintains $\Phi_I x_I \in \R^r$.
    \EndProcedure
    
    \Procedure{Update}{$\delta_{x} \in \R^{n_{\lp}}$}
        \For{all $i\in [n_\lp]$ such that $\delta_{ x,{\text{coord}~i}} \ne 0$}
            \State Let $j\in [n]$ be such that $i$ is in $j$-th block
            \State Update $\cT$ at $j$-th coordinate $\Phi_j x_j \gets \Phi_j x_j + \Phi_{\text{coord}~i} \cdot \delta_{ x,{\text{coord}~i}}$.
            \State $x_{\text{coord}~i} \gets x_{\text{coord}~i} + \delta_{ x,{\text{coord}~i}}$
        \EndFor
    \EndProcedure
    \Procedure{Query}{$v \in V(\cS)$}
        \State Find interval $I$ corresponding to $\chi(v)$
        \State \Return range sum of $\cT$ on interval $I$
    \EndProcedure
\State {\bf end data structure}
\end{algorithmic}
\end{algorithm}

\subsubsection{\textsc{BlockBalancedSketch}} \label{sec:framework:block_balanced_sketch}
In this section we present \textsc{BlockBalancedSketch}, data structure for maintaining a sketch of a vector of form $\cW^\top h$, where $\cW = L_{\ov x}^{-1} A H_{\ov x}^{-1/2}$ and $h\in \R^{m_\lp}$ is a vector changing sparsely.
\begin{lemma}
\label{lem:block_balanced_sketch}
Given the constraint matrix $A$, a block elimination tree $\cT$ with height $\eta$, a JL matrix $\Phi \in \R^{r \times n_\lp}$, and a partition tree $(\cS, \chi)$ constructed as in Definition~\ref{defn:partition_tree_construction}
with height $\wt O(1)$, the data structure \textsc{BlockBalancedSketch} (Algorithm \ref{alg:block_balanced_sketch_1}, \ref{alg:block_balanced_sketch_2}, \ref{alg:block_balanced_sketch_3}), maintains $\Phi_{\chi(v)} (\cW^\top h)_{\chi(v)}$ for each $v\in V(\cS)$ through the following operations
\begin{itemize}
    \item \textsc{Initialize}$(\cS, \chi, \Phi, \ov x, h)$: Initializes the data structure in $$\wt O(T_H + T_n + T_L + r\cdot (n n_{\max}^2 + T_Z + \nnz(A)))$$ time.
    \item \textsc{Update}$(\delta_{\ov x}, \delta_{h})$: Updates all sketches in $\cS$ implicitly to reflect $(\cW, h)$ updating to $(\cW^\new, h^\new)$ in $$\wt O((T_{H,\max} + T_{\Delta L, \max} + \Tmat(n_{\max}) + r\eta^2 m_{\max}^2) \|\delta_{\ov x}\|_{2,0} + r m_{\max} \|\delta_h\|_{2,0})$$ time.
    \item \textsc{Query}$(v)$: Outputs $\Phi_{\chi(v)} (\cW^\top h)_{\chi(v)}$ in $\wt O(r \eta^2 m_{\max}^2)$ time.
\end{itemize}
\end{lemma}

\begin{algorithm}[!ht]\caption{This is used in Algorithm~\ref{alg:batch_sketch_init_update} and \ref{alg:batch_sketch_query}.
}  \label{alg:block_balanced_sketch_1}
\begin{algorithmic}[1]
\State {\bf data structure} \textsc{BlockBalancedSketch} \Comment{Lemma~\ref{lem:block_balanced_sketch}}
\State {\bf private: members}
    \State \hspace{4mm} $\Phi \in \R^{r \times n_\lp}$
    \State \hspace{4mm} Partition tree $(\cS, \chi)$ with balanced binary tree ${\cal B}$
    \State \hspace{4mm} $t \in \mathbb{N}$
    \State \hspace{4mm} $h \in \R^{m_\lp}$, $\ov{x} \in \R^{n_\lp}$, $H_{\ov x} \in \R^{n_\lp \times n_\lp}$
    \State \hspace{4mm} $\{ L[t] \in \R^{m_\lp \times m_\lp} \}_{t\geq 0}$
    \State \hspace{4mm} $\{ J_v \in \R^{r \times m_\lp} \}_{v\in \cS}$
    \State \hspace{4mm} $\{ Z_v \in \R^{r \times m_\lp} \}_{v\in \cB}$
    \State \hspace{4mm} $\{ y_v^{\triangledown} \in \R^r \}_{v\in \cB}$ 
    \State \hspace{4mm} $\{ t_v \in \mathbb{N} \}_{v\in \cB}$
\State {\bf end members}
\Procedure{Initialize}{${\cal S}, \chi: \text{partition tree}, \Phi \in \R^{r \times n_{\lp}}, \ov{x} \in \R^{n_{\lp}}, h \in \R^{m_{\lp}}$} \Comment{Lemma~\ref{lem:block_sketch_init}}
    \State $({\cal S}, \chi ) \gets ( {\cal S}, \chi )$
    \State $\Phi \gets \Phi$
    \State $t \gets 0$, $h \gets h$
    \State Compute $H_{\ov x} \gets \nabla^2 \phi( \ov{x} )$
    \State Find lower Cholesky factor $L_{\ov x}[t]$ of $A H_{\ov x}^{-1} A^\top$
    \For{all $v \in {\cal S}$}
        \State $J_v \gets \Phi_{ \chi(v) } H_{\ov x}^{-1/2} A^\top $ 
    \EndFor 
    \For{all $v \in {\cal B}$}
        \State $Z_v \gets J_v L_{\ov x}[t]^{-\top}$
        \State $y_v^{\triangledown} \gets Z_v ( I - I_{\Lambda(v)} ) h$
        \State $t_v \gets t$
    \EndFor 
\EndProcedure
\Procedure{Query}{$v \in {\cal S}$} \Comment{Lemma~\ref{lem:block_sketch_query}}
    \If{$v \in {\cal S} \backslash {\cal B}$}
        \State \Return $J_v \cdot L_{\ov x}[t]^{-\top} h$
    \EndIf
    \State $\Delta_{L_{\ov x}} \gets ( L_{\ov x}[t] - L_{\ov x}[t_v] ) \cdot I_{\Lambda(v)} $
    \State $\delta_{Z_v} \gets -( L_{\ov x}[t]^{-1} \cdot \Delta_{L_{\ov x}} \cdot Z_v^\top )^\top$
    \State $Z_v \gets Z_v + \delta_{Z_v}$
    \State $\delta_{y_v^{\triangledown}} \gets \delta_{Z_v} \cdot (I - I_{\Lambda(v)}) h$
    \State $y_v^{\triangledown} \gets y_v^{\triangledown} + \delta_{y_v^{\triangledown}}$
    \State $t_v \gets t$
    \State $y_v^{\triangleup} \gets Z_v \cdot I_{\Lambda(v)} \cdot h$
    \State \Return $y_v^{\triangleup} + y_v^{\triangledown}$
\EndProcedure 
\State {\bf end data structure}
\end{algorithmic}
\end{algorithm}

\begin{algorithm}[!ht]\caption{
\textsc{BlockBalancedSketch} Algorithm~\ref{alg:block_balanced_sketch_1} continued.
This is used in Algorithm~\ref{alg:batch_sketch_init_update} and \ref{alg:batch_sketch_query}.
} \label{alg:block_balanced_sketch_2}
\begin{algorithmic}[1]
\State {\bf data structure} \textsc{BlockBalancedSketch} \Comment{Lemma~\ref{lem:block_balanced_sketch}}
\Procedure{Update}{$ \delta_{\ov{x}} \in \R^{n_{\lp}} , \delta_{h} \in \R^{m_{\lp}}$}
    \For{$i \in [n]$ where $\delta_{\ov x,i} \ne 0$}
        \State \textsc{Update$\ov x$}$(\delta_{\ov x,i})$
    \EndFor 
    \For{all $\delta_{h,i} \ne 0$} \label{line:alg:block_balanced_sketch_2_updateh_start}
        \State $v \gets \Lambda^{\circ}(i)$
        \For{all $u \in {\cal P}^{\cal B}(v)$}
            \State $y_u^{\triangledown} \gets y_v^\triangledown + Z_u \cdot I_{ \{ i\} } \cdot \delta_{h}$
        \EndFor
    \EndFor 
    \label{line:alg:block_balanced_sketch_2_updateh_end}
    \State $h \gets h + \delta_{h}$
\EndProcedure 
\Procedure{Update$\ov x$}{$ \delta_{\ov{x},i} \in \R^{n_i} $} \Comment{Lemma~\ref{lem:block_sketch_update_block}}  
    \State $t \gets t+1$ 
    \State $\ov{x}_i \gets \ov x_i + \delta_{\ov{x},i}$
    \State $\Delta_{H_{\ov x}, (i,i)} \gets \nabla^2 \phi_i(\ov{x}_i) - H_{\ov x,(i,i)}$
    \State Find $\Delta_{L_{\ov x}}$ such that $L_{\ov x}[t] \gets L_{\ov x}[t-1] + \Delta_{L_{\ov x}}$ is the lower Cholesky factor of $A (H_{\ov x} + \Delta_{ H_{\ov x}})^{-1}A^\top$
    \State $S \gets {\cal P}^{\cal B} ( \Lambda^{\circ} ( \mathsf{low}^\cT(i) ) ) $
    \State $\textsc{Update$L$}(S, \Delta_{L_{\ov x}}) $
    \State $\textsc{Update$H$}(i, \Delta_{H_{\ov x}, (i,i)})$
\EndProcedure 
\State {\bf end data structure}
\end{algorithmic} 
\end{algorithm} 

\begin{algorithm}[!ht]\caption{
\textsc{BlockBalancedSketch} Algorithm~\ref{alg:block_balanced_sketch_1}, \ref{alg:block_balanced_sketch_2} continued.
This is used in Algorithm~\ref{alg:batch_sketch_init_update} and \ref{alg:batch_sketch_query}.
}\label{alg:block_balanced_sketch_3} 
\begin{algorithmic}[1]
\State {\bf data structure} \textsc{BlockBalancedSketch} \Comment{Lemma~\ref{lem:block_balanced_sketch}}
\State {\bf private:}
\Procedure{Update$L$}{$S \subset {\cal B}$, $\Delta_{L_{\ov x}} \in \R^{m_\lp \times m_\lp}$} \Comment{Lemma~\ref{lem:block_sketch_update_L}}
    \For{all $v \in S$}
        \State $\delta_{Z_v} \gets - ( L_{\ov x}[t-1]^{-1} ( L_{\ov x}[t-1] - L_{\ov x}[t_v] ) \cdot I_{\Lambda(v)} \cdot Z_v^\top )^\top$
        \label{line:alg:block_balanced_sketch_3_updateZv1}
        \State $\delta_{Z_v}'\gets - ( L_{\ov x}[t]^{-1} \cdot \Delta_{L_{\ov x}} \cdot (Z_v+\delta_{Z_v})^\top )^\top$
        \label{line:alg:block_balanced_sketch_3_updateZv2}
        \State $Z_v \gets Z_v +\delta_{Z_v}+\delta_{Z_v}'$
        \State $\delta_{y_v^{\triangledown}} \gets (\delta_{Z_v}+\delta_{Z_v}')(I - I_{\Lambda(v)}) h$
        \State $y_v^{\triangledown} \gets y_v^{\triangledown} + \delta_{y_v^{\triangledown}}$
        \State $t_v \gets t$
    \EndFor
\EndProcedure 
\State {\bf private:}
\Procedure{Update$H$}{$i\in [n], \Delta_{H_{\ov x}, (i,i)} \in \R^{n_i\times n_i}$} \Comment{Lemma~\ref{lem:block_sketch_update_H}}
    \State Find $u$ such that $\chi(u) = \{i\}$
    \State $\Delta_{H_{\ov x}^{-1/2}, (i,i)} \gets (H_{\ov x,(i,i)} + \Delta_{H_{\ov x},(i,i)})^{-1/2} - H_{\ov x,(i,i)}^{-1/2}$
    \State $\delta_{J_u} \gets \Phi_{i} \cdot \Delta_{H_{\ov x}^{-1/2}, (i,i)} \cdot A^\top$
    \For{all $v \in {\cal P}^{\cS} (u)$}
        \State $J_v \gets J_v + \delta_{J_u} $
        \If{$v \in {\cal B}$}
            \State $\delta_{Z_v} \gets \delta_{J_v} \cdot L_{\ov x}[t_v]^{-\top}$
            \State $Z_v \gets Z_v + \delta_{Z_v}$
            \State $\delta_{y_v^{\triangledown}} \gets \delta_{Z_v} \cdot (I - I_{\Lambda(v)}) \cdot h$
            \State $y_v^{\triangledown} \gets y_v^{\triangledown} + \delta_{y_v^{\triangledown}}$
        \EndIf 
    \EndFor 
    \State $H_{\ov x} \gets H_{\ov x} + \Delta_{ H_{\ov x}, (i,i)}$
\EndProcedure
\State {\bf end data structure}
\end{algorithmic}
\end{algorithm}

The algorithm is a block version of \cite[Section 6.6.1]{dly21}. We nevertheless present a full proof for completeness.

Let us first describe the main idea of the algorithm.
Note that in \textsc{BlockBalancedSketch}, we have the freedom of choosing the partition tree $(\cS, \chi)$. Then this tree is used by \textsc{BatchSketch} and \textsc{BlockVectorSketch}.
Therefore we can choose a partition tree which works well for our purpose.

Let us consider the operations \textsc{BlockBalancedSketch} needs to support.
It needs to support updating $\ov x$ and $h$, and answering queries on a subtree $\chi(v)$.
Change of one block in $\ov x$ leads to change of one path in the block elimination tree $\cT$, and change of one block in $h$ leads to change of one subtree in $\cT$.
Therefore we essentially would like a data structure which supports subtree and path updates and subtree queries.
With this in mind, it is natural to use heavy-light decomposition \cite{sleator1981data}.
\begin{lemma}[Heavy-Light Decomposition~\cite{sleator1981data}] \label{lem:heavy_light}
Given a rooted tree $\cT$ with $m$ vertices, we can construct in $O(m)$ time an ordering $\pi$ of the vertices such that (1) every path in $\cT$ can be decomposed into $O(\log m)$ contiguous subseqeuences under $\pi$, and (2) every subtree in $\cT$ is a single contiguous subsequence under $\pi$.
\end{lemma}

\begin{definition}[Construction of Partition Tree] \label{defn:partition_tree_construction}
We fix an ordering $\pi$ of $[m]$ using the heavy-light decomposition (Lemma~\ref{lem:heavy_light}).
We construct complete binary tree $\cB$ with leaves $[m]$ and ordering $\pi$.

To get a partition tree, we need to add leaves $[n]$ to $\cB$.
For every coordinate $i\in [n]$, let $\mathsf{low}^\cT(i)$ be any vertex $v$ in $\cT$ such that the support of $A_{*,i}$ is contained in $\cP^\cT(v)$. (Recall that for a block elimination tree, support of $A_{*,i}$ is contained in a path for any $i\in [n]$.)
For any $j\in [m]$, we construct a complete binary tree with leaves $\{i\in [n]: \mathsf{low}^\cT(i) = j\}$ and hang this tree under leaf $j$ in $\cB$.
This finishes the construction of a partition tree $(\cS, \chi)$.
\end{definition}
As a sanity check, height of $\cB$ is $O(\log m)$, and height of each subtree under leaf $j\in \cB$ is $O(\log n)$.
So height of $\cS$ is $O(\log m + \log n) = O(\log n_\lp) = \wt O(1)$.
Also, from the definition, we see that construction of the partition tree is cheap, i.e., takes $O(\nnz(A))$ time.

The following definitions come from \cite{dly21}.
\begin{definition}
We make the following definitions. For $v\in \cB$, define
\begin{align*}
\Lambda(v) := \left (\bigcup_{i\in \ov \chi(v)} \cP^\cT(i) \right)\cap \left(\bigcup_{i\in \cT \backslash \ov \chi(v)} \cP^\cT(i)\right),\\
\ov \Lambda(v) := \left (\bigcup_{i\in \ov \chi(v)} \cP^\cT(i)\right)\backslash \left (\bigcup_{i\in \cT \backslash \ov \chi(v)} \cP^\cT(i)\right),
\end{align*}
where $\ov \chi(v)$ is the set of leaves in $\cB$ which are descendants of $v$.

For $u\in \cT$, define
$\Lambda^\circ(u)$ be the lowest vertex $v\in \cB$ such that $u\in \ov \Lambda(v)$.
In other words, $\Lambda^\circ(u)$ is the lowest vertex $v\in \cB$ such that $\ov \chi(v)$ contains $\cD^\cT(u)$ (set of descendants of $u$ in $\cT$). Therefore $\Lambda^\circ(u)$ is well-defined.
\end{definition}

\begin{lemma} \label{lem:block-sketch-lambda-size}
For any $v$, $\Lambda(v)$ is contained in the union of two paths in $\cT$.
In particular, $|\Lambda(v)| = O(\eta)$.
\end{lemma}
\begin{proof}
The order $\pi$ in Lemma~\ref{lem:heavy_light} is a pre-order traversal of $\cT$. Let $u$ be the last vertex before $\ov \chi(v)$ under $\pi$, and $w$ be the first vertex after $\ov \chi(v)$ under $\pi$. Then $\Lambda(v)$ is contained in $\cP^\cT(u) \cup \cP^\cT(w)$.
\end{proof}

\begin{lemma} \label{lem:block_sketch_correct}
\textsc{BlockBalancedSketch} correctly maintains a sketch of $\cW^\top h$, and all query results are returned correctly.
\end{lemma}
\begin{proof}
We prove that the following invariant always holds after every call to \textsc{BlockBalancedSketch}.
\begin{align}
    J_v &= \Phi_{\chi(v)} H_{\ov x}^{-1/2} A^\top \qquad \forall v\in \cS, \label{eqn:lem:block_sketch_invariant_Jv}\\
    Z_v &= J_v L_{\ov x}[t_v]^{-\top} \qquad \forall v\in \cB,
    \label{eqn:lem:block_sketch_invariant_Zv}\\
    y_v^\triangledown &= Z_v (I-I_{\Lambda(v)}) h \qquad \forall v\in \cB,
    \label{eqn:lem:block_sketch_invariant_yv}\\
    0 &= (L[t] - L[t_v]) \cdot I_{\ov \Lambda(v)} \qquad \forall v\in \cB.
    \label{eqn:lem:block_sketch_invariant_Ltv}
\end{align}

\textsc{Initialize}: The invariants are clearly satisfied after initialization.

\textsc{Query}: If $v\in \cS \backslash \cB$ then we compute $J_v L_{\ov x}[t]^{-\top} h$ directly and the result is correct.

Now assume $v\in \cB$. We update $t_v \gets t$, and update $Z_v$ and $y_v^\triangledown$ accordingly. We have
Note that
\begin{align*}
\delta_{Z_v} &= -(L_{\ov x}[t]^{-1} (L_{\ov x}[t] - L_{\ov x}[t_v]) I_{\Lambda(v)} Z_v^\top)^\top \\
&= -(L_{\ov x}[t]^{-1} (L_{\ov x}[t] - L_{\ov x}[t_v]) (I-I_{\ov \Lambda(v)}-I_{\cT \backslash \Lambda(v) \backslash \ov \Lambda(v)} Z_v^\top)^\top \\
&= -(L_{\ov x}[t]^{-1} (L_{\ov x}[t] - L_{\ov x}[t_v]) Z_v^\top)^\top
\end{align*}
because Invariant~\eqref{eqn:lem:block_sketch_invariant_Ltv} and column sparsity of $Z_v$.
So
\begin{align*}
    Z_v^\new &= Z_v + \delta_{Z_v} \\
    &= Z_v - (L_{\ov x}[t]^{-1} (L_{\ov x}[t] - L_{\ov x}[t_v]) Z_v^\top)^\top\\
    &= J_v L_{\ov x}[t_v]^{-\top} - J_v (L_{\ov x}[t_v]^{-\top} - L_{\ov x}[t]^{-\top}) \\
    &= J_v L_{\ov x}[t]^{-\top}.
\end{align*}
So Invariant~\eqref{eqn:lem:block_sketch_invariant_Zv} is satisfied.
Updating $y_v^\triangledown$ ensures that Invariant~\eqref{eqn:lem:block_sketch_invariant_yv} is satisfied.
Finally, the return value is correct because of definition of $y_v^\triangledown$ and $y_v^\triangleup$.

\textsc{Update}: We divide the proof into several steps.
Correctness of \textsc{Update$x$} follows from correctness of \textsc{Update$L$} and \textsc{Update$H$} (which we will prove below).

Correctness of \textsc{Update$H$}:
We update $t_v$, $Z_v$ and $y_v^\triangledown$ for $v\in S = \cP^\cB(\Lambda^\circ(\mathsf{low}^\cT(i)))$.
In other words, $S$ is the set of all vertices $v$ with $\mathsf{low}^\cT(i) \in \ov \Lambda(v)$.
For any $v\not \in S$, $L[t] \cdot I_{\ov \Lambda(v)}$ is not changed. So Invariant~\eqref{eqn:lem:block_sketch_invariant_Ltv} is preserved for $v\not \in S$.

Fix $v\in S$.
In Algorithm~\ref{alg:block_balanced_sketch_3}, Line~\ref{line:alg:block_balanced_sketch_3_updateZv1}, we update $t_v$ to $t-1$.
In Algorithm~\ref{alg:block_balanced_sketch_3}, Line~\ref{line:alg:block_balanced_sketch_3_updateZv2}, we update $t_v$ to $t$.
By a similar computation as the one we did for \textsc{Query}, $Z_v$ is updated correctly (i.e., Invariant~\eqref{eqn:lem:block_sketch_invariant_Zv} is preserved).
This implies $y_v^\triangledown$ is updated correctly (i.e., Invariant~\eqref{eqn:lem:block_sketch_invariant_yv} is preserved).

Correctness of \textsc{Update$L$}:
We update $J_v$, $Z_v$, $y_v^\triangledown$ for all $v\in \cP^\cS(u)$ (where $\chi(u) = \{i\}$).
Invariant~\eqref{eqn:lem:block_sketch_invariant_Ltv} is preserved because $t_v$ does not change.
Invariant~\eqref{eqn:lem:block_sketch_invariant_Jv}, \eqref{eqn:lem:block_sketch_invariant_Zv}, \eqref{eqn:lem:block_sketch_invariant_yv} are preserved by our choice of $\delta_{J_v}$, $\delta_{Z_v}$, $\delta_{y_v^\triangledown}$.

Correctness of Algorithm~\ref{alg:block_balanced_sketch_2}, Line~\ref{line:alg:block_balanced_sketch_2_updateh_start} to Line~\ref{line:alg:block_balanced_sketch_2_updateh_end}: This part is ``\textsc{Update$h$}''.
For $i\in [m]$ with $\delta_{h,i} \ne 0$, we update $y_u^\triangledown$ for $u\in \cB$ such that $i\in \ov \Lambda(u)$. Recall that $\cP^\cB(\Lambda^\circ(i))$ contains all vertices $u$ such that $i\in \ov \Lambda(u)$. So Invariant~\eqref{eqn:lem:block_sketch_invariant_yv} is satisfied.
\end{proof}

\begin{lemma} \label{lem:block_sketch_init}
\textsc{BlockBalancedSketch.Initialize}  (Algorithm~\ref{alg:block_balanced_sketch_1}) costs 
\begin{align*} 
\wt O(T_H + T_n + T_L + r\cdot (n n_{\max}^2 + T_Z + \nnz(A)))
\end{align*}
time.
\end{lemma}
\begin{proof}
Computing $H_{\ov x}$ takes $T_H$ time.
Computing $H_{\ov x}^{-1}$, $H_{\ov x}^{-1/2}$ takes $O(T_n)$ time.
Computing $L_{\ov x}[t]$ takes $T_L$ time.

Computation of $J_v$:
For all $i\in [n]$, computing $\Phi_{i} H_{\ov x,(i,i)}^{-1/2}$ takes $r n_i^2$ time.
Support of $\Phi_{i} H_{\ov x,(i,i)}^{-1/2}$ are all disjoint, so computing all $\Phi_{i} H_{\ov x,(i,i)}^{-1/2} A^\top$ takes $r \cdot \nnz(A)$ time.
Then we can compute all $J_v$ by summing from bottom to up.
Height of the tree is $\wt O(1)$, so every non-zero entry of $\Phi_{i} H_{\ov x,(i,i)}^{-1/2} A^\top$ gets propagated $\wt O(1)$ times.
So computing $J_v$ takes 
\begin{align*} 
\wt O(r (\sum_{i\in [n]} n_i^2 + \nnz(A))) = \wt O(r (n n_{\max}^2 + \nnz(A))) 
\end{align*}
time in total.

Computation of $Z_v$:
To compute $Z_v$, we first compute $Z_v$ for all leaves $v\in \cB$.
This takes $r T_Z$ time by assumption (Definition~\ref{defn:lp_general_param}).
Then we sum from bottom to up to compute $Z_v$ for all $v\in \cB$.
Because height of the partition tree is $\wt O(1)$, every non-zero entry in the leaves gets propagated $\wt O(1)$ times.
So computing $Z_v$ takes $\wt O(r T_Z)$ time in total.

Computation of $y_v^\triangledown$: Takes $O(\nnz(Z)) = \wt O(rT_Z)$ time.

Summing everything up we get the desired running time.
\end{proof}

\begin{lemma} \label{lem:block_sketch_update}
\textsc{BlockBalancedSketch.Update} (Algorithm~\ref{alg:block_balanced_sketch_2}) costs $$\wt O((T_{H,\max} + T_{\Delta L, \max} + \Tmat(n_{\max}) + r\eta^2 m_{\max}^2) \|\delta_{\ov x}\|_{2,0} + r m_{\max} \|\delta_h\|_{2,0})$$ time.
\end{lemma}
\begin{proof}
Each call to \textsc{Update$\ov x$} costs $$\wt O(T_{H,\max} + T_{\Delta_L, \max} + \Tmat(n_{\max})+ r \eta^2 m_{\max}^2)$$ time by Lemma~\ref{lem:block_sketch_update_block}.

For each $i\in [m]$ with $\delta_{h,i}\ne 0$ and each $v$, it takes $O(r m_{\max})$ time to update $y_u^\triangledown$.
So the total time needed to update $y_u^\triangledown$ is $O(r m_{\max} \|\delta_h\|_{2,0})$.
\end{proof}

\begin{lemma} \label{lem:block_sketch_update_block}
\textsc{BlockBalancedSketch.Update$\ov x$} (Algorithm~\ref{alg:block_balanced_sketch_2}) costs $$\wt O(T_{H,\max} + T_{\Delta_L, \max} + \Tmat(n_{\max})+ r \eta^2 m_{\max}^2)$$ time.
\end{lemma}
\begin{proof}
Computing $\Delta_{H_{\ov x}}$ takes $O(T_{H,i})$ time.
Computing $\Delta_{L_{\ov x}[t]}$ takes $O(T_{\Delta_L,i})$ time.
Because $S$ lies on a path in $\cB$, we have $|S| = \wt O(1)$.
By Lemma~\ref{lem:block_sketch_update_H} and Lemma~\ref{lem:block_sketch_update_L}, updating $L$ costs $O(r \eta^2 m_{\max}^2 \log n)$ time
and updating $H$ costs $O( \Tmat(n_{\max}) + r \eta^2 m_{\max}^2 \log n)$ time.
\end{proof}

\begin{lemma} \label{lem:block_sketch_update_L}
\textsc{BlockBalancedSketch.Update$L$} (Algorithm~\ref{alg:block_balanced_sketch_3}) costs $O(r \eta^2 m_{\max}^2)$ time.
\end{lemma}
\begin{proof}
By Lemma~\ref{lem:block-sketch-lambda-size}, we have $|\Lambda(v)| = O(\eta)$.
So we can compute $(L_{\ov x}[t-1]-L_{\ov x}[t_v]) \cdot I_{\Lambda(v)}$ in $O(\eta^2 m_{\max}^2)$ time by Lemma~\ref{lem:mat_vec_mult_time}\ref{item:lem_mat_vec_mult_time_Lv_sparse}.
Then computing $(L_{\ov x}[t-1]-L_{\ov x}[t_v]) \cdot I_{\Lambda(v)} \cdot Z_v^\top$ takes $O(r \eta^2 m_{\max}^2)$ time.
Finally, computing $\delta_{Z_v} = L_{\ov x}[t-1]^{-1} \cdot (L_{\ov x}[t-1]-L_{\ov x}[t_v]) \cdot Z_v^\top$ takes
$O(r \eta^2 m_{\max}^2)$ time by Lemma~\ref{lem:mat_vec_mult_time}\ref{item:lem_mat_vec_mult_time_Linvv_path}.
Analysis for $\delta_{Z_v}'$ is the same.

Computing $\delta_{y_v}^\triangledown$ takes $O(r \eta m_{\max})$ time by sparsity pattern of $\delta_{Z_v}+\delta_{Z_v}'$.

Summing everything up we get the desired running time.
\end{proof}

\begin{lemma} \label{lem:block_sketch_update_H}
\textsc{BlockBalancedSketch.Update$H$} (Algorithm~\ref{alg:block_balanced_sketch_3}) costs $\wt O(\Tmat(n_{\max})+ r \eta^2 m_{\max}^2)$ time.
\end{lemma}
\begin{proof}
Computing $\Delta_{H_{\ov x}^{-1/2}}$ takes $\Tmat(n_i)$ time.
Computing $\delta_{J_u}$ takes $O(r \eta m_{\max}^2)$ by sparsity pattern of $A$.
Furthermore, (each row of) $\delta_{J_u}$ is supported on a path in $\cT$.
Therefore, for every $v$, computing $\delta_{Z_v}$ takes $O(r \eta^2 m_{\max}^2)$ time by Lemma~\ref{lem:mat_vec_mult_time}\ref{item:lem_mat_vec_mult_time_Linvv_path}.
Afterwards, computing $\delta_{y_v^\triangledown}$ takes $O(r \eta m_{\max})$ time by sparsity pattern of $\delta_{Z_v}$.
Finally, $|\cP^\cS(u)| = \wt O(1)$, so the overall running time is $\wt O(\Tmat(n_{\max}) + r \eta^2 m_{\max}^2)$.
\end{proof}

\begin{lemma} \label{lem:block_sketch_query}
\textsc{BlockBalancedSketch.Query} (Algorithm~\ref{alg:block_balanced_sketch_1}) takes $O(r \eta^2 m_{\max}^2)$ time.
\end{lemma}
\begin{proof}
If $v\in \cS \backslash{\cB}$, then (each row of) $J_v$ is supported on a path.
So computing  $J_v L_{\ov x}[t]^{-\top} h$ takes $O(r \eta^2 m_{\max}^2)$ time by Lemma~\ref{lem:mat_vec_mult_time}\ref{item:lem_mat_vec_mult_time_Linvv_path}.

Now suppose $v\in \cB$.
By Lemma~\ref{lem:block-sketch-lambda-size}, we have $|\Lambda(v)| = O(\eta)$.
So we can compute $(L_{\ov x}[t]-L_{\ov x}[t_v]) \cdot I_{\Lambda(v)}$ in $O(\eta^2 m_{\max}^2)$ time by Lemma~\ref{lem:mat_vec_mult_time}\ref{item:lem_mat_vec_mult_time_Lv_sparse}.
Then computing $\Delta_{L_{\ov x}} \cdot Z_v^\top$ takes $O(r \eta^2 m_{\max}^2)$ time.
Furthermore, $\Delta_{L_{\ov x}}$ has columns supported on two paths.
Therefore, computing $L_{\ov x}[t]^{-1} \cdot \Delta_{L_{\ov x}} \cdot Z_v^\top$ takes
$O(r \eta^2 m_{\max}^2)$ time by Lemma~\ref{lem:mat_vec_mult_time}\ref{item:lem_mat_vec_mult_time_Linvv_path}.
So the total time needed to compute $\delta_{Z_v}$ is $O(r\eta^2 m_{\max}^2)$.

Computing $y_v^\triangledown$ takes $O(r \eta m_{\max})$ time by sparsity pattern of $Z_v$.
Computing $y_v^\triangleup$ takes $O(r \eta m_{\max})$ time because $|\Lambda(v)| = O(\eta)$.

Summing everything up we get the desired running time.
\end{proof}

Combining everything we finish the proof of Lemma~\ref{lem:block_balanced_sketch}.
\begin{proof}[Proof of Lemma~\ref{lem:block_balanced_sketch}]
Combining Lemma~\ref{lem:block_sketch_correct}, Lemma~\ref{lem:block_sketch_init}, Lemma~\ref{lem:block_sketch_update}, Lemma~\ref{lem:block_sketch_query}.
\end{proof}

\subsection{Analysis of \textsc{CentralPathMaintenance}}\label{sec:framework:cpm_analysis}
The goal of this section is to prove Theorem~\ref{thm:central_path_maintenance_general}.

We first prove correctness of \textsc{CentralPathMaintenance}.
\begin{lemma}[Correctness of \textsc{CentralPathMaintenance}]\label{lem:cpm_correct_general}
Algorithm~\ref{alg:cpm} implicitly
    \begin{align*}
        x &\leftarrow x + H_{\ov{x}}^{-1/2}(I-P_{\ov{x}}) H_{\ov{x}}^{-1/2} \delta_\mu(\ov{x}, \ov{s}, \ov{t})\\
        s &\leftarrow s + t \cdot  H_{\ov{x}}^{1/2} P_{\ov{x}} H_{\ov{x}}^{-1/2} \delta_\mu(\ov{x}, \ov{s}, \ov{t})
    \end{align*}
    where $H_{\ov{x}} := \nabla^2 \phi(\ov{x}) \in \R^{ n_{\lp} \times n_{\lp} }$, $P_{\ov{x}} := H_{\ov{x}}^{-1/2} A^\top (A H_{\ov{x}}^{-1} A^\top)^{-1} A H_{\ov{x}}^{-1/2} \in \R^{n_{\lp} \times n_{\lp}}$,
    and $\ov{t}$ is some earlier timestamp satisfying $|t-\ov{t}| \le \epsilon_t \cdot \ov{t}$.
    
    It also explicitly maintains $(\ov{x}, \ov{s}) \in \R^{n_{\lp} \times n_{\lp}}$ such that $\|\ov{x}_i-x_i\|_{\ov{x}_i} \le \ov{\epsilon}$ and $\|\ov{s}_i-s_i\|^*_{\ov{x}_i} \le t \ov{\epsilon} w_i$ for all $i\in [n]$ with probability at least $0.9$.
\end{lemma}

\begin{proof}
We correctly maintain a multiscale representation of $(x,s)$ because of correctness of $\mathsf{exact}.\textsc{Update}$ (Lemma~\ref{lem:exactds_correctness}).

We show that $\| \ov x_i - x_i\|_{\ov x_i} \le \ov \epsilon$ and $\|\ov s_i - s_i\|^*_{\ov x_i} \le t \ov \epsilon w$ for all $i\in [n]$ (c.f.~Algorithm~\ref{alg:robust_ipm_centering_data_structure}, Line~\ref{line:cpm_guarantee}).
$\mathsf{approx}$ maintains an $\ell_{2,\infty}$-approximation of
$$H_{\ov x}^{1/2} x = H_{\ov x}^{1/2} \wh x+ \beta_x c_x -\cW^\top (\beta_x h+\epsilon_x).$$

For $\ell\le q$, we have
\begin{align*}
\|H_{\ov x}^{1/2} x^{(\ell+1)} - H_{\ov x}^{1/2} x^{(\ell)}\|_2 &= 
\|H_{\ov x}^{1/2} (x^{(\ell+1)} - x^{(\ell)})\|_2  \\
& = w^{-1/2} \|\delta_x^{(\ell)}\|_{\ov x} \\
& \le \frac{9}{8} \alpha w^{-1/2} \\
& \le  \zeta_{x},
\end{align*}
where the second step follows from definition of $\|\cdot\|_{\ov x}$, the third step follows from Lemma \ref{lemma:DLYA.9}, and the last step follows from our choice of $\zeta_{x}$.

By Theorem \ref{thm:approxds}, with probability at least $1-\delta_{\apx}$, 
$\mathsf{approx}$ correctly maintains $\wt x$ such that
\begin{align*}
\|\wt x - H_{\ov{x}}^{1/2} x\|_{2,\infty} \le \epsilon_{\apx,x} = \ov \epsilon.
\end{align*}

We set $\ov x = H_{\ov x}^{-1/2} \wt x$, so
\begin{align} \label{eqn:lem_correct_cpm_1}
\|H_{\ov x}^{1/2}(\ov x - x)\|_{2,\infty} \le \ov \epsilon.
\end{align}
Therefore
\begin{align*}
\| \ov x_i - x_i \|_{\ov x_i} 
= & ~ \| H_{\ov x,(i,i)}^{1/2} (\ov x_i - x_i) \|_2 \\
\le & ~ \| H_{\ov x}^{1/2} (\ov x - x) \|_{2,\infty} \\
\le & ~ \ov \epsilon.
\end{align*}
where the first step follows from definition of $\|\cdot\|_{\ov x_i}$, the second step follows from definition of $\|\cdot\|_{2,\infty}$, the third step follows from \eqref{eqn:lem_correct_cpm_1}.

The proof for $s$ is similar.
We have
\begin{align*}
\|H_{\ov x}^{-1/2} \delta_s^{(\ell)}\|_2 \le w^{1/2} \|\delta_s^{(\ell)}\|_{\ov x}^* \le \frac {9}{8} \alpha \ov t w^{1/2} \le \zeta_s,
\end{align*}
and
\begin{align*}
\|H_{\ov x}^{-1/2} (s-\ov s)\|_{2,\infty} \le \epsilon_{\apx,s} = \frac {1}{2} \ov \epsilon \cdot \ov t \cdot w.
\end{align*}
\end{proof}

Now we prove running time claims in Theorem~\ref{thm:central_path_maintenance_general}.
\begin{lemma}[\textsc{Initialize} part of Theorem~\ref{thm:central_path_maintenance_general}]\label{lem:init_time_general}
\textsc{CentralPathMaintenance.Initialize} (Algorithm~\ref{alg:cpm}) takes $\wt O(T_H + T_n + T_L + T_Z + \nnz(A) + \eta m_\lp m_{\max})$ time.
\end{lemma}
\begin{proof} 
By Theorem~\ref{thm:exactds} and Theorem~\ref{thm:approxds}.
\end{proof}

\begin{lemma}[\textsc{MuliplyAndMove} part of Theorem~\ref{thm:central_path_maintenance_general}]\label{lem:multiply_and_move_time_general}
Total running time of \textsc{MultiplyAndMove} (Algorithm~\ref{alg:cpm}) is
\begin{align*}
    & \wt{O} ( n^{0.5} \nu_{\max}^{0.5} \cdot (T_H+T_L+\eta T_m + T_Z)^{0.5} \\
&\cdot (\Tmat(n_{\max}) + T_{\Delta_L,\max} + T_{H,\max} + \eta^2 m_{\max}^2)^{0.5} \log(t_{\max}/t_{\min})).
\end{align*}
\end{lemma}
\begin{proof}
The choice of parameters under modification of $w$ is summarized in Table \ref{tab:internal_param_general}.

\begin{table}[!ht]
    \centering
    \begin{tabular}{|l|l|l|l|} \hline 
        {\bf Notation} & {\bf Range} & {\bf Meaning} & {\bf Choice} \\ \hline
        $w$ & $\mathbb{R}_{\ge 1}$ & weight & $\nu_{\max}$ \\ \hline
        $N$ & $\mathbb{N}_+$ & number of central path steps & $\sqrt{n \nu_{\max} w}$ \\ \hline
        $q$ & $\mathbb{N}_+$ & number of steps before restart & See Eq.~\eqref{eqn:general_choice_of_q} \\ \hline
        $r$ & $\mathbb{N}_+$ & JL dimension &  $\log^3(n_\lp)$ \\ \hline
        $\epsilon_{\apx,x}$ & $\in \R_{>0}$ & Theorem \ref{thm:approxds} approx parameter& $\ov \epsilon$ \\ \hline
        $\epsilon_{\apx,s}$ & $\in \R_{>0}$ & Theorem \ref{thm:approxds} approx parameter & $  \ov \epsilon \ov t w$ \\ \hline
        $\delta_\apx$ & $\in (0,0.1)$ & Theorem \ref{thm:approxds} failure probability & $1/N$ \\ \hline
        $\eta$  & $\in \mathbb{N}_+$ & elimination tree depth & $\log n$ \\ \hline
        $\zeta_x$ & $\in \R_{>0}$ & $\ell_2$ step size & $2 \alpha w^{-1/2}$ \\ \hline
        $\zeta_s$ & $\in \R_{>0}$ & $\ell_2$ step size & $2 \alpha \ov t w^{1/2}$ \\ \hline
        $\epsilon_t$ & $\in \R_{>0}$ & data structure restart threshold & $\ov \epsilon \min\{1, w/\nu_{\max}\}$\\ \hline
    \end{tabular}
    \caption{Internal parameters of Theorem~\ref{thm:algo_general} and their values. Caller has no access to these variables. For simplicity, we ignore the $O()$ and $\Theta()$ in the table.}
    \label{tab:internal_param_general}
\end{table}

Between two restarts, the total size of $|L_x|$ returned by $\mathsf{approx}.\textsc{Query}$ is bounded by $\wt O(q^2 \zeta_{x}^2 / \epsilon_{\apx,x}^2)$ by Theorem~\ref{thm:approxds}.
By plugging in $\zeta_x = 2 \alpha w^{-1/2}$,  $\epsilon_{\apx,x} = \ov \epsilon$, we have
\begin{align*} 
\sum_{\ell\in [q]} |L_x^{(\ell)}| = \wt O(q^2 w^{-1}).
\end{align*}

Similarly, for $s$, we have
\begin{align*}
\sum_{\ell\in [q]} |L_s^{(\ell)}| 
= & ~ \wt O(q^2 \zeta_{s}^2 / \epsilon_{\apx,s}^2) \\
= & ~ \wt O(q^2 \cdot \Theta(\ov \epsilon \cdot \ov t \cdot w^{1/2})^2 / \Theta(\ov \epsilon \cdot \ov t w)^2) \\
= & ~ \wt O(q^2 w^{-1}),
\end{align*}
where the first step is by Theorem~\ref{thm:approxds}, the second step is by our choice $\zeta_s = 2\alpha \ov t w^{1/2}$, $\epsilon_{\apx, s} = \ov \epsilon \ov t w$.

{\bf Cost per iteration.}

By Theorem~\ref{thm:exactds} and Theorem \ref{thm:approxds}, in a sequence of $q$ update/queries,
\begin{itemize}
\item the total cost for update is
\begin{align*}
&~ \text{number of block updates} \cdot \text{time per block update} \\
= &~ \wt O(q^2 w^{-1} ) \cdot
(\Tmat(n_{\max}) + T_{\Delta_L,\max} + T_{H,\max} + \eta^2 m_{\max}^2),
\end{align*}
\item the total cost for query is
$\wt O(q^2 w^{-1} \cdot \eta^2 m_{\max}^2)$.
\end{itemize}

Therefore (amortized) cost per update is $\wt O(q w^{-1} (\Tmat(n_{\max}) + T_{\Delta_L,\max} + T_{H,\max} + \eta^2 m_{\max}^2))$.

{\bf Time for init/restart.}

We restart the data structure whenever $k > q$ or $|\ov t-t| > \ov t \epsilon_t$, so
there are 
\begin{align*} 
    O(\frac Nq + \log (t_{\max}/t_{\min}) \epsilon_t^{-1})
\end{align*}
restarts in total.
By Theorem~\ref{thm:exactds}, Theorem~\ref{thm:approxds}, time cost per restart is
\begin{align*}
    \wt O(T_H + T_n + T_L + T_Z + \nnz(A) + \eta m_\lp m_{\max})
\end{align*}

and the total initialization time is
\begin{align*}
&~ \text{number of restarts} \cdot \text{time per restart}\\
=&~ \wt O((\frac Nq + \log (t_{\max}/t_{\min}) \epsilon_t^{-1}) \cdot (T_H + T_n + T_L + T_Z + \nnz(A) + \eta m_\lp m_{\max})).
\end{align*}

{\bf Total time.}

The overall running time is
\begin{align*}
\text{~total~} = & ~ 
\text{~init/restart~time} + N \cdot \text{cost~per~iter} \\
= & ~ \wt{O} ( (\frac{N}{q}+\log(t_{\max}/t_{\min}) \epsilon_t^{-1})  \cdot (T_H + T_n + T_L + T_Z + \nnz(A) + \eta m_\lp m_{\max}) \\
&~+ N \cdot q w^{-1} (\Tmat(n_{\max}) + T_{\Delta_L,\max} + T_{H,\max} + \eta^2 m_{\max}^2) )\\
= & ~ \wt{O} ( (\frac{n^{0.5} \nu_{\max}}{q} + \log(t_{\max}/t_{\min}))  \cdot (T_H + T_n + T_L + T_Z + \nnz(A) + \eta m_\lp m_{\max}) \\
&~+ n^{0.5} q (\Tmat(n_{\max}) + T_{\Delta_L,\max} + T_{H,\max} + \eta^2 m_{\max}^2) ) \\
=&~ \wt{O} ( n^{0.5} \nu_{\max}^{0.5} \cdot (T_H + T_n + T_L + T_Z + \nnz(A) + \eta m_\lp m_{\max})^{0.5} \\
&\cdot (\Tmat(n_{\max}) + T_{\Delta_L,\max} + T_{H,\max} + \eta^2 m_{\max}^2)^{0.5} \log(t_{\max}/t_{\min})).
\end{align*}
The third step is by taking $w = \nu_{\max}$, $N = \sqrt{n \nu_{\max} w}$, $\epsilon_t = \frac 12 \ov \epsilon$.
The fourth step is by taking
\begin{align} \label{eqn:general_choice_of_q}
    q =&~ n^{0.5} \nu_{\max}^{0.5} (T_H + T_n + T_L + T_Z + \nnz(A) + \eta m_\lp m_{\max})^{0.5}  \nonumber \\
    &~\cdot (\Tmat(n_{\max}) + T_{\Delta_L,\max} + T_{H,\max} + \eta^2 m_{\max}^2)^{-0.5}
\end{align}
Note that because the initialization time is bounded above by time for running $n$ updates, we always have
\begin{align*}
n (\Tmat(n_{\max}) + T_{\Delta_L,\max} + T_{H,\max} + \eta^2 m_{\max}^2) \ge (T_H + T_n + T_L + T_Z + \nnz(A) + \eta m_\lp m_{\max}).
\end{align*}
Therefore $q \le n^{0.5} \nu_{\max}^{0.5} \le N$ and \eqref{eqn:general_choice_of_q} is a valid choice for $q$.
\end{proof}
\begin{lemma}[\textsc{Output} part of Theorem~\ref{thm:central_path_maintenance_general}]\label{lem:output_time_general}
\textsc{CentralPathMaintenance.Output} (Algorithm~\ref{alg:cpm}) takes $O(\nnz(A) + T_m)$ time.
\end{lemma}
\begin{proof}

The proof directly follows from Theorem~\ref{thm:exactds}.
\end{proof}

\subsection{Proofs of Main Result}\label{sec:framework:main}
In this section we combine everything and prove Theorem~\ref{thm:central_path_maintenance_general} and Theorem~\ref{thm:algo_general}.

\begin{proof}[Proof of Theorem~\ref{thm:central_path_maintenance_general}]

Using Lemma~\ref{lem:cpm_correct_general}, we finish the correctness part.

Using Lemma
~\ref{lem:init_time_general}, we prove the running time for initialization,
\begin{align*} 
    {\cal T}_{\mathrm{init}} = \wt O(\nnz(A) + T_H + T_L + \eta T_m).
\end{align*}

Using Lemma~\ref{lem:multiply_and_move_time_general}, we prove the running time for multiply and move,
\begin{align*} 
    {\cal T}_{\mathrm{multiply}} = & \wt{O} ( n^{0.5} \nu_{\max}^{0.5} \cdot (T_H+T_L+\eta T_m + T_Z)^{0.5} \\
    &\cdot (\Tmat(n_{\max}) + T_{\Delta_L,\max} + T_{H,\max} + \eta^2 m_{\max}^2)^{0.5} \cdot \log(t_{\max}/t_{\min})).
\end{align*}

By Lemma~\ref{lem:output_time_general}, we show the output time,
\begin{align*} 
     {\cal T}_{\mathrm{output}}  = O(\nnz(A) + T_H + \eta T_m)
\end{align*} 

\end{proof}

\begin{proof}[Proof of Theorem~\ref{thm:algo_general}]

So, the total running time is 
\begin{align*}
 & ~{\cal T}_{\mathrm{init}} + {\cal T}_{\mathrm{multiply}} + {\cal T}_{\mathrm{output}} \\
 = & ~  \wt{O} ( n^{0.5} \nu_{\max}^{0.5} \cdot (T_H+T_L+\eta T_m + T_Z)^{0.5} \\
&\cdot (\Tmat(n_{\max}) + T_{\Delta_L,\max} + T_{H,\max} + \eta^2 m_{\max}^2)^{0.5} \\
&\cdot \log (R/(r\epsilon))),
\end{align*}
where we use Theorem~\ref{thm:central_path_maintenance_general} and Theorem~\ref{thm:DLYA.1} (by setting $t_{\max}/t_{\min} = R  /(r\epsilon)$).

\end{proof}

%% file: sdp_first.tex
\section{Our First Result} \label{sec:sdp_first}
In this section, we present an $\wt O(n_\sdp\cdot \poly(\tau_\sdp))$ algorithm for low-treewidth SDP.
Outline of this section is as follows.
\begin{itemize}
    \item In Section~\ref{sec:sdp_first:main}, we present the main result of this section, Theorem~\ref{thm:main_sdp_first}.
    \item In Section~\ref{sec:sdp_first:reduction}, we reduce our SDP to a form which can be handled by Theorem~\ref{thm:algo_general}.
    \item In Section~\ref{sec:sdp_first:choice_of_parameters}, we state and prove parameters needed to apply Theorem~\ref{thm:algo_general}.
    \item In Section~\ref{sec:sdp_first:main_proof}, we plug in all parameters and finish proof of Theorem~\ref{thm:main_sdp_first}.
    \item In Section~\ref{sec:sdp_first:discuss_inequality}, we discuss how to deal with inequality constraints.
\end{itemize}

\subsection{Main Statement}\label{sec:sdp_first:main}
We consider semidefinite program of form~\eqref{eqn:sdp-primal}. Let us restate it here for clarity.

\begin{align}\label{eqn:sdp-primal-restated}
    \min \qquad & C \bullet X\\
    \text{s.t.} \qquad & A_i \bullet X = b_i \qquad \forall i\in [m_\lp] \nonumber\\
    & X \succeq 0 \nonumber
\end{align}
with $C, X \in \R^{n_\sdp \times n_\sdp}$, $A\in \R^{m_\sdp \times n_\sdp^2}$, $b\in \R^{m_\sdp}$.
(We add subscript $m_\sdp$ to avoid confusion with parameters in Theorem~\ref{thm:algo_general}.)

\begin{definition} \label{defn:sdp-assumptions}
We make the following assumptions on program \eqref{eqn:sdp-primal-restated}.
\begin{itemize}
    \item Assume that constraint matrices $A_1,\ldots, A_{m_\sdp} \in \R^{n_\sdp^2}$ are linearly independent.
    \item Assume that we are given a tree decomposition (Definition~\ref{defn:treewidth}) of the SDP graph (Definition~\ref{def:sdp-graph}) of \eqref{eqn:sdp-primal-restated} with maximum bag size $\tau_\sdp$.
    \footnote{If such a tree decomposition is not given, then using \cite{bernstein2022deterministic}, we can find a tree decomposition (Definition~\ref{defn:treewidth}) with maximum bag size $\wt O(\tw(G))$, where $\tw(G)$ is treewidth of the SDP graph $G$.}
    \item Assume that any feasible solution $X\in \bS^{n_\sdp}$ satisfies $\|X\|_\op \le R$.
    \item There exists $Z\in \bS^{n_\sdp}$ such that $A\bullet Z = b$ and $\lambda_{\min}(Z) \ge r$.
\end{itemize}
\end{definition}

\begin{theorem}[Our first result]\label{thm:main_sdp_first}
Under assumptions in Definition~\ref{defn:sdp-assumptions}, for any $0<\epsilon\le \frac12$, there is an algorithm that outputs a solution $U\in \R^{n_\sdp \times \tau_\sdp}$
such that for $X = UU^\top$, we have $A \bullet X=b$ and
\begin{align*}
    C\bullet X \le \min_{A \bullet X' = b, X'\in \bS^{n_\sdp}} C\bullet X + \epsilon \|C\|_2 R
\end{align*}
in expected time
\begin{align*}
\wt O(n_\sdp \tau_\sdp^{6.5} \log(R/(r \epsilon)))
\end{align*}
where $\wt O$ hides $n^{o(1)}$ terms.
\end{theorem}

Our proof of Theorem~\ref{thm:main_sdp_first} is in two steps. First we reduce program~\eqref{eqn:sdp-primal-restated} to form~\eqref{eqn:lp_general}, then we apply Theorem~\ref{thm:algo_general} by plugging in needed parameters.

\subsection{Reduce Low-Treewidth SDP to General Treewidth Program} \label{sec:sdp_first:reduction}
In this section we reduce program~\eqref{eqn:sdp-primal-restated} satisfying Definition~\ref{defn:sdp-assumptions} to form~\eqref{eqn:lp_general}.

Recall that we are given a tree decomposition $T,J_1,\ldots, J_n$ of the SDP graph with maximum bag size $\tau_\sdp$, where $T$ is a tree on $n$ vertices and and $J_j$ are the bags.

Using Lemma~\ref{lem:tree-decomp-const-deg}, we can WLOG assume that $T$ has maximum degree $O(1)$.
\begin{lemma} \label{lem:tree-decomp-const-deg}
Given any tree decomposition $T$ with $n$ bags, we can construct another tree decomposition with at most $2n$ bags, with the same maximum bag size, and maximum degree at most $3$.
\end{lemma}
\begin{proof}
For every vertex $j\in [n]$ with degree $d_j > 3$, we can replace it with $d_j-1$ vertices, each of degree $3$, with the corresponding bags equal to $J_j$. It is easy to verify that this is a valid tree decomposition, and that the total number of bags after this transformation is at most $2n$.
\end{proof}

By Lemma~\ref{lem:zl18-reduction}, to solve program~\eqref{eqn:sdp-primal-restated}, it suffices to solve the following program with fewer variables:
\begin{align} \label{eqn:sdp-ctc-restated}
    \min \qquad & \sum_{j\in [n]} C_j \bullet X_j\\
    \text{s.t.} \qquad & A_i \bullet X_{j_i} = b_i \qquad \forall i\in [m_\sdp] \nonumber\\
    & \cN_{i,j} (X_i) = \cN_{j,i} (X_j) \qquad \forall (i,j)\in E(T) \nonumber\\
    & X_j \succeq 0  \qquad \forall j\in [n] \nonumber
\end{align}
where $j_i$ is any number with $\supp A_i \subseteq J_{j_i}$, and the $\cN$ constraints asserts consistency between different bags.
Here $X_j \in \R^{|J_j| \times |J_j|}$, $A_i \in \R^{|J_{j_i}| \times |J_{j_i}|}$, $b\in \R^{m_\sdp}$, $\cN_{i,j} \in \R^{|J_i\cap J_j|^2 \times |J_i|^2}$.
(We view $X_i$ in $\cN_{i,j}(X_i)$ as a vector of length $\R^{|J_i|^2}$.)
For all $j\in [n]$, $X_j$ corresponds to the $J_j\times J_j$ minor of $X$ (in \eqref{eqn:sdp-primal-restated}).
By properties of the tree decomposition, we have $n = O(n_\sdp)$ and $|J_j| = \wt O(\tau_\sdp)$ for all $j\in [n]$.
So the total variable dimension of program \eqref{eqn:sdp-ctc-restated} is
\begin{align*}
n_\lp = \sum_{j\in n} |J_j|^2 = \wt O(n_\sdp \tau_\sdp^2).
\end{align*}

\begin{lemma}[{\cite[Section 3.3]{zl18}}] \label{lem:zl18-reduction}
Under assumptions in Definition~\ref{defn:sdp-assumptions}, the optimal value of program \eqref{eqn:sdp-primal-restated} is equal to the optimal value of program \eqref{eqn:sdp-ctc-restated}.
Furthermore, given a solution $X_1,\ldots, X_n$ of program \eqref{eqn:sdp-ctc-restated}, there is an algorithm that takes $O(n_\sdp \tau_\sdp^3)$ time and outputs a matrix $U\in \bR^{n_\sdp \times \tau_\sdp}$ such that for $X=UU^\top$, $X$ is a solution of Eq.~\eqref{eqn:sdp-primal} with the same objective value.
\end{lemma}

Now we have reduced program~\eqref{eqn:sdp-primal-restated} to program~\eqref{eqn:sdp-ctc-restated}, which is of form~\eqref{eqn:lp_general}.
Let us compute the treewidth of program~\ref{eqn:sdp-ctc-restated}.

\begin{lemma} \label{lem:sdp-ctc-treewidth}
We can construct a tree decomposition of the LP dual graph (Definition~\ref{def:general_lp_treewidth}) with maximum bag size $\tau_\lp = O(\tau_\sdp^2)$ and at most $n$ blocks.
\end{lemma}
\begin{proof}

Let us define a bag decomposition for the LP dual graph.
For every $j\in [n]$, we define a bag $B_j$ containing
\begin{itemize}
    \item all type-$A$ constraints $A_i$ with $j_i=j$ or $(j_i,j)\in T$, and
    \item all type-$N$ constraints $\cN_{i,j}$ with $(i,j)\in T$.
\end{itemize}
We connect $B_j$ with $B_i$ if and only if $(i,j)\in T$.

Because of linear independence assumption, number of type-$A$ constraints in a block is at most $|J_i|^2$.
Because $T$ has constant maximum degree, number of type-$N$ constraints in a block is at most $O(|J_i|^2)$ (recall that one constraint $\cN_{i,j}(X_i) = \cN_{j,i}(X_j)$ is in fact $|J_i\cap J_j|^2$ equations).
So $|B_j| = O(|J_i|^2)$. The maximum bag size is $\tau_\lp = O(\max_j |J_j|^2) = O(\tau_\sdp^2)$.

Finally, we prove that $(T, B_1,\ldots, B_n)$ is a valid tree decomposition of the LP dual graph.
For any two constraints sharing a variable $X_j$, they must both be in $B_j$. So the first condition in Definition~\ref{defn:treewidth} is satisfied.
The second condition in Definition~\ref{defn:treewidth} is clearly satisfied.
So $(T, B_1,\ldots, B_n)$ is a valid tree decomposition.
\end{proof}
This enables solving program~\eqref{eqn:sdp-ctc-restated} using Theorem~\ref{thm:algo_general}.

\subsection{Choice of Parameters}\label{sec:sdp_first:choice_of_parameters}
In this section we present value of parameters needed to apply Theorem~\ref{thm:algo_general}.
\begin{lemma}[Choice of parameters in the first result] \label{lem:sdp_first_param}
Under the setting of Theorem~\ref{thm:main_sdp_first}, we can choose the following set of parameters.
\begin{enumerate}[label=(\roman*)]
    \item \label{item:lem:sdp_first_n} $n = O(n_\sdp)$, $n_\lp = O(n_\sdp\tau_\sdp^2)$, $n_{\max} = O(\tau_\sdp^2)$.
    \item \label{item:lem:sdp_first_m} $m = O(n_\sdp \tau_\sdp^2)$, $m_\lp = O(n_\sdp \tau_\sdp^2)$, $m_{\max} = 1$.
    \item \label{item:lem:sdp_first_taulp} $\tau_\lp = O(\tau_\sdp^2)$.
    \item \label{item:lem:sdp_first_eta} $\eta = \wt O(\tau_\sdp^2)$.
    \item \label{item:lem:sdp_first_nu} $\nu_{\max} = O(\tau_\sdp)$.
    \item \label{item:lem:sdp_first_nnz} $\nnz(A) = O(n_\sdp \tau_\sdp^4)$.
    \item \label{item:lem:sdp_first_TnTm} $T_n = O(n_\sdp \tau_\sdp^{2\omega})$, $T_m = O(n_\sdp \tau_\sdp^2)$.
    \item \label{item:lem:sdp_first_TH} $T_H = O(n_\sdp \tau_\sdp^{2\omega})$, $T_{H,\max} = O(\tau_\sdp^{2\omega})$.
    \item \label{item:lem:sdp_first_TL} $T_L = \wt O(n_\sdp \tau_\sdp^6)$.
    \item \label{item:lem:sdp_first_TZ} $T_Z = \wt O(n_\sdp \tau_\sdp^6)$.
    \item \label{item:lem:sdp_first_TDeltaLmax} $T_{\Delta L,\max} = \wt O(\tau_\sdp^6)$.
\end{enumerate}
\end{lemma}
\begin{proof}
\begin{enumerate}[label=(\roman*)]
    \item Values for $n$, $n_\lp$, $n_{\max}$ have been discussed in Section~\ref{sec:sdp_first:reduction}.
    \item By discussion in last section, we have $m_\lp = O(n_\sdp \tau_\sdp^2)$.
    We choose block structure $(1,\ldots, 1)$, i.e., $m = m_\lp = O(n_\sdp \tau_\sdp^2)$ and $m_{\max} = 1$.
    \item Proved in Lemma~\ref{lem:sdp-ctc-treewidth}.
    \item By Lemma~\ref{lem:sdp_first_block_elim}.
    \item We use the log barrier (Definition~\ref{defn:log-barrier-psd}).
    By Lemma~\ref{lem:log-barrier-psd}, we have $\nu_{\max} = O(\tau_\sdp)$.
    \item Every $X_j$ is in at most $\tau_\lp = O(\tau_\sdp^2)$ constraints, so $\nnz(A) = O(n_\sdp \tau_\sdp^2 \cdot \tau_\lp) = O(n_\sdp \tau_\sdp^4)$.
    \item Bounds on $T_n$ and $T_m$ are direct consequences of \ref{item:lem:sdp_first_n}, \ref{item:lem:sdp_first_m}.
    \item Computing Hessian of the log-barrier takes $O(\tau_\lp^\omega) = O(\tau_\sdp^{2\omega})$ time. So $T_H = O(n_\sdp \tau_\sdp^{2\omega})$, $T_{H,\max} = O(\tau_\sdp^{2\omega})$.
    \item By Lemma~\ref{lem:known_cholesky_TL}, we have $T_L = O(m_\lp \eta^2) = \wt O(n_\sdp \tau_\sdp^6)$.
    \item By Lemma~\ref{lem:generic_bound_TZ}, we have $T_Z = O(\eta^2 m m_{\max}^2) = \wt O(n_{\sdp} \tau_{\sdp}^6)$.
    \item Updating $AH_{\ov x}A^\top =LL^\top$ under change of one block involves $O(n_{\max}) = O(\tau_\sdp^2)$ rank-$1$ updates, and each such update costs $O(\eta^2) = \wt O(\tau_\sdp^4)$ time by Lemma~\ref{lem:known_cholesky_TDeltaLmax}.
    So $T_{\Delta L,\max} = \wt O(\tau_\sdp^6)$.
\end{enumerate}
\end{proof}

\begin{lemma} \label{lem:sdp_first_block_elim}
We can construct a block elimination tree with block structure $(1,\ldots, 1)$ (i.e., $m=m_\lp$, $m_{\max}=1$) and maximum block-depth $\eta = \wt O(\tau_\lp)$.
\end{lemma}
\begin{proof}
This basically follows from \cite{bodlaender1995approximating}.
We present the construction here for completeness.
Consider Algorithm~\ref{alg:sdp_first_block_elim}.
Because in each step we choose the centroid of $T$, the maximum recursion depth is at most $\wt O(1)$.
Every bag $J_i$ has size $O(\tau_\lp)$, so maximum depth of the constructed tree is $\wt O(\tau_\lp)$.

It remains to prove that the constructed tree $\cT$ is a valid block elimination tree.
By properties of a bag decomposition, the condition in Lemma~\ref{lem:prove_block_elim_tree} is satisfied. So $\cT$ is a valid block elimination tree.
\end{proof}

\begin{algorithm}[!ht]\caption{Construct an Elimination Tree}\label{alg:sdp_first_block_elim}
\begin{algorithmic}[1]
\Procedure{\textsc{ConstructElimTree}}{$T,J_1,\ldots,J_k$}
\State \textbf{Input:} A tree decomposition of a certain graph
\State \textbf{Output:} Root of the constructed elimination tree
\If{$k\le 1$}
\State Construct a rooted path using an arbitrary ordering of $J_1$. Let $r$ be the root.
\State \Return $r$.
\EndIf
\State Let $i\in [k]$ be the centroid of $T$.
\State Let $[k] = S' \cup S'' \cup \{i\}$ be a partition such that $|S'|, |S''| \le \frac 23 k$
\State $r'\gets \textsc{ConstructElimTree}(T', J_j-J_i: j\in S')$ where $T'$ is the spanning forest of $S'$.
\State $r''\gets \textsc{ConstructElimTree}(T'', J_j-J_i: j\in S'')$ where $T''$ is the spanning forest of $S''$.
\State Construct a rooted path using an arbitrary ordering of $J_i$. Let $r$ be the root, $u$ be the leaf.
\State Set $u$ as the parent of $r'$ and $r''$.
\State \Return $r$.
\EndProcedure
\end{algorithmic}
\end{algorithm}

\subsection{Proof of Theorem~\ref{thm:main_sdp_first}} \label{sec:sdp_first:main_proof}
In this section, we prove Theorem~\ref{thm:main_sdp_first}.

\begin{proof}[Proof of Theorem~\ref{thm:main_sdp_first}]
According to Theorem~\ref{thm:algo_general}, there is an algorithm solving SDP in
\begin{align*}
&\wt{O} ( n^{0.5} \nu_{\max}^{0.5} \cdot (T_H + T_n + T_L + T_Z + \nnz(A) + \eta m_\lp m_{\max})^{0.5} \\
&\cdot (\Tmat(n_{\max}) + T_{\Delta_L,\max} + T_{H,\max} + \eta^2 m_{\max}^2)^{0.5} \log(R/(r\epsilon)))
\end{align*}
time.

It remains to compute the parameters $T_H$, $T_L$, $\eta$, $T_m$, $m_{\max}$, $n_{\max}$, $T_{\Delta_L, \max}$, $T_{H,\max}$, the details can be found in Lemma~\ref{lem:sdp_first_param}.

Let $A $ and $B$ be defined as:
\begin{align*}
    A: = & ~ T_H + T_n + T_L  + T_Z + \nnz(A) + \eta m_{\lp} m_{\max}, \\
    B := & ~ \Tmat(n_{\max}) + T_{\Delta_L,\max} + T_{H,\max} + \eta^2 m_{\max}^2.
\end{align*}

We can show that
\begin{align}\label{eq:sdp_first_A}
    A = & ~ T_H + T_n + T_L + T_Z + \nnz(A) + \eta m_{\lp} m_{\max} \notag \\
    = & ~ O(n_{\sdp} \tau_{\sdp}^{2\omega}) + T_n + T_L + T_Z + \nnz(A) + \eta m_{\lp} m_{\max} \notag \\
    = & ~ O(n_{\sdp} \tau_{\sdp}^{2\omega}) + O(n_{\sdp} \tau_{\sdp}^{2\omega}) + T_L + T_Z + \nnz(A) + \eta m_{\lp} m_{\max} \notag \\
    = & ~ O(n_{\sdp} \tau_{\sdp}^{2\omega}) + T_L + T_Z + \nnz(A) + \eta m_{\lp} m_{\max} \notag \\
    = & ~ O(n_{\sdp} \tau_{\sdp}^{6})  + O(n_{\sdp} \tau_{\sdp}^{2\omega}) + T_Z + \nnz(A) + \eta m_{\lp} m_{\max} \notag \\
    = & ~ O(n_{\sdp} \tau_{\sdp}^{6}) + T_Z + \nnz(A) + \eta m_{\lp} m_{\max} \notag \\
    = & ~ O(n_{\sdp} \tau_{\sdp}^{6}) + O(n_{\sdp} \tau_{\sdp}^{6}) + \nnz(A) + \eta m_{\lp} m_{\max} \notag \\
    = & ~ O(n_{\sdp} \tau_{\sdp}^{6}) + \nnz(A) + \eta m_{\lp} m_{\max} \notag \\
    = & ~ O(n_{\sdp} \tau_{\sdp}^{6}) + O(n_{\sdp} \tau_{\sdp}^{4}) + \eta m_{\lp} m_{\max} \notag \\
    = & ~ O(n_{\sdp} \tau_{\sdp}^{6}) + O(n_{\sdp} \tau_{\sdp}^{4}) + O(n_{\sdp} \tau_{\sdp}^{4}) \notag \\
    = & ~ O(n_{\sdp} \tau_{\sdp}^{6})
\end{align}
where the second step follows from Lemma~\ref{lem:sdp_first_param}\ref{item:lem:sdp_first_TH} ($T_H = O(n_{\sdp} \tau_{\sdp}^{2\omega})$) and the third step follows from Lemma~\ref{lem:sdp_first_param}\ref{item:lem:sdp_first_n} ($T_n = O(n_{\sdp} \tau_{\sdp}^{2\omega})$), the forth step follows from merging the terms, the fifth step follows from Lemma~\ref{lem:sdp_first_param}\ref{item:lem:sdp_first_TL} ($T_L = \wt{O}(\tau_{\sdp} \tau^{2\omega})$, the sixth step follows from merging the terms, the seventh step follows from $T_Z = O(n_{\sdp} \tau_{\sdp}^6)$, the eighth step follows from merging the terms, the ninth step follows   Lemma~\ref{lem:sdp_first_param}\ref{item:lem:sdp_first_nnz} ($\nnz(A) = n_{\sdp} \tau_{\sdp}^4$), and the tenth step follows from Lemma~\ref{lem:sdp_first_param}\ref{item:lem:sdp_first_eta} and Lemma~\ref{lem:sdp_first_param}\ref{item:lem:sdp_first_m}.

For the term $B$, we have
\begin{align}\label{eq:sdp_first_B}
    B = & ~ \Tmat(n_{\max}) + T_{\Delta_L,\max} + T_{H,\max} + \eta^2 m_{\max}^2 \notag \\
    = & ~ O( \tau_{\lp}^{\omega} ) + T_{\Delta_L,\max} + T_{H,\max} + \eta^2 m_{\max}^2 \notag \\
    = & ~ O( \tau_{\lp}^{\omega} ) + O( \tau_{\lp}^3 ) + T_{H,\max} + \eta^2 m_{\max}^2 \notag \\
    = & ~ O( \tau_{\lp}^{\omega} ) + O( \tau_{\lp}^3 ) + O( \tau_{\lp}^{\omega} ) + \eta^2 m_{\max}^2 \notag \\
    = & ~ O( \tau_{\lp}^{\omega} ) + O( \tau_{\lp}^3 ) + O( \tau_{\lp}^{\omega} ) + \wt{O}( \tau_{\lp}^2 ) \notag \\
    = & ~ O(\tau_{\lp}^3) \notag \\
    = & ~ O(\tau_{\sdp}^6)
\end{align}
where the second step follows Lemma~\ref{lem:sdp_first_param}\ref{item:lem:sdp_first_n} ($n_{\max} = \tau_{\lp}$), and the third step follows from Lemma~\ref{lem:sdp_first_param}\ref{item:lem:sdp_first_TDeltaLmax} ($T_{\Delta L, \max} = O(\tau_{\lp}^3)$), the forth step follows from Lemma~\ref{lem:sdp_first_param}\ref{item:lem:sdp_first_TH} ($T_{H,\max} = O(\tau_{\lp}^{\omega})$), the fifth step follows from Lemma~\ref{lem:sdp_first_param}\ref{item:lem:sdp_first_eta} ($\eta = \wt{O}( \tau_{\lp} )$ and $m_{\max} = O(1)$), and the last step follows from $\tau_{\lp} = \tau_{\sdp}^2$.

Finally, we have
\begin{align*}
    & ~ \wt{O}( n^{0.5} \nu_{\max}^{0.5} \cdot A^{0.5} \cdot B^{0.5} \cdot \log(1/\epsilon) ) \\
    = & ~ \wt{O} (n_{\sdp}^{0.5} \tau_{\sdp}^{0.5} \cdot A^{0.5} \cdot B^{0.5} \cdot \log(1/\epsilon)) \\
    = & ~ \wt{O} (n_{\sdp}^{0.5} \tau_{\sdp}^{0.5} \cdot (n_{\sdp} \tau_{\sdp}^6)^{0.5} \cdot (\tau_{\sdp}^6)^{0.5} \cdot \log(1/\epsilon)) \\
    = & ~ \wt{O}( n_{\sdp} \tau_{\sdp}^{6.5} \cdot \log(1/\epsilon) )
\end{align*}
where the first step follows from $n = n_{\sdp}$ and Lemma~\ref{lem:sdp_first_param}\ref{item:lem:sdp_first_eta} ($\nu_{\max} = O(\tau_{\sdp})$), the second step follows from $A = \wt{O}(n_{\sdp} \tau_{\sdp}^6)$ (see Eq.~\eqref{eq:sdp_first_A}) and $B=\wt{O}(\tau_{\sdp}^6)$ (see Eq.~\eqref{eq:sdp_first_B}).
Thus, we complete the proof.

\end{proof}

\subsection{Discussions on Inequality Constraints} \label{sec:sdp_first:discuss_inequality}
In certain problems (e.g., $\mathsf{MaxCut}$) there are inequality constraints. In this section we briefly discuss how to deal with  inequality constraints.

For every inequality constraint of form $A_i \bullet X_{J_i} \ge b_i$ (where $J_i$ is a bag in the tree decomposition), we add a real variable $v_i\in \R$ and replace the inequality constraint with $A_i \bullet X_{J_i} - v_i = b_i$ and $v_i\ge 0$.

This adds $m_{\sdp} = O(n_\sdp \tau_\sdp)$ real variables in total.
Let us discuss the variable block structure.
For every $j\in [n]$, consider the set $S_j := \{i\in [m]: J_i = j\}$.
We divide $S_j$ into $\lceil |S_j| / \tau_\sdp\rceil$ blocks, each with $O(\tau_\sdp)$ variables.
In this way the number of variable blocks $n$ increases by at most a constant factor, maximum size $n_{\max}$ of a variable block does not change, and $\nu_{\max}$ is still $O(\tau_\sdp)$.

Number of constraints $m_\lp$ stays the same, and we can use the same block structure as in Lemma~\ref{lem:sdp_first_block_elim} or Lemma~\ref{lem:sdp_second_block_elim}.

Therefore, the running time claims of Theorem~\ref{thm:main_sdp_first} and Theorem~\ref{thm:main_sdp_second} still hold in the presence of inequality constraints.

%% file: sdp_second.tex
\section{Our Second Result, An Improved Version of Our First Result}
\label{sec:sdp_second}
In this section, we improve our algorithm in Section~\ref{sec:sdp_first} to get an $\wt O(n_\sdp \tau^{2\omega+0.5})$ time algorithm for low-treewidth SDP.
Outline of this section is as follows.

\begin{itemize}
    \item In Section~\ref{sec:sdp_second:main}, we present the main result of this section, Theorem~\ref{thm:main_sdp_second}.
    \item In Section~\ref{sec:sdp_second:choice_of_parameters}, we state and prove parameters needed to apply Theorem~\ref{thm:algo_general}.
    
    \item In Section~\ref{sec:sdp_second:block_cholesky}, we develop improved algorithms for Cholesky-related computation.
    \item In Section~\ref{sec:sdp_second:main_proof}, we plug in all parameters and finish proof of Theorem~\ref{thm:main_sdp_second}.
\end{itemize}

\subsection{Main Statement}\label{sec:sdp_second:main}
We work under the same setting as Section~\ref{sec:sdp_first}, where we are given program~\eqref{eqn:sdp-ctc-restated} with assumptions in Definition~\ref{defn:sdp-assumptions}.
\begin{theorem}[Our second result]\label{thm:main_sdp_second}
Under assumptions in Definition~\ref{defn:sdp-assumptions}, for any $0<\epsilon\le \frac 12$, there is an algorithm that outputs a solution $U\in \R^{n_\sdp \times \tau_\sdp}$ such that for $X = UU^\top$, we have $A\bullet X = b$ and 
\begin{align*}
    C\bullet X \le \min_{A \bullet X' = b, X'\in \bS^{n_\sdp}} C\bullet X + \epsilon \|C\|_2 R
\end{align*}
in expected time
\begin{align*}
\wt O( n \cdot \tau^{2\omega + 0.5} \log(R/(r\epsilon)))
\end{align*}
where $\wt O$ hides $n^{o(1)}$ terms.
\end{theorem}

\subsection{Choice of Parameters}\label{sec:sdp_second:choice_of_parameters}
In this section we present value of parameters needed to apply Theorem~\ref{thm:algo_general}.

\begin{lemma}[Choice of parameters in the second result]\label{lem:sdp_second_param}
Under the setting of Theorem~\ref{thm:main_sdp_second}, we can choose the following set of parameters.
\begin{enumerate}[label=(\roman*)]
    \item \label{item:lem:sdp_second_n} $n = O(n_\sdp)$, $n_\lp = O(n_\sdp\tau_\sdp^2)$, $n_{\max} = O(\tau_\sdp^2)$.
    \item \label{item:lem:sdp_second_m} $m = O(n_\sdp)$, $m_\lp = O(n_\sdp \tau_\sdp^2)$, $m_{\max} = O(\tau_\sdp^2)$.
    \item \label{item:lem:sdp_second_taulp} $\tau_\lp = O(\tau_\sdp^2)$.
    \item \label{item:lem:sdp_second_eta} $\eta = \wt O(1)$.
    \item \label{item:lem:sdp_second_nu} $\nu_{\max} = O(\tau_\sdp)$.
    \item \label{item:lem:sdp_second_nnz} $\nnz(A) = O(n_\sdp \tau_\sdp^4)$.
    \item \label{item:lem:sdp_second_TnTm} $T_n = \wt O(n_\sdp \tau_\sdp^{2\omega})$, $T_m = \wt O(n_\sdp \tau_\sdp^{2\omega})$.
    \item \label{item:lem:sdp_second_TH} $T_H = \wt O(n_\sdp \tau_\sdp^{2\omega})$, $T_{H,\max} = \wt O(\tau_\sdp^{2\omega})$.
    \item \label{item:lem:sdp_second_TL} $T_L = \wt O(n_\sdp \tau_\sdp^{2\omega})$.
    \item \label{item:lem:sdp_second_TZ} $T_Z = \wt O(n_\sdp \tau_\sdp^{4})$.
    \item \label{item:lem:sdp_second_TDeltaLmax} $T_{\Delta L,\max} = \wt O(\tau_\sdp^{2\omega})$.
\end{enumerate}
\end{lemma}
\begin{proof}
\begin{enumerate}[label=(\roman*)]
    \item Same as Lemma~\ref{lem:sdp_first_param}\ref{item:lem:sdp_first_n}.
    \item Bound on $m_\lp$ is same as Lemma~\ref{lem:sdp_first_param}\ref{item:lem:sdp_first_m}.
    For choice of block structure (and thus $m$ and $m_{\max}$), see Lemma~\ref{lem:sdp_second_block_elim}.
    \item Same as Lemma~\ref{lem:sdp_first_param}\ref{item:lem:sdp_first_taulp}.
    \item By Lemma~\ref{lem:sdp_second_block_elim}.
    \item Same as Lemma~\ref{lem:sdp_first_param}\ref{item:lem:sdp_first_nu}.
    \item Same as Lemma~\ref{lem:sdp_first_param}\ref{item:lem:sdp_first_nnz}.
    \item Bounds on $T_n$ and $T_m$ are direct consequences of \ref{item:lem:sdp_second_n}, \ref{item:lem:sdp_second_m}.
    \item Same as Lemma~\ref{lem:sdp_first_param}\ref{item:lem:sdp_first_TH}.
    \item By Lemma~\ref{lem:block_cholesky_decomposition}, we have $T_L  = \wt O(m \tau_\lp^\omega) = \wt O(n_\sdp \tau_\sdp^{2\omega})$.
    \item By Lemma~\ref{lem:generic_bound_TZ}, we have $T_Z = O(\eta^2 m m_{\max}^2) = \wt O(n_\sdp \tau_{\sdp}^4)$.
    \item  By Lemma~\ref{lem:block_cholesky_update}, we have $T_{\Delta L, \max} = \wt O(\tau_{\lp}^\omega) = \wt O(\tau_\sdp^{2\omega})$.
\end{enumerate}
\end{proof}

\begin{lemma} \label{lem:sdp_second_block_elim}
We can construct a block elimination tree with a block structure where $m = O(n_\sdp)$, $m_{\max}=O(\tau_\lp)$, such that the maximum block-depth is $\eta = \wt O(1)$.
\end{lemma}
\begin{proof}
We run Algorithm~\ref{alg:sdp_second_block_elim}, a block version of Algorithm~\ref{alg:sdp_first_block_elim}.
The difference is that, for every block, we create a new vertex containing all elements in the block, instead of constructing a path.
In this way, the constructed block elimination tree $\cT$ has block depth $\wt O(1)$. Every block in $\cT$ is contained in an original block $J_i$, so maximum block size is $O(\tau_\lp)$.
Furthermore, number of vertices in $T$ is at most number of original blocks, which is at most $n = O(n_\sdp)$ by Lemma~\ref{lem:sdp-ctc-treewidth}.

It remains to prove that the constructed tree $\cT$ is a valid block elimination tree.
By properties of a bag decomposition, the condition in Lemma~\ref{lem:prove_block_elim_tree} is satisfied. So $\cT$ is a valid block elimination tree.
\end{proof}

\begin{algorithm}[!ht]\caption{Construct a Block Elimination Tree}\label{alg:sdp_second_block_elim}
\begin{algorithmic}[1]
\Procedure{\textsc{ConstructBlockElimTree}}{$T,J_1,\ldots,J_k$}
\State \textbf{Input:} A tree decomposition of a graph
\State \textbf{Output:} Root of the constructed elimination tree
\If{$k\le 1$}
\State Let $r$ be a single vertex containing block $J_1$.
\State \Return $r$.
\EndIf
\State Let $i\in [k]$ be the centroid of $T$.
\State Let $[k] = S' \cup S'' \cup \{i\}$ be a partition such that $|S'|, |S''| \le \frac 23 k$
\State $r'\gets \textsc{ConstructBlockElimTree}(T', J_j-J_i: j\in S')$ where $T'$ is the spanning forest of $S'$.
\State $r''\gets \textsc{ConstructBlockElimTree}(T'', J_j-J_i: j\in S'')$ where $T''$ is the spanning forest of $S''$.
\State Let $r$ be a single vertex containing block $J_i$.
\State Set $r$ as the parent of $r'$ and $r''$.
\State \Return $r$.
\EndProcedure
\end{algorithmic}
\end{algorithm}

\subsection{Cholesky Decomposition Using Block Structures} \label{sec:sdp_second:block_cholesky}
In this section we discuss how to utilize the block elimination tree constructed in Lemma~\ref{lem:sdp_second_block_elim} to compute and update Cholesky decomposition faster.

\begin{algorithm}[!ht] \caption{Compute Cholesky factorization using Block Structure}
\label{alg:block_cholesky}
\begin{algorithmic}[1] 
\Procedure{BlockCholesky}{$M, \cT$} \Comment{Lemma~\ref{lem:block_cholesky_decomposition}}
\State \textbf{Input:} PSD matrix $M$ with block elimination tree $\cT$
\State \textbf{Output:} Lower-triangular $L$ such that $M=LL^\top$
\For{$j=1$ to $m$}
\State $L_{j,j} \leftarrow (M_{j,j} - \sum_{k\in \cD(j)\backslash j} L_{j,k} L_{j,k}^\top)^{1/2}$
\Comment{$\cD(j)$ is the set of descendants of $j$}
\label{line:alg_block_cholesky_sqrt}
\For{$i\in \cP(j)$}
\Comment{$\cP(j)$ is the set of ancestors of vertex $j$}
\State $L_{i,j} \leftarrow (M_{i,j} - \sum_{k\in \cD(j)\backslash j} L_{i,k} L_{j,k}^\top) L_{j,j}^{-\top}$ \label{line:alg_block_cholesky_non_diag}
\EndFor
\EndFor
\label{line:alg_block_cholesky_part1end}
\For{$j=1$ to $m$} \label{line:alg_block_cholesky_part2}
\State Compute QR decomposition $L_{j,j} = Q_j R_j$
\State $L_{i,j} \gets L_{i,j} Q_j^\top$ for $i\in \cP(j)$ \label{line:alg_block_cholesky_ortho}
\EndFor
\State \Return $L$
\EndProcedure
\end{algorithmic}
\end{algorithm}

\begin{lemma}[]\label{lem:block_cholesky_decomposition}
Let $M\in \R^{m_\lp\times m_\lp}$ be a PSD matrix with block elimnation tree $\cT$, such that the coordinates have been re-ordered in post-traversal order of $\cT$.
Suppose that $\cT$ has $m$ vertices, has maximum depth $\wt O(1)$, and maximum block size $O(\tau)$.
Then we can compute the Cholesky factorization $M = LL^\top$ in $\wt O(m \tau^\omega)$ time.
\end{lemma}
\begin{proof}
We run Algorithm~\ref{alg:block_cholesky}.

\textbf{Correctness:}
From the algorithm we can see $L_{i,j}\ne 0$ only when $i\in \cP(j)$.
So for $i,j\in [m]$, if $(LL^\top)_{i,j}\ne 0$, then either $i\in \cP(k)$ or $j\in \cP(i)$.
WLOG assume that $i\in \cP(j)$.
If $j=i$, then Line~\ref{line:alg_block_cholesky_sqrt} shows that $(LL^\top)_{i,i} = M_{i,i}$.
If $j\ne i$, then Line~\ref{line:alg_block_cholesky_non_diag} shows that $(LL^\top)_{i,j} = M_{i,j}$.
So $LL^\top = M$.

Then we prove that the square root in Line~\ref{line:alg_block_cholesky_sqrt} always exists.
Let $\cD'(j) := \cD(j) \backslash j$.
Recall that
\begin{align*}
L_{j, \cD'(j)} L_{\cD'(j),\cD'(j)}^{\top} = M_{j, \cD'(j)}.
\end{align*}
So
\begin{align*}
&~M_{j,j} - L_{j, \cD'(j)}L_{j, \cD'(j)}^\top\\
=&~ M_{j,j} - M_{j, \cD'(j)} L_{\cD'(j),\cD'(j)}^{-\top} L_{\cD'(j),\cD'(j)}^{-1} M_{\cD'(j), j}\\
=&~ M_{j,j} - M_{j, \cD'(j)} M_{\cD'(j),\cD'(j)}^{-1} M_{\cD'(j), j} \\
\succeq &~ 0
\end{align*}
where the last step is by property of PSD matrix.
So the square root in Line~\ref{line:alg_block_cholesky_sqrt} can always be taken.

Finally, let us examine the effect of Line~\ref{line:alg_block_cholesky_ortho}.
Note that before Line~\ref{line:alg_block_cholesky_part2}, all $L_{j,j}$ are PSD matrices.
Line~\ref{line:alg_block_cholesky_ortho} makes update $L_{j,j} \gets L_{j,j} Q_j^\top = (Q_j L_{j,j})^\top = R_j^\top$, which makes $L_{j,j}$ a lower triangular matrix.
So in the end $L$ is a lower triangular matrix as desired.

\textbf{Running time:}
Because the block elimination tree has block depth $\wt O(1)$, there are $\wt O(m)$ triples $(i,j,k)$ with $i\in \cP(j)$, $j\in \cP(k)$. For each such triple, we take $\wt O(\tau^\omega)$ time to perform the corresponding computations.
So computation before Line~\ref{line:alg_block_cholesky_part1end} takes $\wt O(m \tau^\omega)$ time.
By \cite{demmel2007fast}, computing QR decomposition of a matrix of size $\tau$ takes $\wt O(\tau^{\omega})$ time.
So computation starting from Line~\ref{line:alg_block_cholesky_part2} takes $\wt O(m\tau^{\omega})$ time.
Therefore the whole algorithm runs in $\wt O(m \tau^\omega)$ time.
\end{proof}

\begin{algorithm}[!ht] \caption{Update Cholesky factorization using Block Structure}
\label{alg:block_cholesky_update}
\begin{algorithmic}[1] 
\Procedure{BlockCholeskyUpdate}{$M, \cT, \Delta_M, v$} \Comment{Lemma~\ref{lem:block_cholesky_update}}
\State $L^\new \gets L$
\For{$j\in \cP(v)$ from bottom to up}
\State $L^\new_{j,j} \gets (L_{j,j} L_{j,j}^\top + \Delta_{M,(j,j)}- \sum_{k\in \cP(v) \cap (\cD(j)\backslash j)} (L_{j,k}^\new L_{j,k}^{\new \top} - L_{j,k}L_{j,k}^\top))^{1/2}$
\label{line:alg_block_cholesky_update_sqrt}
\For{$i\in \cP(j)$}
\State $L_{i,j}^\new \leftarrow (L_{i,j} L_{j,j}^\top + \Delta_{M,(i,j)}- \sum_{k\in \cP(v) \cap (\cD(j)\backslash j)} (L_{i,k}^\new L_{j,k}^{\new \top} - L_{i,k} L_{j,k}^\top))(L^{\new}_{j,j})^{-\top}$
\label{line:alg_block_cholesky_update_non_diag}

\EndFor
\EndFor \label{line:alg_block_cholesky_update_part1end}
\For{$j\in \cP(v)$ from bottom to up}
\State Compute QR decomposition $L^\new_{j,j} = Q_j R_j$
\State $L^\new_{i,j} \gets L^\new_{i,j} Q_j^\top$ for $i\in \cP(j)$
\EndFor
\State \Return $(L^\new-L)_{\cP(v),v}$
\EndProcedure
\end{algorithmic}
\end{algorithm}

\begin{lemma}\label{lem:block_cholesky_update}
Work under the setting of Lemma~\ref{lem:block_cholesky_decomposition}. In addition, assume that we already computed the Cholesky decomposition $M=LL^\top$.
Suppose we perform an update $M\gets M + \Delta_M$, where support of $\Delta_M$ is contained in the union of block row $v$ and block column $v$, for some vertex $v\in \cT$.
Then in $\wt O(\tau^\omega)$ time, we can compute $\Delta_L$ such that $M+\Delta_M = (L+\Delta_L) (L+\Delta_L)^\top$ is the Cholesky factorization of $M+\Delta_M$.
\end{lemma}
\begin{proof}
We run Algorithm~\ref{alg:block_cholesky_update}.

\textbf{Correctness:}
From Algorithm~\ref{alg:block_cholesky} we can see that only $L_{*,j}$ for $j\in \cP(v)$ are update. So we do not need to care about $L_{*,j}$ for $j\not \in \cP(v)$.
Let $M^\new = M + \Delta_M$.
In Algorithm~\ref{alg:block_cholesky_update}, Line~\ref{line:alg_block_cholesky_update_sqrt}, we have
\begin{align*}
L^\new_{j,j} &= (L_{j,j} L_{j,j}^\top + \Delta_{M,(j,j)}- \sum_{k\in \cP(v) \cap (\cD(j)\backslash j)} (L_{j,k}^\new L_{j,k}^{\new \top} - L_{j,k}L_{j,k}^\top))^{1/2}\\
&= (M_{j,j} + \Delta_{M,(j,j)} - \sum_{k\in \cD(j)\backslash(j)} L_{j,k} L_{j,k}^\top - \sum_{k\in \cP(v) \cap (\cD(j)\backslash j)} (L_{j,k}^\new L_{j,k}^{\new \top} - L_{j,k}L_{j,k}^\top))^{1/2} \\
&= (M^\new_{j,j} - \sum_{k\in \cD(j)\backslash j} L_{j,k}^\new L_{j,k}^{\new \top})^{1/2}.
\end{align*}
In Algorithm~\ref{alg:block_cholesky_update}, Line~\ref{line:alg_block_cholesky_update_non_diag}, we have
\begin{align*}
    L_{i,j}^\new &= (L_{i,j} L_{j,j}^\top + \Delta_{M,(i,j)}- \sum_{k\in \cP(v) \cap (\cD(j)\backslash j)} (L_{i,k}^\new L_{j,k}^{\new \top} - L_{i,k} L_{j,k}^\top))(L^{\new}_{j,j})^{-\top} \\
    &= (M_{i,j} + \Delta_{M,(i,j)} - \sum_{k\in \cD(j) \backslash j} L_{i,k} L_{j,k}^\top -  \sum_{k\in \cP(v) \cap (\cD(j)\backslash j)} (L_{i,k}^\new L_{j,k}^{\new \top} - L_{i,k} L_{j,k}^\top))(L^{\new}_{j,j})^{-\top}\\
    &= (M^\new_{i,j} - \sum_{k\in \cD(j) \backslash j} L_{i,k}^\new L_{j,k}^{\new \top}) (L^{\new}_{j,j})^{-\top}.
\end{align*}
Comparing with Algorithm~\ref{alg:block_cholesky}, we see that the updated values correct, i.e., at the end of Algorithm~\ref{alg:block_cholesky_update}, Line~\ref{line:alg_block_cholesky_update_part1end}, we have $M^\new = L^{\new} L^{\new \top}$. Finally, we rotate the updated columns by orthogonal matrices to achieve lower triangular matrix.

\textbf{Running time:}
There are $\wt O(1)$ tuples $(i,j,k)$ such that $k\in \cP(v), j\in \cP(k), i\in \cP(j)$.
For every such tuple, computation time is $\wt O(\tau^\omega)$.
So total update time is $\wt O(\tau^\omega)$.
\end{proof}

\subsection{Proof of Theorem~\ref{thm:main_sdp_second}} \label{sec:sdp_second:main_proof}

\begin{proof}[Proof of Theorem~\ref{thm:main_sdp_second}]
According to Theorem~\ref{thm:algo_general}, there is an algorithm solving SDP in
\begin{align*}
&\wt{O} ( n^{0.5} \nu_{\max}^{0.5} \cdot (T_H + T_n + T_L + T_Z + \nnz(A) + \eta m_\lp m_{\max})^{0.5} \\
&\cdot (\Tmat(n_{\max}) + T_{\Delta_L,\max} + T_{H,\max} + \eta^2 m_{\max}^2)^{0.5} \log(R/(r\epsilon)))
\end{align*}
time.

It remains to compute the parameters $T_H$, $T_L$, $\eta$, $T_m$, $m_{\max}$, $n_{\max}$, $T_{\Delta_L, \max}$, $T_{H,\max}$, the details can be found in Lemma~\ref{lem:sdp_second_param}.

Let $A $ and $B$ be defined as:
\begin{align*}
    A: = & ~ T_H + T_n + T_L + T_Z + \nnz(A) + \eta m_{\lp} m_{\max}, \\
    B := & ~ \Tmat(n_{\max}) + T_{\Delta_L,\max} + T_{H,\max} + \eta^2 m_{\max}^2.
\end{align*}

We can show that
\begin{align}\label{eq:sdp_second_A}
    A = & ~ T_H + T_n + T_L + T_Z + \nnz(A) + \eta m_{\lp} m_{\max} \notag \\
    = & ~ O(n_{\sdp} \tau_{\sdp}^{2\omega}) + T_n + T_L + T_Z + \nnz(A) + \eta m_{\lp} m_{\max} \notag \\
    = & ~ O(n_{\sdp} \tau_{\sdp}^{2\omega}) + O(n_{\sdp} \tau_{\sdp}^{2\omega}) + T_L + T_Z + \nnz(A) + \eta m_{\lp} m_{\max} \notag \\
    = & ~ O(n_{\sdp} \tau_{\sdp}^{2\omega}) + T_L + T_Z + \nnz(A) + \eta m_{\lp} m_{\max} \notag \\
    = & ~ O(n_{\sdp} \tau_{\sdp}^{2\omega})  + O(n_{\sdp} \tau_{\sdp}^{2\omega}) + T_Z + \nnz(A) + \eta m_{\lp} m_{\max} \notag \\
    = & ~ O(n_{\sdp} \tau_{\sdp}^{2\omega}) + T_Z + \nnz(A) + \eta m_{\lp} m_{\max} \notag \\
    = & ~ O(n_{\sdp} \tau_{\sdp}^{2\omega}) + O(n_{\sdp} \tau_{\sdp}^{2\omega}) + \nnz(A) + \eta m_{\lp} m_{\max} \notag \\
    = & ~ O(n_{\sdp} \tau_{\sdp}^{2\omega}) + \nnz(A) + \eta m_{\lp} m_{\max} \notag \\
    = & ~ O(n_{\sdp} \tau_{\sdp}^{2\omega}) + O(n_{\sdp} \tau_{\sdp}^{4}) + \eta m_{\lp} m_{\max} \notag \\
    = & ~ O(n_{\sdp} \tau_{\sdp}^{2\omega}) + O(n_{\sdp} \tau_{\sdp}^{4}) + O(n_{\sdp} \tau_{\sdp}^{4}) \notag \\
    = & ~ O(n_{\sdp} \tau_{\sdp}^{2\omega})
\end{align}
where the second step follows from Lemma~\ref{lem:sdp_second_param}\ref{item:lem:sdp_second_TH} ($T_H = O(n_{\sdp} \tau_{\sdp}^{2\omega})$) and the third step follows from Lemma~\ref{lem:sdp_second_param}\ref{item:lem:sdp_second_n} ($T_n = O(n_{\sdp} \tau_{\sdp}^{2\omega})$), the forth step follows from merging the terms, the fifth step follows from Lemma~\ref{lem:sdp_second_param}\ref{item:lem:sdp_second_TL} ($T_L = \wt{O}(\tau_{\sdp} \tau^{2\omega})$, the sixth step follows from merging the terms, the seventh step follows from $T_Z = O(n_{\sdp} \tau_{\sdp}^4)$, the eighth step follows from merging the terms, the ninth step follows   Lemma~\ref{lem:sdp_second_param}\ref{item:lem:sdp_second_nnz} ($\nnz(A) = n_{\sdp} \tau_{\sdp}^4$), and the tenth step follows from Lemma~\ref{lem:sdp_second_param}\ref{item:lem:sdp_second_eta} and Lemma~\ref{lem:sdp_second_param}\ref{item:lem:sdp_second_m}.

For the term $B$, we have
\begin{align}\label{eq:sdp_second_B}
    B = & ~ \Tmat(n_{\max}) + T_{\Delta_L,\max} + T_{H,\max} + \eta^2 m_{\max}^2 \notag \\
    = & ~ O( \tau_{\lp}^{\omega} ) + T_{\Delta_L,\max} + T_{H,\max} + \eta^2 m_{\max}^2 \notag \\
    = & ~ O( \tau_{\lp}^{\omega} ) + O( \tau_{\lp}^{\omega} ) + T_{H,\max} + \eta^2 m_{\max}^2 \notag \\
    = & ~ O( \tau_{\lp}^{\omega} ) + O( \tau_{\lp}^{\omega} ) + O( \tau_{\lp}^{\omega} ) + \eta^2 m_{\max}^2 \notag \\
    = & ~ O( \tau_{\lp}^{\omega} ) + O( \tau_{\lp}^{\omega} ) + O( \tau_{\lp}^{\omega} ) + \wt{O}( \tau_{\lp}^2 ) \notag \\
    = & ~ O(\tau_{\lp}^{\omega}) \notag \\
    = & ~ O(\tau_{\sdp}^{2\omega})
\end{align}
where the second step follows Lemma~\ref{lem:sdp_second_param}\ref{item:lem:sdp_second_n} ($n_{\max} = \tau_{\lp}$), and the third step follows from Lemma~\ref{lem:sdp_second_param}\ref{item:lem:sdp_second_TDeltaLmax} ($T_{\Delta L, \max} = O(\tau_{\lp}^{\omega})$), the forth step follows from Lemma~\ref{lem:sdp_second_param}\ref{item:lem:sdp_second_TH} ($T_{H,\max} = O(\tau_{\lp}^{\omega})$), the fifth step follows from Lemma~\ref{lem:sdp_second_param}\ref{item:lem:sdp_second_eta} ($\eta = \wt{O}(1) $ and $m_{\max} = ({\tau}_{\lp})$), and the last step follows from $\tau_{\lp} = \tau_{\sdp}^2$.

Finally, we have
\begin{align*}
    & ~ \wt{O}( n^{0.5} \nu_{\max}^{0.5} \cdot A^{0.5} \cdot B^{0.5} \cdot \log(1/\epsilon) ) \\
    = & ~ \wt{O} (n_{\sdp}^{0.5} \tau_{\sdp}^{0.5} \cdot A^{0.5} \cdot B^{0.5} \cdot \log(1/\epsilon)) \\
    = & ~ \wt{O} (n_{\sdp}^{0.5} \tau_{\sdp}^{0.5} \cdot (n_{\sdp} \tau_{\sdp}^{2\omega})^{0.5} \cdot (\tau_{\sdp}^{2\omega})^{0.5} \cdot \log(1/\epsilon)) \\
    = & ~ \wt{O}( n_{\sdp} \tau_{\sdp}^{2\omega+0.5} \cdot \log(1/\epsilon) )
\end{align*}
where the first step follows from $n = n_{\sdp}$ and Lemma~\ref{lem:sdp_second_param}\ref{item:lem:sdp_second_eta} ($\nu_{\max} = O(\tau_{\sdp})$), the second step follows from $A = \wt{O}(n_{\sdp} \tau_{\sdp}^{2\omega})$ (see Eq.~\eqref{eq:sdp_second_A}) and $B=\wt{O}(\tau_{\sdp}^{2\omega})$ (see Eq.~\eqref{eq:sdp_second_B}). 

Thus, we complete the proof.

\end{proof}

%% file: sdp_general.tex
\section{Decomposable SDP} \label{sec:sdp_general}
In this section we give a more general version of Theorem~\ref{thm:main_sdp_second}, which can handle more general decomposable SDPs.
Outline of this section is as follows.

\begin{itemize}
    \item In Section~\ref{sec:sdp_general:main}, we present the main result of this section, Theorem~\ref{thm:main_sdp_general}.
    \item In Section~\ref{sec:sdp_general:reduction}, we reduce decomposable SDPs to a form which can be handled by Theorem~\ref{thm:algo_general}.
    \item In Section~\ref{sec:sdp_general:choice_of_parameters}, we state and prove parameters needed to apply Theorem~\ref{thm:algo_general}.
  
    \item In Section~\ref{sec:sdp_general:main_proof}, we plug in all parameters and finish proof of Theorem~\ref{thm:main_sdp_general}.
\end{itemize}

\subsection{Main Statement} \label{sec:sdp_general:main}
We consider semidefinite programs of form \eqref{eqn:sdp-primal-restated}.
Instead of assumptions in Definition~\ref{defn:sdp-assumptions}, we make the following assumptions on the program.

\begin{definition} \label{defn:sdp-general-assumptions}
We make the following assumptions on program \eqref{eqn:sdp-primal-restated}.
\begin{itemize}
    \item Assume that constraint matrices $A_1,\ldots, A_{m_\sdp} \in \R^{n_\sdp \times n_{\sdp}}$ are linearly independent.
    \item Assume that we are given a tree decomposition (Definition~\ref{defn:treewidth}) $(T, J_1,\ldots, J_k)$ of the sparsity graph (Definition~\ref{def:sparsity-graph}) of \eqref{eqn:sdp-primal-restated} with maximum bag size $\tau_\sdp$.
    Note that here we work with the sparsity graph, not the SDP graph (Definition~\ref{def:sdp-graph}).
    \item Assume that for every constraint $A_i$, we have a connected (in the given tree decomposition) set $S_i$ of bags, such that the union of these bags contains the support of $A_i$, i.e.,
    \begin{align*}
        \{(u,v)\in [n_\sdp] \times [n_\sdp] : A_{i,u,v}\ne 0\} \subseteq \bigcup_{j\in S_i} J_j\times J_j.
    \end{align*}
    Let $\gamma_{\max}$ be the maximum number of constraints a bag correspond to.
    \item Assume that any feasible solution $X\in \bS^{n_\sdp}$ satisfies $\|X\|_\op \le R$.
    \item Assume that there exists $Z\in \bS^{n_\sdp}$ such that $A\bullet Z = b$ and $\lambda_{\min}(Z) \ge r$.
\end{itemize}
\end{definition}

\begin{theorem}\label{thm:main_sdp_general}
Under assumptions in Definition~\ref{defn:sdp-general-assumptions}, for any $0<\epsilon\le \frac 12$, there is an algorithm that outputs a solution $U\in \R^{n_\sdp \times \tau_\sdp}$ such that for $X = UU^\top$, we have $A\bullet X = b$ and
\begin{align*}
    C\bullet X \le \min_{A\bullet X' = b, X' \in \bS^{n_\sdp}} C\bullet X + \epsilon \|C\|_2 R
\end{align*}
in expected time
\begin{align*}
    \wt O(n \cdot \tau_\sdp^{0.5} (\tau_\sdp^2+ \gamma_{\max})^\omega)
\end{align*}
where $\wt O$ hides $n^{o(1)}$ terms.
\end{theorem}

\subsection{Reduce Decomposable SDP to General Treewidth Program} \label{sec:sdp_general:reduction}
In this section we reduce program~\eqref{eqn:sdp-primal-restated} satisfying Definition~\ref{defn:sdp-general-assumptions} to form \eqref{eqn:lp_general}.

Let $T, J_1,\ldots, J_n$ be the given tree decomposition. Using Lemma~\ref{lem:tree-decomp-const-deg}, we can WLOG assume that $T$ has maximum degree $O(1)$.

Lemma~\ref{lem:zl18-reduction-general} tells us that we can reduce program~\ref{eqn:sdp-primal-restated} to the following program.
\begin{align} \label{eqn:sdp-general-reduced}
    \min \qquad & C \bullet X\\
    \text{s.t.} \qquad & A_i \bullet X = b_i \qquad \forall i\in [m_\sdp] \nonumber\\
    & X_{J_i,J_i} \succeq 0 \qquad  \forall i\in [n] \nonumber
\end{align}

\begin{lemma}[{\cite[Section 3.3]{zl18}}] \label{lem:zl18-reduction-general}
Under assumptions in Definition~\ref{defn:sdp-general-assumptions}, the optimal value of program \eqref{eqn:sdp-primal-restated} is equal to the optimal value of \eqref{eqn:sdp-general-reduced}.

Furthermore, given a solution $X_1,\ldots, X_n$ of program \eqref{eqn:sdp-ctc-restated}, there is an algorithm that takes $O(n_\sdp \tau_\sdp^3)$ time and outputs a matrix $U\in \bR^{n_\sdp \times \tau_\sdp}$ such that for $X=UU^\top$, $X$ is a solution of Eq.~\eqref{eqn:sdp-primal} with the same objective value.
\end{lemma}

Therefore, similar to Section \ref{sec:sdp_first:reduction}, we decompose $X$ into $n$ smaller PSD variables $X_1,\ldots, X_n$, with $X_i \in \bS^{J_i}$, satisfying overlapping constraints.

It remains to decompose the constraints $A_i$ to handle the new variables.
Recall that we are given the set $S_i$.
Because union of $(J_j \times J_j)_{j\in S_i}$ contains the support of $i$, we can find constraints $A_{i,1}, \ldots, A_{i,|S_i|}$ with $A_{i,j} \in \R^{J_j \times J_j}$ such that
\begin{align*}
    A_i \bullet X = \sum_{j\in S_i} A_{i,j} \bullet X_{J_j,J_j}.
\end{align*}
We also decompose $C$ into $C_1,\ldots, C_n$ with $C_i \in \R^{J_i\times J_i}$
such that
\begin{align*}
    C\bullet X = \sum_{j\in [n]} C_j \bullet X_{J_j, J_j}.
\end{align*}
In this way, we further reduce \eqref{eqn:sdp-general-reduced} into the following form.
\begin{align} \label{eqn:sdp-general-lp-form}
    \min \qquad & \sum_{j\in [n]} C_j \bullet X_j\\
    \text{s.t.} \qquad & \sum_{j\in S_i} A_{i,j} \bullet X_j = b_i \qquad \forall i\in [m_\sdp] \nonumber\\
    & \cN_{i,j}(X_i) =\cN_{j,i}(X_j) \qquad \forall (i,j)\in E(T) \nonumber\\
    & X_j \succeq 0 \qquad  \forall j\in [n] \nonumber
\end{align}
where $X_j\in \R^{|J_j|\times |J_j|}$, $A_{i,j} \in \R^{|J_j| \times |J_j|}$, $\cN$ constraints asserts consistency between different bags, $\cN_{i,j}\in \R^{|J_i\cap J_j|^2 \times |J_i|^2}$.

Program~\eqref{eqn:sdp-general-lp-form} is a form which can be handled by Theorem~\ref{thm:algo_general}.
Let us compute the treewidth of program~\eqref{eqn:sdp-general-lp-form}.

\begin{lemma} \label{lem:sdp-general-reduced-treewidth}
We can construct a treewidth decomposition of the LP dual graph (Definition~\ref{def:general_lp_treewidth}) with maximum bag size $\tau_\lp = O(\tau_\sdp^2 + \gamma_{\max})$ and at most $n$ blocks.
\end{lemma}
\begin{proof}
For every $j\in [n]$, we define a bag $B_j$ containing
\begin{itemize}
    \item all type-$A$ constraints $A_i$ with $j\in S_i$,
    \item all type-$N$ constraints $\cN_{i,j}$ with $(i,j)\in T$.
\end{itemize}
We connected $B_i$ with $B_j$ if and only if $(i,j)\in T$.

By assumption, number of type-$A$ constraints in a bag is at most $\gamma_{\max}$.
Number of type-$N$ constraints in a bag is at most $O(|J_j|^2)$ because $T$ has constant maximum degree.
So $|B_j| \le O(|J_j|^2 + \gamma_{\max})$.
The maximum bag size is $\tau_\lp = O(\max_j |J_j|^2 + \gamma_{\max}) = O(\tau_\sdp^2 + \gamma_{\max})$.

Finally, we prove that $(T, B_1,\ldots, B_n)$ is a valid tree decomposition of the LP dual graph.
Recall that we assume that every $S_i$ is connected in $T$. So the second condition in Definition~\ref{defn:treewidth} is satisfied.
For any two constraints sharing a variable $X_j$, they must both be in $B_j$. So the first condition in Definition~\ref{defn:treewidth} is satisfied.
Therefore $(T, B_1,\ldots, B_n)$ is a valid tree decomposition.
\end{proof}

\subsection{Choice of Parameters} \label{sec:sdp_general:choice_of_parameters}
In this section we present value of parameters needed to apply Theorem~\ref{thm:algo_general}.

\begin{lemma}[Choice of parameters in the general result]\label{lem:sdp_general_param}
Under the setting of Theorem~\ref{thm:main_sdp_general}, we can choose the following set of parameters.

\begin{enumerate}[label=(\roman*)]
    \item \label{item:lem:sdp_general_n} $n = O(n_\sdp)$, $n_\lp = O(n_\sdp\tau_\sdp^2)$, $n_{\max} = O(\tau_\sdp^2)$.
    \item \label{item:lem:sdp_general_m} $m = O(n_\sdp)$, $m_\lp = O(n_\sdp (\tau_\sdp^2+\gamma_{\max}))$, $m_{\max} = O(\tau_\sdp^2+\gamma_{\max})$.
    \item \label{item:lem:sdp_general_taulp} $\tau_\lp = O(\tau_\sdp^2+\gamma_{\max})$.
    \item \label{item:lem:sdp_general_eta} $\eta = \wt O(1)$.
    \item \label{item:lem:sdp_general_nu} $\nu_{\max} = O(\tau_\sdp)$.
    \item \label{item:lem:sdp_general_nnz} $\nnz(A) = O(n_\sdp \tau_\sdp^2(\tau_\sdp^2+\gamma_{\max}))$.
    \item \label{item:lem:sdp_general_TnTm} $T_n = \wt O(n_\sdp \tau_\sdp^{2\omega})$, $T_m = \wt O(n_\sdp (\tau_\sdp^2+\gamma_{\max})^\omega)$.
    \item \label{item:lem:sdp_general_TH} $T_H = \wt O(n_\sdp \tau_\sdp^{2\omega})$, $T_{H,\max} = \wt O(\tau_\sdp^{2\omega})$.
    \item \label{item:lem:sdp_general_TL} $T_L = \wt O(n_\sdp (\tau_\sdp^2+\gamma_{\max})^\omega)$.
    \item \label{item:lem:sdp_general_TZ} $T_Z = \wt O(n_\sdp (\tau_\sdp^2+\gamma_{\max})^2)$.
    \item \label{item:lem:sdp_general_TDeltaLmax} $T_{\Delta L,\max} = \wt O((\tau_\sdp^2+\gamma_{\max})^\omega)$.
\end{enumerate}
\end{lemma}
\begin{proof}
\begin{enumerate}[label=(\roman*)]
    \item By discussion in Section~\ref{sec:sdp_general:reduction}.
    \item For block structure, we apply Lemma~\ref{lem:sdp_second_block_elim}.
    So $m = O(n_\sdp)$, $m_{\max} = O(\tau_\lp) = O(\tau^2+ \gamma_{\max})$.
    \item By Lemma~\ref{lem:sdp-general-reduced-treewidth}.
    \item By Lemma~\ref{lem:sdp_second_block_elim}.
    \item Same as Lemma~\ref{lem:sdp_first_param}\ref{item:lem:sdp_first_nu}.
    \item Every $X_j$ is in at most $\tau_\lp = O(\tau_\sdp^2 + \gamma_{\max})$ constraints, so $\nnz(A) = O(n_\sdp \tau_\sdp^2 \cdot \tau_\lp) = O(n_\sdp \tau_\sdp^2(\tau_\sdp^2 + \gamma_{\max}))$.
    \item Bounds on $T_n$ and $T_m$ are direct consequences of \ref{item:lem:sdp_second_n}, \ref{item:lem:sdp_second_m}.
    \item Same as Lemma~\ref{lem:sdp_first_param}\ref{item:lem:sdp_first_TH}.
    \item By Lemma~\ref{lem:block_cholesky_decomposition}, we have $T_L = \wt O(m \tau_\lp^\omega) = \wt O(n_\sdp (\tau_\sdp^2 + \gamma_{\max})^\omega)$.
    \item By Lemma~\ref{lem:generic_bound_TZ}, we have $T_Z = O(\eta^2 m m_{\max}^2) = \wt O(n_\sdp (\tau_\sdp^2 + \gamma_{\max})^2)$.
    \item  By Lemma~\ref{lem:block_cholesky_update}, we have $T_{\Delta L, \max} = \wt O(\tau_\lp^\omega) = \wt O((\tau_\sdp^2 + \gamma_{\max})^\omega)$.
\end{enumerate}
\end{proof}

\subsection{Proof of Theorem~\ref{thm:main_sdp_general}} \label{sec:sdp_general:main_proof}
In this section, we prove Theorem~\ref{thm:main_sdp_general}.

\begin{proof}[Proof of Theorem~\ref{thm:main_sdp_general}]

It remains to compute the parameters $T_H$, $T_L$, $\eta$, $T_m$, $m_{\max}$, $n_{\max}$, $T_{\Delta_L, \max}$, $T_{H,\max}$, the details can be found in Lemma~\ref{lem:sdp_second_param}.

Let $A $ and $B$ be defined as:
\begin{align*}
    A: = & ~ T_H + T_n + T_L + T_Z + \nnz(A) + \eta m_{\lp} m_{\max}, \\
    B := & ~ \Tmat(n_{\max}) + T_{\Delta_L,\max} + T_{H,\max} + \eta^2 m_{\max}^2.
\end{align*}

We can show that
\begin{align}\label{eq:sdp_general_A}
    A = & ~ T_H + T_n + T_L + T_Z + \nnz(A) + \eta m_{\lp} m_{\max} \notag \\
    = & ~ O(n_{\sdp} \tau_{\sdp}^{2\omega}) + T_n + T_L + T_Z + \nnz(A) + \eta m_{\lp} m_{\max} \notag \\
    = & ~ O(n_{\sdp} \tau_{\sdp}^{2\omega}) + O(n_{\sdp} \tau_{\sdp}^{2\omega}) + T_L + T_Z + \nnz(A) + \eta m_{\lp} m_{\max} \notag \\
    = & ~ O(n_{\sdp} \tau_{\sdp}^{2\omega}) + T_L + T_Z + \nnz(A) + \eta m_{\lp} m_{\max} \notag \\
    = & ~ O(n_{\sdp} \tau_{\sdp}^{2\omega})  + O(n_{\sdp} (\tau_{\sdp}^{2} + \gamma_{\max} )^{\omega} ) + T_Z + \nnz(A) + \eta m_{\lp} m_{\max} \notag \\
    = & ~ O(n_{\sdp} ( \tau_{\sdp}^{2} + \gamma_{\max})^{\omega}) + T_Z + \nnz(A) + \eta m_{\lp} m_{\max} \notag \\
    = & ~ O(n_{\sdp} ( \tau_{\sdp}^{2} + \gamma_{\max})^{\omega}) + O(n_{\sdp} \tau_{\sdp}^{2\omega}) + \nnz(A) + \eta m_{\lp} m_{\max} \notag \\
    = & ~ O(n_{\sdp} ( \tau_{\sdp}^{2} + \gamma_{\max})^{\omega}) + \nnz(A) + \eta m_{\lp} m_{\max} \notag \\
    = & ~ O(n_{\sdp} ( \tau_{\sdp}^{2} + \gamma_{\max})^{\omega}) + O(n_{\sdp} \tau_{\sdp}^{2} ( \tau_{\sdp}^2 + \gamma_{\max} ) ) + \eta m_{\lp} m_{\max} \notag \\
    = & ~ O(n_{\sdp} ( \tau_{\sdp}^{2} + \gamma_{\max})^{\omega}) + O(n_{\sdp} \tau_{\sdp}^{2} (\tau_{\sdp}^2 + \gamma_{\max}) ) + O(n_{\sdp} ( \tau_{\sdp}^{2} + \gamma_{\max} )^2 ) \notag \\
    = & ~ O(n_{\sdp} ( \tau_{\sdp}^{2} + \gamma_{\max})^{\omega})
\end{align}
where the second step follows from Lemma~\ref{lem:sdp_general_param}\ref{item:lem:sdp_general_TH} ($T_H = O(n_{\sdp} \tau_{\sdp}^{2\omega})$) and the third step follows from Lemma~\ref{lem:sdp_general_param}\ref{item:lem:sdp_general_n} ($T_n = O(n_{\sdp} \tau_{\sdp}^{2\omega})$), the forth step follows from merging the terms, the fifth step follows from Lemma~\ref{lem:sdp_general_param}\ref{item:lem:sdp_general_TL} ($T_L = \wt{O}(\tau_{\sdp} ( \tau_{\sdp}^{2} + \gamma_{\max} )^{\omega} )$, the sixth step follows from merging the terms, the seventh step follows from $T_Z = O(n_{\sdp} ( \tau_{\sdp}^2 + \gamma_{\max})^2 )$, the eighth step follows from merging the terms, the ninth step follows   Lemma~\ref{lem:sdp_general_param}\ref{item:lem:sdp_general_nnz} ($\nnz(A) = n_{\sdp} \tau_{\sdp}^2 ( \tau_{\sdp}^2 + \gamma_{\max}) $), and the tenth step follows from Lemma~\ref{lem:sdp_general_param}\ref{item:lem:sdp_general_eta} and Lemma~\ref{lem:sdp_general_param}\ref{item:lem:sdp_general_m}.

For the term $B$, we have
\begin{align}\label{eq:sdp_general_B}
    B = & ~ \Tmat(n_{\max}) + T_{\Delta_L,\max} + T_{H,\max} + \eta^2 m_{\max}^2 \notag \\
    = & ~ O( \tau_{\sdp}^{2 \omega} ) + T_{\Delta_L,\max} + T_{H,\max} + \eta^2 m_{\max}^2 \notag \\
    = & ~ O( \tau_{\sdp}^{2\omega} ) + O( ( \tau_{\sdp}^{2} + \gamma_{\max} )^{\omega} ) + T_{H,\max} + \eta^2 m_{\max}^2 \notag \\
    = & ~ O( \tau_{\sdp}^{2\omega} ) + O( ( \tau_{\sdp}^{2} + \gamma_{\max} )^{\omega} ) + O( \tau_{\sdp}^{2\omega} ) + \eta^2 m_{\max}^2 \notag \\
    = & ~ O( \tau_{\sdp}^{2\omega} ) + O( ( \tau_{\sdp}^{2} + \gamma_{\max} )^{\omega} ) + O( \tau_{\sdp}^{2\omega} ) + \wt{O}( \tau_{\sdp}^4 ) \notag \\
    = & ~ O( (\tau_{\sdp}^{2} + \gamma_{\max} )^{\omega} )
\end{align}
where the second step follows Lemma~\ref{lem:sdp_general_param}\ref{item:lem:sdp_general_n} ($n_{\max} = \tau_{\sdp}^2$), and the third step follows from Lemma~\ref{lem:sdp_general_param}\ref{item:lem:sdp_general_TDeltaLmax} ($T_{\Delta L, \max} = O(\tau_{\lp}^{\omega})$), the forth step follows from Lemma~\ref{lem:sdp_general_param}\ref{item:lem:sdp_general_TH} ($T_{H,\max} = O(\tau_{\lp}^{\omega})$), the fifth step follows from Lemma~\ref{lem:sdp_general_param}\ref{item:lem:sdp_general_eta} ($\eta = \wt{O}(1) $ and $m_{\max} = ({\tau}_{\sdp}^2)$). 

Finally, we have
\begin{align*}
    & ~ \wt{O}( n^{0.5} \nu_{\max}^{0.5} \cdot A^{0.5} \cdot B^{0.5} \cdot \log(1/\epsilon) ) \\
    = & ~ \wt{O} (n_{\sdp}^{0.5} \tau_{\sdp}^{0.5} \cdot A^{0.5} \cdot B^{0.5} \cdot \log(1/\epsilon)) \\
    = & ~ \wt{O} (n_{\sdp}^{0.5} \tau_{\sdp}^{0.5} \cdot (n_{\sdp} ( \tau_{\sdp}^{2} + \gamma_{\max} )^{\omega} )^{0.5} \cdot (\tau_{\sdp}^{2\omega})^{0.5} \cdot \log(1/\epsilon)) \\
    = & ~ \wt{O}( n_{\sdp} \tau_{\sdp}^{ 0.5} \cdot  (\tau_{\sdp}^2 + \gamma_{\max})^{\omega} \cdot \log(1/\epsilon) )
\end{align*}
where the first step follows from $n = n_{\sdp}$ and Lemma~\ref{lem:sdp_general_param}\ref{item:lem:sdp_general_eta} ($\nu_{\max} = O(\tau_{\sdp})$), the second step follows from $A = \wt{O}(n_{\sdp} (  \tau_{\sdp}^{2} + \gamma_{\max})^{\omega} )$ (see Eq.~\eqref{eq:sdp_general_A}) and $B=\wt{O}( ( \tau_{\sdp}^{2} + \gamma_{\max} )^{\omega} )$ (see Eq.~\eqref{eq:sdp_general_B}). 

Thus, we complete the proof.
\end{proof}

%% file: lp.tex
\section{Improving Running Time for Linear Programming}
\label{sec:lp}

In this section, we discuss our result for linear programming,

\begin{itemize}
    \item In Section~\ref{sec:lp:main}, we present the main result in this Section, Theorem~\ref{thm:main_lp_formal}.
    \item In Section~\ref{sec:lp:choice_of_parameters}, we state and prove parameters needed to apply Theorem~\ref{thm:algo_general}.
    \item In Section~\ref{sec:lp:cholesky}, we develop improved algorithms for Cholesky-related computation.
    \item In Section~\ref{sec:lp:main_proof}, we plug in all parameters and finish proof of Theorem~\ref{thm:main_lp_formal}.
\end{itemize}

\subsection{Main Statement}\label{sec:lp:main}
We consider linear program of form~\eqref{eqn:lp-primal}. Let us restate it here for clarity.
\begin{align}\label{eqn:lp-primal-restated}
    \min \qquad & c^\top x\\
    \text{s.t.} \qquad & A x = b \nonumber\\
    & \ell \le x \le u \nonumber
\end{align}
\begin{definition} \label{defn:lp-assumptions}
We make the following assumptions on program~\eqref{eqn:lp-primal-restated}.
\begin{itemize}
    \item Assume that the constraint matrix $A\in \R^{m \times n}$ has full column rank.
    \item Assume that we are given a tree decomposition (Definition~\ref{defn:treewidth}) of the LP dual graph (Definition~\ref{defn:lp-dual-graph}) with maximum bag size $\tau$.
    \item Assume that there exists $x$ such that $A x = b$ and $\ell +r \le x \le u-r$, for some $r > 0$.
\end{itemize}
\end{definition}

\begin{theorem}[Our LP result, formal version of Theorem~\ref{thm:main_lp_informal}]\label{thm:main_lp_formal}
Under assumptions in Definition~\ref{defn:lp-assumptions}, for any $0<\epsilon\le \frac 12$, there is an algorithm that outputs a solution $x\in \R^n$ such that $Ax=b$, $\ell \le x\le u$, and
\begin{align*}
    c^\top x \le \min_{Ax'=b, \ell \le x'\le u, x'\in \R^n} c^\top x' + \epsilon L R
\end{align*}
where $L=\|c\|_2$, $R=\|u-\ell\|_2$, in expected time
\begin{align*}
\wt O( n \cdot \tau^{(\omega+1)/2} \log(R/(r\epsilon))).
\end{align*}
\end{theorem}

\subsection{Choice of Parameters}\label{sec:lp:choice_of_parameters}
\begin{lemma}[Choice of parameters in the LP result]\label{lem:lp_param}
Under the setting of Theorem~\ref{thm:main_lp_formal}, we can choose the following set of parameters.
\begin{enumerate}[label=(\roman*)]
    \item \label{item:lem:lp_n} $n_\lp=n$, $n_{\max} = 1$.
    \item \label{item:lem:lp_m} $m = O(n)$, $m_\lp = O(n)$, $m_{\max} = 1$.
    \item \label{item:lem:lp_taulp} $\tau_\lp = O(\tau)$.
    \item \label{item:lem:lp_eta} $\eta = \wt O(\tau)$.
    \item \label{item:lem:lp_nu} $\nu_{\max} = O(1)$.
    \item \label{item:lem:lp_nnz} $\nnz(A) = O(n \tau)$.
    \item \label{item:lem:lp_TnTm} $T_n = O(n)$, $T_m = O(n)$.
    \item \label{item:lem:lp_TH} $T_H = O(n)$, $T_{H,\max} = O(1)$.
    \item \label{item:lem:lp_TL} $T_L = O(n \tau^{\omega-1})$.
    \item \label{item:lem:lp_TZ} $T_Z = O(n \tau^{\omega-1})$.
    \item \label{item:lem:lp_TDeltaLmax} $T_{\Delta L,\max} = O(\tau^2)$.
\end{enumerate}
\end{lemma}
\begin{proof}
\begin{enumerate}[label=(\roman*)]
    \item We let every variable block contain a single element.
    \item Because $A$ is of full rank, $m\le n$. We let every constraint block contain a single element.
    \item The LP dual graph for program~\ref{eqn:lp_general} and \ref{eqn:lp-primal-restated} are the same graphs.
    \item Follows from Lemma~\ref{lem:sdp_first_block_elim}.
    \item We use the log barrier $\phi_i(x_i) = -\log(u_i-x_i)-\log(x_i-l_i)$.
    \item Every variable is contained in at most $\tau$ constraints, so $\nnz(A) = O(n\tau)$.
    \item Bounds on $T_n$ and $T_m$ are direct consequences of \ref{item:lem:lp_n}, \ref{item:lem:lp_m}.
    \item Computing a single Hessian takes $O(1)$ time.
    \item Follows from Lemma~\ref{lem:lp:TL}.
    \item Follows from Lemma~\ref{lem:lp:TZ}. There is one caveat: in the statement of Lemma~\ref{lem:lp:TZ}, we use a block elimination tree $\cT_2$ constructed using Lemma~\ref{lem:sdp_second_block_elim}. For definition of $T_Z$, we use block elimination tree $\cT_1$ constructed using Lemma~\ref{lem:sdp_first_block_elim}. However by examining both algorithms, we see that every path in $\cT_1$ is contained in a path in $\cT_2$, and vice versa. So Lemma~\ref{lem:lp:TZ} can be applied here.
    \item Follows from Lemma~\ref{lem:known_cholesky_TDeltaLmax}.
\end{enumerate}
\end{proof}

\subsection{Improvement for Cholesky-Related Computation}\label{sec:lp:cholesky}
In this section we prove Lemma~\ref{lem:lp:TL} and Lemma~\ref{lem:lp:TZ}.

In Lemma~\ref{lem:lp_param} we choose constraint block structure $(1,\ldots, 1)$ for the purpose of smaller updating time (e.g., $T_{\Delta L,\max}$).
Nevertheless, we can utilize another block structure to achieve faster initialization (e.g., $T_L$, $T_Z$).

\begin{lemma} \label{lem:lp:TL}
Let $M\in \R^{m\times m}$ be a PSD matrix.
Let $G$ be its adjacency graph (i.e., there is an edge $(i,j)$ if and only if $M_{i,j}\ne 0$).
Suppose we are given a tree decomposition of $G$ with maximum bag size $\tau$.
Then we can compute the Cholesky factorization $M = LL^\top$ in $\wt O(m \tau^{\omega-1})$ time.
\end{lemma}

\begin{proof}
Using Lemma~\ref{lem:sdp_second_block_elim}, we can construct a block elmination tree $\cT$ with constant maximum degree, maximum depth $\wt O(1)$ and maximum bag size $O(\tau)$.
However, this block elimination tree could potentially have $\Omega(m)$ vertices.
So we use Lemma~\ref{lem:lp:reduce_num_blocks} to compute a new block elimination tree $\cT'$ with number of blocks $O(m/\tau)$.
Finally, we use Lemma~\ref{lem:block_cholesky_decomposition} to compute Cholesky factorization using $\cT'$, which takes $\wt O(m/\tau \cdot \tau^\omega) = \wt O(m \tau^{\omega-1})$ time.
\end{proof}

\begin{lemma}\label{lem:lp:reduce_num_blocks}
Given a block elimination tree $(\cT, B_1,\ldots, B_b)$ with maximum degree $O(1)$, maximum depth $\wt O(1)$, maximum block size $\tau$, and total block size $m$, we can construct a block elimination tree $(\cT', B'_1,\ldots, B'_{b'})$ with maximum degree $O(1)$, maximum depth $\wt O(1)$, maximum block size $O(\tau)$, and $O(m/\tau)$ blocks in total (i.e., $b' = O(m/\tau)$).
\end{lemma}
\begin{proof}
We perform a bottom up process to construct the new tree $\cT'$.

For each node $v$ from bottom to up, if any of its children is in a block of size smaller than $\tau$, then we combine their blocks with $v$ (if there are multiple such children, then all of their blocks are merged).

In this way every block (except possibly the root block) has size at least $\tau$.
Because $\cT$ has $O(1)$ maximum degree, every new block has $O(\tau)$ size.
So the number of blocks in $\cT'$ is $O(m/\tau)$.

Because $\cT'$ preserves all ancestor-descendant relationships in $\cT$, condition in Lemma~\ref{lem:prove_elim_tree} is satisfied by $\cT'$. So $\cT'$ is still a block elimination tree.

The only remaining problem is that $\cT'$ could have nodes with large degree.
For every node $v$ with number of children $c$ larger than $3$, we replace this node with a perfect binary tree with $c$ leaves, where the root node is $v$, and all other nodes are empty. Then we link $v$'s original children to leaves of the perfect binary tree.

The number of added nodes is at most $O(m/\tau)$. So this final tree satisfies all requirements.
\end{proof}

\begin{lemma} \label{lem:lp:TZ}
Under the setting of Lemma~\ref{lem:lp:TL}, there is an algorithm to compute $L^{-1} v_i$ for all $i\in [m]$, where $v_i$ is supported on a single path in $\cT$, in $O(m \tau^{\omega-1})$ time.
\end{lemma}
\begin{proof}
By our construction in proof of Theorem~\ref{lem:lp:TL}, every block in $\cT'$ is the union of several blocks in $\cT$.
For $i\in [m]$, let $b_i\in [b']$ be the block it belongs to in $\cT'$.
Note that $\cT'$ preserves all ancestor-descendant relationships in $\cT$.
So for every $i\in [b]$, we have
\begin{align*}
    \bigcup_{j\in \cP^{\cT}(i)} B_j \subseteq \bigcup_{j\in \cP^{\cT'}(b_i)} B'_j.
\end{align*}
So every path in $\cT$ is contained in a path in $\cT'$.

Therefore, we only need to solve the following problem:
Compute $L^{-1} v_i$ for $i\in [m]$, where every $v_i$ is supported on a single path in $\cT'$.
Because $\cT'$ has maximum depth $\wt O(1)$, we can assume that every $v_i$ is supported on a single block in $\cT'$ (with an $\wt O(1)$ factor loss in running time).
Then the desired result follows from combining Lemma~\ref{lem:fast_cholesky_inverse} and Lemma~\ref{lem:fast_cholesky_inverse_vector} on $\cT'$ and $L$.
\end{proof}

\begin{lemma}[Inverse of Cholesky factor] \label{lem:fast_cholesky_inverse}
Suppose we have a block elimination tree $(\cT, B_1,\ldots, B_b)$ with maximum depth $\wt O(1)$, maximum block size $O(\tau)$ and the total number of blocks is $b=O(m/\tau)$.
Let $L\in \R^{m\times m}$ be a lower-triangular matrix compatible with $\cT$, i.e., for $i,j\in [b]$, $L_{i,j}\ne 0$ only if $i\in \cP(j)$.
(Note that here $L_{i,j}$ denotes the $i$-th block-row and $j$-th block-column.)

Then $L^{-1}$ is also a lower-triangular matrix compatible with $\cT$, and we can compute $L^{-1}$ in $\wt O(n \tau^{\omega-1})$ time.
\end{lemma}
\begin{proof}

Run Algorithm \ref{alg:block_cholesky_inverse}. Correctness is obvious.
Let us focus on running time.

By induction, we can see that $X_{j,k} \ne 0$ only when $j\in \cP(k)$.
So we perform $O(1)$ matrix multiplications for every tuple $(i,j,k)$ with $i\in \cP(j)$, $j\in \cP(k)$. Because maximum depth is $\wt O(1)$, number of such triples is $\wt O(m/\tau)$.
So total running time is $\wt O(m/\tau \cdot \tau^{\omega-1}) = \wt O(m \tau^{\omega-1}).$
\end{proof}

\begin{algorithm}[!ht] \caption{} \label{alg:block_cholesky_inverse}
\begin{algorithmic}[1] 
\Procedure{BlockCholeskyInverse}{$ $}
 \Comment{Lemma~\ref{lem:fast_cholesky_inverse}}
\State $V\gets I, X\gets 0$
\For{$j \gets 1$ to $b$}
\State $X_{j,*} \gets L_{j,j}^{-1} V_{j,*}$
\State $V \gets V - L_{*,j} X_{j,*}$
\EndFor
\State \Return $X$
\EndProcedure
\end{algorithmic}
\end{algorithm}

\begin{lemma} \label{lem:fast_cholesky_inverse_vector}
Work under the setting of Lemma~\ref{lem:fast_cholesky_inverse}.
Let $P$ be a lower-triangular matrix compatible with $\cT$.
Let $v_1,\ldots,v_m\in \R^{m}$ where support of every $v_i$ is contained in a block.
Then we can compute $P v_1,\ldots, P v_m$ in $\wt O(n \tau^{\omega-1})$ time.
\end{lemma}
\begin{proof}
For $j\in [b]$, let $p_j$ be the number of vectors whose support is contained in $j$-th block. We combine these vectors into a matrix $V$ of size $\R^{B_j \times p_j}$.
To compute $P V$, it suffices to compute $P_{i,j} V$ for $i\in \cP(j)$.
We use fast matrix multiplication to perform this task.

Overall running time is
\begin{align*}
&~ \sum_{j\in [b]} \sum_{i\in \cP(j)} \wt O(\Tmat(|B_i|, |B_j|, p_j)) \\
=&~ \sum_{j\in [b]} \sum_{i\in \cP(j)} \wt O(\Tmat(\tau, \tau, \tau+p_j)) \\
=&~ \sum_{j\in [b]} \wt O(\Tmat(\tau, \tau, \tau+p_j)) \\
=&~ \sum_{j\in [b]} \wt O((\tau+p_j) \tau^{\omega-1}) \\
=&~ \sum_{j\in [b]} \wt O(m\tau^{\omega-1})
\end{align*}
where in the last step we use that $b = O(m/\tau)$ and $\sum_{j\in [b]} p_j = O(m)$.
\end{proof}

\subsection{Proof of Theorem~\ref{thm:main_lp_formal}} \label{sec:lp:main_proof}
\begin{proof}
According to Theorem~\ref{thm:algo_general}, there is an algorithm solving LP in
\begin{align*}
&\wt{O} ( n^{0.5} \nu_{\max}^{0.5} \cdot (T_H + T_n + T_L + T_Z + \nnz(A) + \eta m_\lp m_{\max})^{0.5} \\
&\cdot (\Tmat(n_{\max}) + T_{\Delta_L,\max} + T_{H,\max} + \eta^2 m_{\max}^2)^{0.5} \log(R/(r\epsilon)))
\end{align*}
time.

It remains to compute the parameters $T_H$, $T_L$, $\eta$, $T_m$, $m_{\max}$, $n_{\max}$, $T_{\Delta_L, \max}$, $T_{H,\max}$,
the details can be found in Lemma~\ref{lem:lp_param}.

Let $A $ and $B$ be defined as: 
\begin{align*}
    A: = & ~ T_H + T_n + T_L + T_Z + \nnz(A) + \eta m_{\lp} m_{\max}, \\
    B := & ~ \Tmat(n_{\max}) + T_{\Delta_L,\max} + T_{H,\max} + \eta^2 m_{\max}^2.
\end{align*}

We can show that
\begin{align}\label{eq:lp_A}
    A = & ~ T_H + T_n + T_L + T_Z + \nnz(A) + \eta m_{\lp} m_{\max} \notag \\
    = & ~ O(n ) + T_n + T_L + T_Z + \nnz(A) + \eta m_{\lp} m_{\max}  \notag \\
    = & ~ O(n ) + O(n \tau^{\omega-1}) + T_L + T_Z + \nnz(A) + \eta m_{\lp} m_{\max} \notag \\
    = & ~ O(n ) + O(n \tau^{\omega-1}) + \wt{O}(n \tau^{\omega-1} ) + T_Z + \nnz(A) + \eta m_{\lp} m_{\max} \notag \\
    = & ~ \wt{O}(n \tau^{\omega-1} ) + T_Z + \nnz(A) + \eta m_{\lp} m_{\max} \notag \\
    = & ~   \wt{O}(n \tau^{\omega-1} ) +  O(n \tau^{\omega-1} ) + \nnz(A) + \eta m_{\lp} m_{\max}  \notag \\
    = & ~  \wt{O}(n \tau^{\omega-1} ) +  O(n \tau^{\omega-1} ) + O(n\tau) + \eta m_{\lp} m_{\max}  \notag \\
    = & ~  \wt{O}(n \tau^{\omega-1} ) +  O(n \tau^{\omega-1} ) + O(n \tau ) + O(n \tau)  \notag \\
    = & ~ \wt{O}(n \tau^{\omega-1} )
\end{align}
where the second step follows from Lemma~\ref{lem:lp_param}\ref{item:lem:lp_TH} ($T_H = O(n)$) and the third step follows from Lemma~\ref{lem:lp_param}\ref{item:lem:lp_TnTm} ($T_n = O(n \tau^{\omega-1})$), the forth step follows from Lemma~\ref{lem:lp_param}\ref{item:lem:lp_TL} ($T_L = O(n \tau^{\omega-1})$), the fifth step follows from merging the terms, the six step follows Lemma~\ref{lem:lp_param}\ref{item:lem:lp_TZ}, the seventh step follows from Lemma~\ref{lem:lp_param}\ref{item:lem:lp_nnz}, the eighth step follows from Lemma~\ref{lem:lp_param}\ref{item:lem:lp_eta} ($\eta = \wt{O}(\tau), m_{\lp} = O(n), m_{\max} = 1$), and the last step follows from merging the terms.

For the term $B$, we have
\begin{align}\label{eq:lp_B}
    B = & ~ \Tmat(n_{\max}) + T_{\Delta_L,\max} + T_{H,\max} + \eta^2 m_{\max}^2 \notag \\
    = & ~ O( 1 ) + T_{\Delta_L,\max} + T_{H,\max} + \eta^2 m_{\max}^2 \notag \\
    = & ~ O( 1 ) + O( \tau^2 ) + T_{H,\max} + \eta^2 m_{\max}^2 \notag \\
    = & ~ O( 1 ) + O( \tau^2 ) + O( 1 ) + \eta^2 m_{\max}^2 \notag \\
    = & ~ O( 1 ) + O( \tau^2 ) + O( 1 ) + \wt{O}( \tau^2 ) \notag \\
    = & ~ \wt{O}(\tau^2) 
\end{align}
where the second step follows Lemma~\ref{lem:lp_param}\ref{item:lem:lp_n} ($n_{\max} = O(1)$), and the third step follows from Lemma~\ref{lem:lp_param}\ref{item:lem:lp_TDeltaLmax} ($T_{\Delta L, \max} = O(\tau^2)$), the forth step follows from Lemma~\ref{lem:lp_param}\ref{item:lem:lp_TH} ($T_{H,\max} = O(1)$), the fifth step follows from Lemma~\ref{lem:lp_param}\ref{item:lem:lp_eta} ($\eta = \wt{O}( \tau )$ and $m_{\max} = O(1)$), and the last step follows from merging the terms.

Finally, we have
\begin{align*}
    & ~ \wt{O}( n^{0.5} \nu_{\max}^{0.5} \cdot A^{0.5} \cdot B^{0.5} \cdot \log(1/\epsilon) ) \\
    = & ~ \wt{O} (n^{0.5}  \cdot A^{0.5} \cdot B^{0.5} \cdot \log(1/\epsilon)) \\
    = & ~ \wt{O} (n^{0.5} \cdot (n \tau^{\omega-1} )^{0.5} \cdot (\tau^2)^{0.5} \cdot \log(1/\epsilon)) \\
    = & ~ \wt{O}( n \tau^{(\omega+1)/2} \cdot \log(1/\epsilon) )
\end{align*}
where the first step follows from Lemma~\ref{lem:lp_param}\ref{item:lem:lp_eta} ($\nu_{\max} = O(1)$), the second step follows from $A = \wt{O}(n \tau^{\omega-1})$ (see Eq.~\eqref{eq:lp_A}) and $B=\wt{O}(\tau^2)$ (see Eq.~\eqref{eq:lp_B}). 

Thus, we complete the proof.
\end{proof}

%% file: robust_ipm.tex
\section{Robust IPM Analysis} \label{sec:robust_ipm}
The goal of this section is to present some existing tools which give us a bound on the number of iterations of IPM. To really improve the total running time, we still need to improve the cost per iteration which is a major contribution of our work. Those discussions can be found in Section~\ref{sec:framework}.

Let us begin with a roadmap for this section. 
In Section~\ref{sec:robust_ipm:convergence}, we present the main convergence statement. In Section~\ref{sec:robust_ipm:definitions}, we explain the the choice of step and present a useful lemma.

\subsection{Main Statement}\label{sec:robust_ipm:convergence}
We consider program of form~\eqref{eqn:lp_general}. The following theorem gives a bound on the number of iterations of a robust IPM algorithm.

\begin{theorem}[{\cite[Theorem A.1]{dly21}}] \label{thm:DLYA.1}
Consider program of form~\eqref{eqn:lp_general}.
Assume we are given $\nu_i$-self-concordant barriers $\phi_i : \cK_i \rightarrow \mathbb{R}$. Furtheremore, assume that the program satisfies the following properties.
\begin{itemize}
    \item Inner radius $r$: There exists a $z$ such that $A z = b$ and $B(z,r) \subset \cK$.
    \item Outer radius $R$: We have $\cK \subset B(x,R)$ for some $x \in \R^{n_\lp}$.
    \item Lipschitz constant $L$: $\| c \|_2 \leq L$.
\end{itemize}
Let $w \in \R_{\geq 1}^n$ be any weight vector, and $\kappa = \sum_{i=1}^n w_i \nu_i$. For any $\epsilon \in (0,1/2]$, there is an algorithm that runs in $O(\sqrt \kappa \log n \log(\frac{n_{\lp} \kappa R}{\epsilon r}))$ iterations and output $x\in \cK$ and $Ax=b$ such that:
\begin{align}
    c^\top x \le \min_{Ax=b,x\in K} c^\top x + \epsilon L R.
\end{align}
\end{theorem}

The above Theorem~\ref{thm:DLYA.1} provides the iteration bound for Algorithm~\ref{alg:robust_ipm_main}. 
The overall running time largely depends on cost per iteration, which is decided by the implementation of \textsc{CentralPathMaintenance}.
The high level framework of Algorithm~\ref{alg:robust_ipm_main} 
 and~\ref{alg:robust_ipm_centering_data_structure} are the same as \cite{dly21}, but the reduction to form~\eqref{eqn:lp_general} and the implementations of \textsc{CentralPathMaintenance} are different.

\begin{algorithm}[!ht]\caption{Robust IPM algorithm. This decides the number of iterations but not the cost per iteration.
} \label{alg:robust_ipm_main}
\begin{algorithmic}[1]
\Procedure{RobustIPM}{$A \in \R^{m_\lp \times n_\lp}, b\in \R^{m_\lp}, c\in \R^{n_\lp}, (\phi_i : \R^{n_i}\to \R)_{i\in [n]}, w\in \R^n$}
    \State \textbf{Input:} Program of form~\eqref{eqn:lp_general} satisfying the assumptions in Theorem~\ref{thm:DLYA.1}.
    \State \textbf{Output:} A solution $x$ satisfying the statement of Theorem~\ref{thm:DLYA.1}.
    \State Let $\phi(x) = \sum_{i=1}^n w_i \phi_i(x_i)$, $L = \| c \|_2$, $\kappa = \sum_{i=1}^n w_i \nu_i$
    \State Let $t = (n_\lp+\kappa)^5 \cdot \frac{LR}{\delta} \cdot \frac{R}{r}$ with $\delta =1/128$
    \State Compute $x_c = \arg\min_{x \in \cK} c^\top x+ t \cdot \phi(x)$ and $x_{\circ} = \arg\min_{Ax = b} \| x - x_c \|_2$
    \State Let $x = (x_c, 3R+x_{\circ} - x_c, 3R)$ and $s = ( -\nabla \phi(x_c), \frac{t}{3R+x_{\circ} - x_c} , \frac{t}{3R})$
    \State Let the new matrix $A^{\new} = [A, A, -A]$, the new barrier and new weight
    \begin{align*}
        \phi_i^{\new} 
        = \begin{cases}
        \phi_i, & \text{~if~} i \in [n_\lp],\\
        -\log x, & \text{else}
        \end{cases}
        \text{~~~and~~~}
            w_i^{\new} = 
            \begin{cases}
                    w_i, & \text{~if~} i \in [n_\lp], \\
                    1, & \text{else}
            \end{cases}
    \end{align*}
    \State $( x^{(1)}, x^{(2)}, x^{(3)} ), (s^{(1)}, s^{(2)}, s^{(3)}) \gets \textsc{Centering}(A^{\new}, \phi^{\new}, w^{\new}, x, s, t, LR)$
    \State $(x,s) \gets (x^{(1)} + x^{(2)} - x^{(3)}, s^{(1)})$
    \State $(x,s) \gets \textsc{Centering}(A,\phi,w,x,s,LR, \frac{\epsilon}{4 \sum_i w_i \nu_i})$
    \State \Return $x$
\EndProcedure 
\end{algorithmic}
\end{algorithm}

\begin{algorithm}[!ht]\caption{Centering Algorithm. Continuation of Algorithm~\ref{alg:robust_ipm_main}.
} \label{alg:robust_ipm_centering_data_structure}
\begin{algorithmic}[1]
\Procedure{Centering}{$A, \phi, w, x, s, t_{\mathrm{start}}, t_{\mathrm{end}}$}
\State {\bf private : member}
\State \hspace{4mm} \textsc{CentralPathMaintenance} cp \Comment{Algorithm~\ref{alg:cpm}}
\State {\bf end members}
    \State {\color{blue}/* Choose parameters */}
    \State $\lambda \eqsim \log( n \sum_{i=1}^{n} w_i ), \ov{\epsilon} \eqsim 1/\lambda$, $\alpha \eqsim \ov{\epsilon}$
    \State $\epsilon_t \eqsim \ov{\epsilon} \cdot \min_{i \in [n] } \frac{w_i}{w_i + \nu_i}$
    \State {\color{blue}/* Definition of several functions */}
     \For{$i = 1 \to n$}
        \State $\mu_i^t(x,s) := s_i/t+ w_i \nabla \phi_i(x_i) $ 
        \State $\gamma_i^t(x,s) : = \| \mu_i^t(x,s) \|_{x_i}^* $
        \State  $c_i^t(x,s):= \frac{ \sinh( \lambda w_i^{-1} \gamma_i^t(x,s) ) }{ \gamma_i^t(x,s) \cdot ( \sum_{j=1}^{n} w_j^{-1} \cosh^2( \lambda w_i^{-1} \gamma_j^t(x,s) ) )^{1/2} } $
        \State $\delta_{\mu,i}(x,s,t) \gets -\alpha \cdot c_i^{t}(x, s) \cdot \mu_i^{t}(x, s)$ for all $i\in [n]$
     \EndFor
    \State $\Psi_{\lambda}(r) := \sum_{i=1}^n \cosh( \lambda r_i / w_i )$
    \State $\Phi^t(x,s) := \Psi_{\lambda}(\gamma^t(x,s))$
    \State $H_x := \nabla^2 \phi(x)$
    \State $P_x := H_x^{-1/2} A^\top ( A H_x^{-1} A^\top )^{-1} A H_x^{-1/2}$ 
    \State {\color{blue}/* Main Loop */}
    \State $t \gets t_{\mathrm{start}}$
    \State cp.\textsc{Initialize}($x, s, t, \ov \epsilon$) \Comment{Algorithm~\ref{alg:cpm}}
    \While{$t \geq t_{\mathrm{end}}$}
        \State cp.\textsc{MultiplyAndMove}($t$) \Comment{Algorithm~\ref{alg:cpm}}
        \State \Comment{In the above line, we maintain $(\ov{x}, \ov{s}, \ov{t})$ such that $\| \ov{x}_i - x_i \|_{ \ov{x}_i } \leq \ov{\epsilon}$, $\| \ov{s}_i - s_i \|_{ \ov{x}_i }^* \leq \ov{t} \ov{\epsilon} w_i$, $\forall i \in [n]$ and $|\ov{t} - t| \leq \epsilon_t \cdot \ov{t}$}, \label{line:cpm_guarantee}
        \State \Comment{and perform central path step $x \gets x+ \delta_x$, $s \gets s+ \delta_s$, where $\delta_x$ and $\delta_s$ satisfy $A \delta_x = 0$, $\delta_s \in \mathrm{Range}(A^\top)$ and
        \begin{align*}
            \| H_{\ov{x}}^{1/2} \delta_x - (I - P_{\ov{x}} ) H_{\ov{x}}^{-1/2} \delta_{\mu}(\ov x, \ov s, \ov t) \|_2 \leq & ~ \ov{\epsilon} \cdot \alpha \\
            \| \ov{t}^{-1} H_{\ov{x}}^{-1/2} \delta_s - P_{\ov{x} } H_{\ov{x}}^{-1/2} \delta_{\mu}(\ov x, \ov s, \ov t) \|_2 \leq & ~ \ov{\epsilon} \cdot \alpha.
        \end{align*}}
        \State $t \gets \max\{ (1-h) t, t_{\mathrm{end}} \}$
    \EndWhile 
    \State \Return cp.\textsc{Output}() \Comment{Algorithm~\ref{alg:cpm}}
\EndProcedure 
\end{algorithmic}
\end{algorithm}

\subsection{Definitions and Useful Lemmas}\label{sec:robust_ipm:definitions}
We define the following induced norms.
\begin{definition}[{\cite[Definition A.5]{dly21}}]
For each block $\cK_i$, we define 
\begin{align*}
\| v \|_{x_i} := & ~ \| v \|_{ \nabla^2 \phi_i (x_i) } \\
\| v \|_{x_i}^* := & ~ \| v \|_{ ( \nabla^2 \phi_i(x_i) )^{-1} }
\end{align*}
for $v \in \R^{n_i}$.

For the whole domain $\cK = \prod_{i=1}^n \cK_i$, we define
\begin{align*}
    \| v \|_x := \| v \|_{\nabla^2 \phi(x) } = ( \sum_{i=1}^n w_i \| v_i \|_{x_i}^2 )^{1/2}
\end{align*}
and 
\begin{align*}
    \| v \|_x^{*} := \| v \|_{ ( \nabla^2 \phi(x) )^{-1} } = ( \sum_{i=1}^n w_i^{-1} ( \| v_i \|_{x_i}^* )^2 )^{1/2}
\end{align*}
for $v \in \R^{n_\lp}$.
\end{definition}

The central path is
\begin{align*}
    x(t) := \arg\min_{Ax = b} c^\top x + t \cdot \phi(x) \mathrm{~~~with~~~} \phi(x) := \sum_{i=1}^n w_i \phi_i(x_i).
\end{align*}

Instead of following the path $x(t)$ exactly, we follow the path
\begin{align*}
    \frac{s}{t} + w \cdot \nabla \phi(x) = & ~ \mu, \\
    A x = & ~ b, \\
    A^\top y + s = & ~ c 
\end{align*}
where $\mu$ is close to $0$ under $\|\cdot\|_x^*$. The norm of $\mu$ is controlled using the following potential function.
\begin{definition}[{Potential function, \cite[Definition A.7]{dly21}}]
For $i\in [n]$, define error at $i$-th variable block as
\begin{align*}
    \mu_i^t(x,s) : = \frac{s_i}{t} + w_i \cdot \nabla \phi_i(x_i).
\end{align*}
Define $\gamma_i^t(x,s) : =\| \mu_i^t(x,s) \|_{x_i}^*$.
Define the soft-max function as
\begin{align*}
    \Psi_{\lambda}( r ) := \sum_{i=1}^{n} \cosh( \lambda \frac{r_i}{w_i} ) 
\end{align*}
for some $\lambda > 0$.
Finally, the potential function is the soft-max of norm of the error at each variable block
\begin{align*}
    \Phi^t(x,s) = \Psi_{\lambda} ( \gamma^t (x,s) ).
\end{align*}
\end{definition}

Since our goal is to decrease $\Phi(x,s) = \Psi_{\lambda}(\gamma)$, a natural choice is the steepest descent direction (\cite[Section A.4]{dly21}):
\begin{align}
    \delta_\mu^* = \arg \min_{ \| \delta_{\mu} \|_x^* = \alpha } \langle \nabla_{\mu} \Psi_{\lambda} ( \| \mu_i \|_{x_i}^* ) , \mu + \delta_\mu \rangle
\end{align}
with step size $\alpha$. 

Solving this gives
\begin{align*}
    \delta_{\mu,i}^*(x,s,t) = - \frac{ \alpha \sinh( \frac{\lambda}{w_i} \gamma_i^t(x,s) ) }{ \gamma_i^t(x,s) \cdot ( \sum_{j\in [n]} w_j^{-1} \sinh^2( \frac{\lambda}{w_j} \gamma_j^t(x,s) ) )^{1/2} } \cdot \mu_i^t(x,s).
\end{align*}
(In the actual algorithm, $\sinh$ in the denominator is replaced by $\cosh$ for easier analysis.)
To move $\mu \in \R^{n_\lp}$ to $\mu+\delta_{\mu} \in \R^{n_\lp}$ approximately, we take Newton step $\delta_x^*, \delta_s^* \in \R^{n_\lp}$:
\begin{align*}
\frac{1}{t} \delta_s^* + \nabla^2 \phi(x) \delta_x^* = & ~  \delta_{\mu}^*(x,s) \\
A \delta_x^* = & ~ 0, \\
A^\top \delta_y^* + \delta_s^* = & ~ 0. 
\end{align*}

Using $H_x$ to denote $\nabla^2 \phi(x)$ and solve the above equations, we get 
\begin{align*}
    \delta_x^* &= H_x^{-1} \delta_{\mu}^* - H_x^{-1} A^\top (A H_{\ov x}^{-1} A^\top)^{-1} A H_x^{-1} \delta_\mu^*(x, s, t),\\
    \delta_s^* &= t A^\top (A H_x^{-1} A^\top)^{-1} A H_x^{-1} \delta_\mu^*(x, s, t).
\end{align*}

This is the ideal IPM step. In robust IPM, we compute the steps using $(\ov x, \ov s, \ov t)$, a sparsely changing approximation of $(x, s, t)$, giving
\begin{align*}
    \delta_x &= H_{\ov x}^{-1} \delta_{\mu}^* - H_{\ov x}^{-1} A^\top (A H_{\ov x}^{-1} A^\top)^{-1} A H_{\ov x}^{-1} \delta_\mu(\ov x, \ov s, \ov t),\\
    \delta_s &= \ov t A^\top (A H_{\ov x}^{-1} A^\top)^{-1} A H_{\ov x}^{-1} \delta_\mu(\ov x, \ov s, \ov t).
\end{align*}

We state a useful lemma for bounding the step size.
\begin{lemma}[{\cite[Lemma A.9]{dly21}}] \label{lemma:DLYA.9}
The steps $\delta_x$ and $\delta_s$ satisfy
\begin{align*}
&( \sum_{i\in [n]} w_i \|\delta_{x,i}\|_{\ov x_i}^2 )^{1/2} \le \frac{9}{8} \alpha,\\
&( \sum_{i\in [n]} w_i^{-1} (\|\delta_{s,i}\|_{\ov x_i}^*)^2 )^{1/2} \le \frac{9}{8} \alpha \cdot t.
\end{align*}
\end{lemma}